\documentclass[dvipsnames,11pt,reqno,twoside,final]{amsart}
\usepackage[a4paper,left=3cm,right=3cm,top=3cm,bottom=3cm]{geometry}
\usepackage[T1]{fontenc}
\usepackage[utf8]{inputenc}
\usepackage{xcolor}
\usepackage{amssymb,mathtools}  
\usepackage{dsfont}
\usepackage{enumitem}
\usepackage{subcaption}
\usepackage{centernot}
\usepackage{booktabs}
\usepackage{enumerate}
\usepackage{mathrsfs}
\usepackage{mathtools}
\mathtoolsset{showonlyrefs}

\usepackage{amsfonts}
\usepackage[utf8]{inputenc}
\usepackage{amsmath,amssymb,amsthm}
\usepackage{extarrows} 

\usepackage{cancel}
\newcommand\concel[2]{\ooalign{$\hfil#1\mkern0mu/\hfil$\crcr$#1#2$}}
\newcommand\nxlra[1]{\mathrel{\mathpalette\concel{\xlongleftrightarrow{#1}}}}

\usepackage{geometry}
\usepackage{soul}

\usepackage{comment}
\usepackage[most]{tcolorbox}

\usepackage{mathalfa}
\usepackage{colortbl} 
\usepackage{amsmath}
\usepackage{amsfonts}
\usepackage{longtable}
\usepackage{dsfont}
\usepackage{booktabs}
\usepackage{centernot}
\usepackage{amsthm}
\usepackage{amssymb}
\usepackage{mathrsfs}
\usepackage{bbm}
\usepackage{marginnote}
\usepackage{scalerel}
\usepackage{tikz}
\usepackage{tikz-cd}
\usepackage{graphicx}
\usepackage{caption}
\usepackage{subcaption}
\usepackage{pdfpages}
\usepackage{mhequ}
\usepackage{hyperref}

\usepackage{preamble/mhequ}

\usepackage{amsthm}

\numberwithin{counterEnvDefault}{section}

\theoremstyle{plain}
\newtheorem{Th}{Theorem}[section]
\newtheorem{Lemma}[Th]{Lemma}
\newtheorem{Cor}[Th]{Corollary}
\newtheorem{Prop}[Th]{Proposition}

\theoremstyle{definition}
\newtheorem{Def}[Th]{Definition}

\newtheorem{Rem}[Th]{Remark}

\newtheorem*{Lemma*}{Lemma}

\usepackage{microtype}
\renewcommand\epsilon\varepsilon

\usepackage{hyperref}
\definecolor{colorlinks}{RGB}{0, 24, 168}
\definecolor{colorcites}{RGB}{0, 138, 118}
\hypersetup{
    colorlinks=true,
    linkcolor=colorlinks,
    citecolor=colorcites,
    urlcolor=colorlinks,
    pdfborder={0 0 0}
}

\newcommand{\bfP}{\mathbf{P}}

\newcommand{\bfj}{\mathbf{j}}

\newcommand{\calY}{\mathcal{Y}}
\newcommand{\calC}{\mathcal{C}}
\newcommand{\calH}{\mathcal{H}}
\newcommand{\calD}{\mathcal D}
\newcommand{\calF}{\mathcal{F}}

\newcommand{\calE}{\mathcal{E}}

\newcommand{\lra}{\longleftrightarrow}

\newcommand{\N}{\ensuremath{\mathbb{N}}}
\newcommand{\R}{\ensuremath{\mathbb{R}}}

\newcommand{\Z}{\ensuremath{\mathbb{Z}}}

\newcommand{\PP}{\ensuremath{\mathbb{P}}}

\newcommand{\e}{\ensuremath{\mathrm{e}}}
\newcommand{\eps}{\varepsilon}

\newcommand{\1}{\mathds{1}}

\newcommand{\om}{\omega}

\newcommand{\La}{\Lambda}

\newcommand{\calA}{\mathcal{A}}
\newcommand{\calB}{\mathcal{B}}

\newcommand{\bbS}{\mathbb{S}}
\newcommand{\bbR}{\mathbb{R}}

\newcommand{\calL}{\mathcal L}

\newcommand{\sfC}{\mathsf C}
\newcommand{\sfD}{\mathsf D}
\newcommand{\bbP}{\mathbb P}
\newcommand{\bfK}{\mathsf K}

\usepackage{microtype}
\def\les{\!{\begin{array}{c}<\\[-10pt]\scriptstyle{\frown}\end{array}}\!}
\def\ges{\!{\begin{array}{c}>\\[-10pt]\scriptstyle{\frown}\end{array}}\!}

\newcommand\xlra{\xleftrightarrow}
\renewcommand\lra{\leftrightarrow}
\usepackage{mathtools}

\mathtoolsset{showonlyrefs}

\mathtoolsset{showonlyrefs}
\usepackage{enumerate}

\makeatletter
\@namedef{subjclassname@2020}{\textup{2020} Mathematics Subject Classification}
\makeatother
\microtypesetup{expansion=false}

\title{Near-critical Ornstein--Zernike theory for the planar random-cluster model}

\author{Lucas D'Alimonte}
\address{LPSM, Sorbonne Université, Paris, France}
\email{dalimonte@lpsm.paris}

\author{Ioan Manolescu}
\address{Université de Fribourg, Switzerland}
\email{ioan.manolescu@unifr.ch}

\date{\today}

\begin{document}

\maketitle

\begin{abstract}
We develop an Ornstein--Zernike theory for the two-dimensional random-cluster model with $1 \leq q <4$ that also applies in its near-critical regime. 
In particular, we prove an asymptotic formula for the two-point function which holds uniformly for~$p < p_c$ and blends the subcritical and near-critical behaviours of the model.

The analysis is carried out by studying the renewal properties of a subcritical percolation cluster, \emph{at the scale of the correlation length}. 
More precisely, we explore sequentially the cluster in a given direction, by slices of thickness comparable to the correlation length. 
We show that this exploration satisfies the properties of a {\em killed Markov renewal process} --- a class of processes that may be analysed independently and have Brownian behaviour. In addition to the  two-point function estimate, we derive other consequences of the Ornstein--Zernike theory such as an invariance principle for the rescaled cluster and the strict convexity of the inverse correlation length --- all at the scale of the correlation length, uniformly in~$p<p_c$. 

Finally, our approach differs from that of earlier papers of Campanino, Ioffe, Velenik and others, with the cluster being dynamically explored rather than constructed from its diamond decomposition. 
\end{abstract}

\setcounter{tocdepth}{2}

\makeatletter
\def\l@section{\@tocline{1}{0pt}{1em}{}{\bfseries}}
\def\l@subsection{\@tocline{2}{0pt}{4em}{}{}}
\makeatother

\tableofcontents

\section{Introduction}

Rigorous understanding of the correlation structure of random fields is one of the main objectives of modern statistical mechanics. 
This article focuses on one of the most studied such models: the so-called random-cluster measure (also known as FK-percolation).
The classical Ornstein--Zernike (OZ) formula describes the behaviour of the two-point function~$G(x) := \phi_{p,q}[0 \lra x]$ when~$x$ goes to infinity, where~$\phi_{p,q}$ is the random-cluster measure of parameters~$p$ and~$q$, and~$p$ is such that the model is in its subcritical (or disordered) phase. It quantifies the behaviour of~$G$ beyond its exponential decay and is sharp when~$x$ tends to infinity. 

The OZ asymptotic was first conjectured in two seminal works by Ornstein and Zernike in 1914~\cite{OrnsteinandZernike}, and Zernike~\cite{Zernike}. Trying to correct a formula describing the phenomenon of \emph{opalescence} in a crystal, they provided a non-rigorous computation of what would later become the OZ formula. 
The first rigorous implementation of this reasoning is due to Abraham and Kunz~\cite{abrahamkunzOZ} and Paes--Leme~\cite{paeslemeOZ} in the context of classical lattice gases theory, by means of a graphical representation of the partition function. 
Soon after, the Ornstein--Zernike asymptotic was proved for a variety of models \emph{in a perturbative regime}, see~\cite{bricmontfrohlichozperturbative, perturbativeoz}.

The next breakthrough consisted in a rigorous derivation of the Ornstein--Zernike asymptotic in the whole regime of exponential decay of the correlation functions; this was done in the case of the self-avoiding walk in the direction of an axis in~\cite{chayeschayesozsawaxis} and later on in any direction in~\cite{Ioffe98ornstein-zernikebehaviour}. For percolation models, the case of Bernoulli percolation was first treated for an axis direction in~\cite{CampaninoChayesChayesBernoulliperco} and later on in any direction in~\cite{ornsteinzernikebernoulli2002}. The case of subcritical Ising models was treated in~\cite{ozisingcampanino} via the random-line graphical representation of the two-point function. Finally, the analysis was carried out for all the subcritical random-cluster measures in~\cite{campaninonioffevelenikozrandomcluster}. 

In recent developments, the theory has been extended to Ising models with long-range interactions~\cite{longrangeozvelenikaounott}, but was also shown to fail in some long-range Ising models, when the coupling constants decay too slowly~\cite{nonanalyticitycorrelationlength, 2ptfunctionsaturation}.   

All works cited above aim to understand  correlations at a fixed, subcritical temperature. In this article, we are interested in how the OZ formula behaves when the temperature parameter ($p$ in the random-cluster model) approaches its critical value. Our study is limited to the random-cluster model on the two-dimensional square lattice~$\mathbb Z^2$, with~$q \in [1,4)$ and~$p<p_c(q) = \frac{\sqrt q}{1 + \sqrt q}$. The near-critical properties of the model were recently studied in \cite{DuminilCopinManolescuScalingRelations}.

Our main output is an asymptotic formula for the function~$x \mapsto G(x)$ \emph{uniformly both in~$p$ and~$x$}, in the regime~$p<p_c$ and~$\|x\|$  above the correlation length. 
When~$p=p_c$ and~$q \in [1,4]$  the phase transition is known to be continuous \cite{DCSidoraviciusTassionContinuity} and the function~$G$ decays algebraically in~$\Vert x \Vert$. 
The main feature of our analysis is that it captures how the two-point function switches from a disordered behaviour (characterised by its exponential decay) 
to a near-critical behaviour (with the same algebraic decay as at criticality). 
To do this, we needed to replace the arguments of \cite{campaninonioffevelenikozrandomcluster} relying on the finite energy property by RSW-type constructions at the scale of the correlation length. 

Besides this new result, another interest of the present paper is to revisit the approach developed in~\cite{campaninonioffevelenikozrandomcluster} and related papers. 
Indeed, while our argument uses key elements of~\cite{campaninonioffevelenikozrandomcluster}, several differences may be highlighted. 
We circumvent the use of the ``skeleton calculus" and employ a direct exploration of the subcritical cluster in a given direction. 
The argument also bypasses the perturbation theory used in~\cite{campaninonioffevelenikozrandomcluster} to deduce the strict convexity properties of the correlation length and of the associated Wulff shape. 
Instead, we deduce these properties from rather standard large deviations estimates for the associated random exploration.  

Our argument is limited to the two-dimensional model as this is the only case where the near-critical behaviour of the random-cluster model has been established. 
We believe that the same strategy may be implemented in general dimension for fixed~$p < p_c$, where all RSW arguments may be replaced by finite-energy constructions, to reprove the results of \cite{campaninonioffevelenikozrandomcluster}.

We conclude this introduction by mentioning that the crossover between the Ornstein--Zernike regime and polynomial decay characterizing criticality has been investigated in two different contexts. 
 In~\cite{MichtaSlade}, the authors obtain precise results on the asymptotic decay of the Green function of the massive random walk on $\Z^d$, using precise computations involving Bessel-type representations of the Green function. 
Closer to our result, the recent~\cite{LiuSlade} establishes a uniform formula for the two-point function of various \emph{high-dimensional} statistical mechanics models, including Bernoulli percolation in dimension $d \geq 15$. 
The result is more precise in this setting, as the authors identify the polynomial exponent of the critical term, arising from the mean-field behaviour of high-dimensional percolation.

\subsection{Definition of the model}\label{sec:def_model}

In this section we briefly define the model, with some additional features deferred to Section~\ref{sec:background}. We refer to the monographs~\cite{grimmetttherandomclustermodel, duminilcopin2017lectures} for further background.

We slightly abuse notation by writing~$\Z^2 = (\Z^2, E(\Z^2))$ for the two-dimensional square lattice, with vertex set~$\mathbb Z^2$ and edges between vertices at Euclidean distance~$1$.  
For~$G = (V(G),E(G))$ a subgraph of~$\Z^2$, the space of percolation configurations on~$G$ is~$\Omega^G := \{0,1\}^{E(G)}$. 
For an edge~$e \in E(G)$ and~$\om \in \Omega^G$, we say that~$e$ is \emph{open}  if~$\om(e) = 1$ and \emph{closed} otherwise. 
A percolation configuration will be identified both with the set of its open edges as well as with the sub-graph of~$G$ with vertex set~$V(G)$ and edge-set formed of the open edges of~$\omega$. 
The connected components of~$\omega$ are called \emph{clusters}.  

The boundary~$\partial G$ of~$G$ is the set of vertices of~$G$ incident to at least one edge of~$E(\Z^2)\setminus E(G)$.
A \emph{boundary condition}~$\eta$ is a partition of~$\partial G$; we say that the vertices of~$\partial G$ that belong to the same component of~$\eta$ are {\em wired} together. 
To a boundary condition~$\eta$ and a percolation configuration~$\om \in \{0,1\}^{E(G)}$, associate the graph~$\om^\eta$ which is obtained from $\omega$ by identifying all the mutually wired vertices of~$\partial G$. 

A percolation configuration~$\xi$ on~$\Z^2$ induces certain boundary conditions on~$\partial G$: two vertices are wired together if they are connected in~$\xi\setminus E(G)$. We shall make a slight notational abuse by identifying the percolation configuration~$\xi$ with the boundary condition it induces on~$\partial G$, and keeping the notation~$\om^\xi$ when~$\xi$ is a percolation configuration on~$\Z^2 \setminus G$. 
Two boundary conditions play a special role: the {\em free} boundary condition, denoted by~$0$, in which no boundary vertices are wired together,
and the {\em wired} boundary condition, denoted by~$1$, in which all boundary vertices are wired together.  

The random-cluster measure on a finite subgraph~$G$ of~$\Z^2$, with boundary conditions~$\eta$ and parameters~$p \in (0,1)$ and~$q \geq 1$ is defined as follows. For a percolation configuration~$\om$ on~$G$ write~$o(\om)$ for number of open edges of~$\om$, and~$k(\om^\eta)$ the number of clusters of~$\om^\eta$. 
Set
 \begin{equation*}
	\phi^\eta_{G,p,q}[\om] = \frac{1}{Z^\eta_{G,p,q}} \left(\frac{p}{1-p}\right)^{o(\om)}q^{k(\om^\eta)},
 \end{equation*}
 where~$Z^{\eta}_{G,p,q}$ is called the \emph{partition function} of the model, and is the unique constant guaranteeing that~$\phi^\eta_{G,p,q}$ is a probability measure. 
 
Random-cluster measures may also be defined on the full graph~$\Z^2$ either through the DLR formalism or by taking limits of measures on increasing finite subgraphs of~$\Z^2$. Due to monotonicity properties, it is classical that the free and wired measures~$\phi^0_{G,p,q}$ and~$\phi^1_{G,p,q}$ admit limits as~$G$ increases to~$\Z^2$. 
These will be denoted by~$\phi^0_{p,q}$ and~$\phi^1_{p,q}$, respectively, and are instances of {\em infinite-volume measures}, which are invariant under translations and ergodic.
 
It was proved in~\cite{duminilbeffara} that the model exhibits a phase transition at the self-dual parameter~$p_c(q)= \tfrac{\sqrt q}{1 + \sqrt q}$. 
That is, for any~$p<p_c(q)$ (in the so-called \emph{subcritical regime}), there exists a unique infinite-volume measure ($\phi^0_{p,q}=\phi^1_{p,q}$) and all clusters are almost surely finite under this measure. When~$p > p_c(q)$, the infinite-volume measure is again unique and contains almost surely a unique infinite cluster. 
The phase transition was shown to be continuous for~$q \in [1,4]$~\cite{DCSidoraviciusTassionContinuity}, in the sense that there exists a unique infinite-volume measure also at~$p=p_c(q)$. 
Additionally, strong RSW-type estimates were established at criticality for these values of~$q$~\cite{DCSidoraviciusTassionContinuity,DumTas19,DumManTas21}. 
For~$q >4$, the phase transition is discontinuous~\cite{DCgagnebinHarelManolescuDiscontinuity,RaySpi20}, 
which is to say that~$\phi^0_{p_c(q),q}\neq\phi^1_{p_c(q),q}$, with the former having a subcritical behaviour and the latter a super-critical one.

As already mentioned, we are treating here the case of a continuous phase transition, for which the correlation length diverges as $p \nearrow p_c$. 
Furthermore, for reasons explained in Remark~\ref{rem:q=4}, the case $q = 4$ is excluded from our analysis. 

\begin{center}
\bf Henceforth, when~$q \in [1,4)$ is fixed, we will omit it from the notation. 
As the infinite-volume measure is unique for any $p \in [0,1]$, we denote it by~$\phi_p$.
\end{center}

For the exposition of our results, we need to introduce two classical quantities in the study of subcritical and near-critical random-cluster model. Fix $q\in [1,4)$ and~$p<p_c$.
The \emph{correlation length} of the model is defined as the following limit, the existence of which is based on super-multiplicativity arguments. For~$\vec v \in \R^2 \setminus \{0\}$, set\footnote{$\lfloor n\vec v \rfloor$ denotes the vertex of~$\Z^2$ which is the closest to~$n\vec v$. In case of equality, we arbitrarily choose the top-left most one.}
\begin{equation*}
\xi_p(\vec v) = \big(\lim_{n \rightarrow \infty} -\tfrac 1 n \log \phi_p[0 \lra \lfloor n\vec v \rfloor]\big)^{-1}.
\end{equation*}

Write $\|.\|$ for the Euclidian norm on $\bbR^2$ and let $\bbS^1 = \{ \vec v \in \bbR^2: \|\vec v\|  =1\}$.
Notice that, for any $\vec v \neq 0$, 
\begin{align}\label{eq:xi_norm}
	\| \vec v\| \xi_p(\vec v) = \xi_p(\vec v/\|\vec v\|),
\end{align} 
so we will mostly consider the case~$\vec v \in \bbS^1$. 
The results of~\cite{duminilbeffara} imply that whenever~$p<p_c$,~$\xi_p(\vec v) > 0$ for any~$\vec v \in \bbS^1$.

For~$R \geq 0$, let~$\Lambda_R := \{-R, \dots, R\}^2$ and~$\La_R(x)$ for its translate by~$x$. 
Define the \emph{critical one-arm probability} at distance~$R$ as
\begin{equation*}
\pi_1(R) = \phi_{p_c}^0[0 \lra \partial \Lambda_R]. 
\end{equation*}

\subsection{Results}

We are now ready to state our main results. 
For positive quantities~$f,g$ we write~$f \les g$ to mean that there exists a constant~$C >0$ such that~$f \leq Cg$. The constants~$C$ may be chosen uniformly in certain parameters in the definitions of~$f$ and~$g$ which will be explicitly stated -- these will generally be~$p < p_c$,~$n$ and~$t$ below. 
Write~$f \asymp g$ when~$f \les g$ and~$g \les f$. 
We will often use the expression {\em uniformly in~$p < p_c$} to mean uniformly for~$p \in [\eps, p_c)$ for some~$\eps > 0$ fixed throughout the paper. The choice of~$\eps$ is arbitrary, but may affect the value of certain constants. 

\begin{Th}[Ornstein--Zernike asymptotics at the scale of the correlation length]\label{thm:main}
	Fix~$q \in [1,4)$. Then, uniformly in~$\vec v \in \mathbb S^1$,~$p  < p_c$ and~$n \geq \xi_p(\vec v)$,
	\begin{equation}\label{eq:main}
    	\phi_p\big[0 \lra \lfloor n \vec v \rfloor\big] \asymp 
		\pi_1( \xi_p(\vec v) ) ^2 \sqrt{\tfrac{\xi_p(\vec v)}{n}}\, \e^{-\frac {n}{\xi_p(\vec v)}}.
	\end{equation}
\end{Th}

Compared to the classical OZ results of~\cite{campaninonioffevelenikozrandomcluster}, the advantage of our result is that it is uniform in~$p$. 
Indeed, for~$q \in [1,4)$,~$\xi_p(\vec v) \to \infty$ as~$p \nearrow p_c$. 
Previous results often employ local surgeries based on finite energy at scales lower than~$\xi_p(\vec v)$. As such they do not apply when~$\xi_p(\vec v)$ diverges.  
Our approach blends the critical and subcritical behaviour, each accounting for one of the terms in the right-hand side of~\eqref{eq:main}: 
the term~$\e^{-\frac {n}{\xi_p(\vec v)}} (\frac{n}{\xi_p(\vec v)})^{-1/2}$ is reminiscent of classical OZ-type formulas and is entirely a subcritical phenomenon, while the term~$\pi_1( \xi_p(\vec v)) ^2$ appears due to the near-critical behaviour of the model. 
Finally, when~$n \leq \xi_p(\vec v)$, the right-hand side may be replaced by~$\pi_1(n) ^2$, as proved in~\cite{DuminilCopinManolescuScalingRelations}.

Theorem \ref{thm:main} is a consequence of a more detailed description of the cluster of~$0$ when conditioned to be connected to a half-plane at a distance~$n$ in the direction~$\vec w$, dual to~$\vec v$; see Theorem~\ref{thm:killed_renewal_structure} below. 
We would like to emphasise that Theorem~\ref{thm:killed_renewal_structure} also has other consequences, such as: 
\begin{itemize}
\item Corollary~\ref{cor:pure_exp} which states that the probability for $0$ to be connected to the half-space~$\{\langle x,\vec w\rangle \geq n\}$ is uniformly comparable to $\pi_1( \xi_p(\vec w))$ times a pure exponential of $n/\xi_p(\vec w)$. 
\item Theorem~\ref{thm: strict convexity} which states that the inverse correlation length~$\xi_p$ is strictly convex in terms of the direction, as is its convex dual, known as the \emph{Wulff shape}. 
\item Theorem~\ref{thm: invariance principle} which establishes an invariance principle for the cluster of~$0$ when conditioned to be connected to~$\lfloor n \vec v\rfloor$ or to a half-plane.
\end{itemize}
Contrary to previous approaches, Theorem~\ref{thm:killed_renewal_structure} provides all these consequences simultaneously. 

\begin{Rem}\label{rem:ising}
In the special case $q=2$ of the FK-Ising model, it is interesting to confront~\eqref{eq:main} with the exact solution of the Ising model. 
It was pointed out to us by Gordon Slade, whom we thank for this observation, that a uniform version of~\cite[Theorem 2.7.2]{Palmer_book} indeed reveals the crossover between the Ornstein--Zernike term and the critical term $\pi^1(\xi_p(\vec v))$, which in the cited reference corresponds to the square of the critical magnetisation.
Both versions of the near-critical OZ formula are thus compatible.
\end{Rem}

\begin{Rem}\label{rem:q=4}
We conclude this introduction by discussing the requirement that $q \in [1,4)$. 
When $q > 4$, it has been shown in~\cite{DCgagnebinHarelManolescuDiscontinuity} that the correlation length is uniformly bounded in the whole subcritical regime and that the phase transition is of first order. 
Thus, a OZ result valid uniformly for $p<p_c$ would also hold for the free measure at $p_c$. 
At the critical point, interfaces between the free and wired phases are expected to have random-walk behaviour, 
while primal clusters under the free measure, when conditioned to hit a far-away half-space, resemble the space between two random walks conditioned not to intersect --- the upcoming paper~\cite{dober2025discontinuoustransition2dpotts} aims to prove this. 
This behaviour is inconsistent with the OZ behaviour of Theorem~\ref{thm:main}: indeed by analogy with~\cite[Corollary 1.8]{DAl2024}, it is expected that in that setting, the two-point function in the free measure at the critical parameter behaves as
$
\phi^0_{p_c}[0 \lra \lfloor nx \rfloor] \asymp n^{-2}\exp(-\tfrac{n}{\xi_p(\vec x)} ).
$
This is thus incompatible with a uniform OZ-type behaviour in the subcritical regime. 

For $q=4$, the correlation length does diverge as $p \nearrow p_c$, as for $q <4$. 
However, contrary to the case $q < 4$, we expect that the ratio between the typical width of a long cluster and the correlation length diverges as $p \nearrow p_c$, possibly logarithmically. 
The core difference between the cases $q <4$ and $q = 4$ is that for $q  < 4$ the RSW theory crossing estimate holds up to the boundary, with adverse boundary conditions. 
This is ultimately reflected in Lemma~\ref{lem:L(p)expdecay}, which we do not believe holds for $q = 4$; see Section~\ref{sec:background} for details. 

That being said, our strategy may be used for any $q \geq 1$ and $p < p_c$ (and even any dimension $d \geq2$) to obtain OZ results which are uniform in $n$ and $\vec v$, 
but not in $p < p_c$.
Thus, the present paper offers an alternative approach to the theory of~\cite{campaninonioffevelenikozrandomcluster}.
\end{Rem}

\subsection{Overview of the proof}

Fix some~$\vec w \in \mathbb S^1$. We aim to describe the cluster of the origin, conditioned to hit a distant half-plane in the direction $\vec w$. 
The idea is to slice the plane with lines orthogonal to~$\vec w$ placed 
at regular intervals of length comparable to~$\xi_p(\vec w)$. 
We call these lines {\em hyperplanes}, as they should have co-dimension~$1$ in the more general~$d$-dimensional setting. 
For simplicity of exposition, imagine that~$\vec w = e_1$ is the horizontal unit vector. This is not limiting, as the symmetries of the lattice are never used. 

The cluster of~$0$ is then explored from left to right in a Markovian way: 
explore the connected component in the half-space left of the~$k^{\rm th}$ hyperplane and write~$X_k$ its highest point on the hyperplane\footnote{The choice of $X_k$ as highest on the hyperplane is arbitrary. Any choice determined by the explored cluster should have the same properties}. 
When no such intersection exists, write~$X_k = \dagger$ and say that~$X_k$ dies.
We will show that the process~$(X_k)_k$ has a certain renewal structure, even when conditioned on surviving for~$n$ steps.
As such, the process is decomposed into irreducible pieces and behaves essentially as a random walk, with all consequences following by standard tools. 

The ideas are similar to the previous works such as~\cite{campaninonioffevelenikozrandomcluster}, but we believe are rephrased in a novel way. 
The decomposition of the cluster in irreducible pieces, and its random walk interpretation, appeared already in these works. 
However, the present work does not use the ``diamond decomposition'' and explores the cluster in a dynamical way, that is less dependent on the point it ultimately reaches.

\section{Background on the random-cluster model}\label{sec:background}

In this section we provide some more elements from the standard theory of the random-cluster model, which readers familiar with the model may skip.
We also review some recent results of~\cite{DuminilCopinManolescuScalingRelations} about the near-critical regime of the two-dimensional random-cluster model.
Fix~$q \geq 1$ for the whole section. 

\subsection{Standard properties of the random-cluster model}

\paragraph{Monotonicity properties.}
The set of percolation configurations on~$\Z^2$ is equipped with the partial order defined by~$\om \leq \om'$ if for any~$e \in E(\Z^2)$,~$\om(e) \leq \om'(e)$.
An event~$A$ is said to be \emph{increasing} if for any two percolation configurations~$\om \leq \om'$,~$\om \in A \Rightarrow \om' \in A$.
Set~$q \geq 1$,~$p\in [0,1]$ and~$G$ a finite subgraph of~$\Z^2$.
The \emph{FKG inequality} asserts that for any increasing events~$A,B$, and any boundary condition~$\eta$ on~$G$,
\begin{equation}\tag{FKG}\label{eq:FKG}
\phi^\eta_{G,p}[A \cap B] \geq \phi^\eta_{G,p}[A]\phi^\eta_{G,p}[B].
\end{equation}

The random-cluster measure also possesses the following monotonicity property. 
If~$\xi \leq \xi'$  are boundary conditions on~$G$ (which is to say that $\xi'$ is a coarser partition than $\xi$), then for any increasing event~$A$,
\begin{equation}\tag{MON}\label{eq:mon}
\phi^\xi_{G,p}[A] \leq  \phi^{\xi'}_{G, p}[A].
\end{equation}
Both properties above extend to infinite-volume limits.

\paragraph{Domain Markov property.}
Let~$G$ be a finite subgraph of~$\Z^2$. Fix~$q \geq 1$ and~$p \in (0,1)$.
Let~$G' = (V',E')$ be a subgraph of~$G$. Then for any boundary condition~$\eta$ on~$G$, any percolation configuration~$\xi \in\{0,1\}^{E \setminus G'}$,
\begin{equation}\label{equ: DMP}\tag{DMP}
\phi^\eta_{G,p}[ \cdot_{G'} \vert~\om_{E\setminus E'} = \xi ] = \phi^{\xi^\eta}_{G',p}[\cdot],
\end{equation} 
where~$\xi^\eta$ is the boundary condition induced on the complement of~$G'$ by~$\xi$ together with the boundary condition~$\eta$.

\paragraph{Duality.}
Consider the dual graph~$(\Z^2)^*$ with vertex set~$V(\Z^2)+(1/2,1/2)$ and edge set~$\{ i + (1/2), j+(1/2) \},$ for~$i,j$ such that~$\{i,j\}\in E(\Z^2)$.
To any percolation configuration on~$\Z^2$ we associate its \emph{dual configuration}, defined on the graph~$(\Z^2)^*$ by setting~$\om^*(e)= 1- \om(e^*)$, where~$e^*$ is the unique edge of~$(\Z^2)^*$ that crosses~$e$.
It is classical that when~$\om \sim \phi^0_{p}$ then~$\om^* \sim \phi^1_{p^*}$, where~$p$ and~$p^*$ are linked by the following duality relation:
\begin{equation*}
pp^* = q(1-p)(1-p^*).
\end{equation*}
Note that the value~$p_{\mathrm{sd}} := \frac{\sqrt q}{1+\sqrt q}$ is the unique solution of~$p = p^*$, and coincides with the critical parameter~$p_c$ as first proved in~\cite{duminilbeffara}.

\subsection{Near-critical theory}\label{section: near-critical RSW theory}

The near-critical regime of the random-cluster model is the set of parameters~$n$ and~$p$ for which~$n$ is sufficiently small --- or~$p$ sufficiently close to~$p_c$ --- so that the system behaves almost critically at scale~$n$. It is expected, and is indeed the case for the two-dimensional random-cluster model, that the system behaves roughly critically in the near-critical regime, and sub- or super-critically outside of it. 

The rigorous understanding of the near-critical regime of percolation models in two dimensions started with Kesten's seminal work on Bernoulli percolation~\cite{Kestenscalingrelations}. Kesten's results were adapted to the random-cluster model on~$\mathbb Z^2$ with~$q \in [1,4]$ in~\cite{DuminilCopinManolescuScalingRelations}.
Here we will mention only the consequences of these works that are relevant to us. 

For~$q \geq1$ fixed and~$p < p_c$, define the {\em characteristic length} 
\begin{align}
L(p) = \inf \big\{ n \geq 0: \phi_p[{\rm Cross}(\Lambda_n)] \notin [\delta,1-\delta]\big\},
\end{align}
where~${\rm Cross}(\Lambda_n)$ is the event that~$\Lambda_n$ contains an open path between its left and right sides,
and~$\delta > 0$ is some small fixed quantity. 

It was proved in~\cite{DuminilCopinManolescuScalingRelations} that, for any~$\vec v \in \mathbb S^1$, 
\begin{align}\label{eq:L(p)xi}
	L(p) \asymp \xi_p(\vec v)
\end{align}
uniformly in~$p <p_c$, for~$q \in [1,4]$.
Moreover, $L(p) \to \infty$ as $p \nearrow p_c$. 

The equivalence \eqref{eq:L(p)xi} also holds for~$q >4$, but only because both~$L(p)$ and~$\xi_p(\vec v)$ are uniformly bounded away from~$0$ and~$\infty$.

Henceforth we will use~$L(p)$ rather than~$\xi_p(\vec v)$ to designate a quantity of the order of the correlation length, 
for instance when referring to the interspacing of the hyperplanes used to define the process~$(X_n)_n$. 
We do so to emphasise that its use is not related to the direction~$\vec v$ and is only important up to a bounded multiplicative constant. 

Fix~$\vec w \in \mathbb S^1$; for all practical purposes, think of~$\vec w$ as the unit vector in the horizontal direction.
Define the half-spaces 
\begin{align*}
	\calH_{\leq t}^{\vec w} = \{ x \in \R^2 : \langle x, \vec w \rangle \leq t\,  L(p) \} \quad\text{ and }\quad
	\calH_{\geq t}^{\vec w} = \{ x \in \R^2 : \langle x, \vec w \rangle \geq t\,  L(p) \}.
\end{align*}
Call~$\partial \calH_{\leq t}^{\vec w}= \partial \calH_{\geq t}^{\vec w} = \{ x \in \R^2 : \langle x,\vec w \rangle = t \, L(p) \}$ a hyperplane
and set~$\calH_{< t}^{\vec w}= \calH_{\leq t}^{\vec w} \setminus \partial \calH_{\leq t}^{\vec w}$.
In the following we will use an arbitrary integer approximation of these objects, which we do not detail. 
We will mostly work with~$\vec w$ fixed, and will omit it from the notation whenever no ambiguity is possible. 

A particular consequence of the discussion above is the following bound on the speed of exponential decay. 
The particular form of the statement below is due to its specific use in our proofs. 
A {\em potential past} is any connected set of edges $A \subset \calH_{\leq 0}$. Write $\partial_{\leq 0} A$ for all edges adjacent to $A$, contained in $\calH_{\leq 0}$. 

\begin{Lemma}\label{lem:L(p)expdecay}
	Fix~$q \in [1,4)$. There exists a constant~$c> 0$ such that the following holds. 
	For any~$p < p_c$, any~$n \geq L(p)$ and any potential past $A$, 
	\begin{align}\label{eq:L(p)expdecay}
		\phi_{\Lambda_n}^1\big[\La_{L(p)}\xlra{A^c} \partial \La_n \,\big|\, \omega \equiv 1 \text{ on $A$ and } \omega \equiv 0 \text{ on $\partial_{\leq 0} A$}\big] \les \e^{-c n/L(p)}. 
	\end{align}
\end{Lemma}

The novelty in the statement above is the conditioning on the edges in $A$ and $\partial_{\leq 0}A$; 
indeed the result for the unconditioned measure is a simple consequence of \cite{DuminilCopinManolescuScalingRelations}. 
Note that the connection in \eqref{eq:L(p)expdecay} needs to occur \emph{outside} of $A$. 
Thus, it would be tempting to assume that the probability in \eqref{eq:L(p)expdecay}  decreases with~$A$, 
but the potential boundary conditions $A$ induces may actually increase the probability of the connection.  
Indeed, if we set $D = \La_n \setminus (A \cup \partial_{\leq 0} A)$, we work above in the measure $\phi_D^\xi$, 
where $\xi$ are the boundary conditions on $D$ that are wired the endpoints of edges of $A$ on $\partial \calH_{\leq 0}$, wired on $\partial \La_n$, but free on the endpoints of $\partial_{\leq 0} A$.

Finally, the assumption $q < 4$ is important for the above. For $q > 4$ the result is known to fail, and is expected to fail for $q =4$. 
In the latter case, if $A = \calH_{\leq 0}$, the cluster of the origin may  ``creep along'' $\partial \calH_{\leq 0}$ at a slightly smaller cost than if $A = \emptyset$.

We turn our attention to scales below the characteristic length, called the {\em critical window}. 
It was proved in~\cite[Thm. 2.1]{DuminilCopinManolescuScalingRelations}  that the RSW property still holds in this regime. 
Let~${\rm Circ}(r,R)$ be the event that~$\La_R$ contains an open circuit surrounding~$\La_r$. 
Write~${\rm Circ}^*(r,R)$ for the event that the dual configuration contains such a circuit, 
which for the primal model translates to~$\La_r$ not being connected to~$\partial \La_R$. 

\begin{Prop}[RSW in the critical window]\label{prop:RSW}
	Fix~$q \in [1,4]$. There exists~$c > 0$ such that for any~$p < p_c$ and~$n \leq L(p)$ and any boundary condition~$\xi$ on~$\La_{2n}$,
    \begin{equation}\label{eq:RSWnc}\tag{RSW}
        c\le  \phi_{\La_{2n},p}^\xi[{\rm Circ}(n,2n)] \le 1-c \quad \text{ and } 
        \quad         c\le  \phi_{\La_{2n},p}^\xi[{\rm Circ}^*(n,2n)] \le 1-c.
    \end{equation}
\end{Prop}

Finally, the most significant contribution of~\cite{DuminilCopinManolescuScalingRelations} 
was to prove the stability of the arm event probabilities in the near-critical regime. 
\begin{Th}\label{thm: one arm} Fix~$q \in [1,4]$. Then
	\begin{equation}\label{eq:one_arm}
	\phi_{\Lambda_{2n},p}^\xi[0 \lra \partial \La_n] \asymp \pi_1(n),
	\end{equation}
	uniformly in~$p<p_c$,~$n \leq L(p)$ and any boundary conditions~$\xi$ on~$\partial \La_{2n}$
\end{Th}

We finish the section with the proof of Lemma~\ref{lem:L(p)expdecay}. 
This is a relatively standard consequence of the exponential decay above~$L(p)$ \cite[Prop. 2.13]{DuminilCopinManolescuScalingRelations}, 
but the authors are unaware of a specific reference. We will omit some details in the proof below.

\begin{proof}[Proof of Lemma~\ref{lem:L(p)expdecay}]
	Fix~$q \in [1,4)$ and~$p<p_c$; write $L = L(p)$.  
	Constants below are independent of~$p$. Figure~\ref{fig:L(p)expdecay} contains helpful illustrations for the proof. 
	
	Before we start, we mention that it is an immediate consequence of  \cite[Prop. 2.13]{DuminilCopinManolescuScalingRelations} that, 
	for all $n\geq L$,
	\begin{align}\label{eq:no_D}
	    	\phi_{\La_{2n}}^{1}  [ \La_n \lra \La_{2n} \big] \leq e^{-c n/L},
	\end{align}
	for some universal constant $c > 0$. In particular, this implies \eqref{eq:L(p)expdecay} for $A = \emptyset$. 
	
	We turn to the general proof.	
	For $A$ a potential past 
	set ${\rm Past}(A)$ for the event that all edges of $A$ are open, but all those of $\partial_{\leq 0}A$ are closed. 
	Also, write $A^*$ for all edges dual to edges of $A$ or $\partial_{\leq 0}A$. 
	
	For $k\geq 1$, set\footnote{Recall that $\La_r(a,b)$ is the translate of $\La_r$ by the vector $(a,b) \in \bbR^2$.}
	\begin{align}
			u_k =\inf_A \phi_{\La_{2^k L}}^{1} \big[ \La_{2^{k-2} L}(-2^{k-1} L, 0) \xlra{\omega^* \cup A^*} \La_{2^{k-2} L}(2^{k-1} L, 0) \,\big|\, {\rm Past}(A)\big],
	\end{align}
	where the infimum is over all potential pasts $A$.
	We will start by proving that there exists some universal constant~$c_1 > 0$ such that, for any~$k\geq 1$,
	\begin{align}\label{eq:lower_bound_u_k}
		u_{k} \geq 1 - e^{-c_1 2^k}.
	\end{align}

	We do this by proving that $u_1$ is bounded away from $0$ uniformly in $p$ and that 
	\begin{align}\label{eq:u_k_rec}
		u_{k+1} \geq (1 - \e^{-c_0 2^k})^2 (1-(1-u_k)^2)
	\end{align}
	for all $k \geq 1$, where $c_0 > 0$ is a universal constant. 
	
	To prove  the bound on $u_1$, consider a potential past $A$. 
	If $A\cap \La_{2L}$ has a diameter smaller than $L$, there exists a rectangle of width $L/2$ and length $2L$ connecting the boxes $\La_{L/2}(-L, 0)$ to $\La_{L/2}(L, 0)$. 
	Using \eqref{eq:RSWnc}, we prove that the two boxes are connected in $\omega^*$ with uniformly positive probability. 
	If $A\cap \La_{2L}$ has a diameter larger than $L$, there exists tubes of length at most $4L$ and width at most $L/2$ that connect $\La_{L/2}(-L, 0)$ and $\La_{L/2}(L, 0)$ to $A$. Then\footnote{It is essential here that $q < 4$; this estimate is expected to fail at $q=4$.}, applying~\cite[Thm.7]{DumManTas21}, we find
	\begin{align}
	    	&\phi_{\La_{2 L}}^{1}  [ \La_{L/2}(-L, 0) \xlra{\omega^*} A^* \,\big|\, {\rm Past}(A)\big] \ges 1 
  \text{ and } 	
    		\phi_{ \La_{2L}}^{1}  [ \La_{L/2}(L, 0) \xlra{\omega^*} A^*\,\big|\, {\rm Past}(A) \big] \ges 1.
	\end{align}
	Applying \eqref{eq:FKG}, we conclude that, in all cases,
	\begin{align}		
		\phi_{ \La_{2L}}^{1} \big[ \La_{L/2}(-L, 0) \xlra{\omega^* \cup A^*} \La_{L/2}(L, 0)\,\big|\, {\rm Past}(A)\big] \ges 1.
	\end{align}
	
	We turn to the recurrence relation \eqref{eq:u_k_rec}. Fix $k \geq 1$ and $A$ a potential past. Write $N = 2^k L$. 
	Consider the two non-overlapping vertical translates of~$\La_{N}$ by~$(0,-N)$ and~$(0,N)$, 
	respectively, and the corresponding translates  $H_+$ and $H_-$ of the events in the definition of $u_k$:
	\begin{align}
		H_\pm = \Big\{ \La_{N/4}(-N/2, \pm N) \xlra{(\omega^* \cup A^*) \cap  \La_{N} ( 0,\pm N) } \La_{N/4}(N/2, \pm N)  \Big\}.
	\end{align}
	See Figure~\ref{fig:L(p)expdecay}, left diagram.
		
	Under~$\phi_{\La_{2N}}^{1}[.\,|\, {\rm Past}(A)]$, that the events $H_+$ and $H_-$ occur may be bounded by independent Bernoulli variables of parameter~$u_k$. 
	Thus, 
	\begin{align}
	\phi_{\La_{2N}}^{1}\big[ H_+ \cup H_- \,\big|\, {\rm Past}(A) \big] \geq 1-(1-u_k)^2.
	\end{align}
	
	\begin{figure}
	\begin{center}
	\includegraphics[width = 0.4\textwidth, page =1]{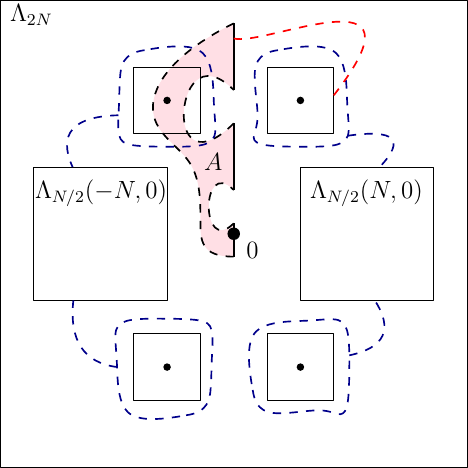}\qquad \qquad \qquad 
	\includegraphics[width = 0.4\textwidth, page =2]{L_p_expdecay3.pdf}	
	\caption{{\em Left:} in the proof of~\eqref{eq:u_k_rec} we consider two translates $H_+$ and $H_-$ of the event in $u_k$, 
	one in the upper half-plane, one in the lower one.
	Here the red path ensures that $H_+$ occurs. The blue paths (also part of $\omega^* \cup A^*$) produce the events $B_L$ and $B_R$. 
	When at least one of $H_+$ and $H_-$ occur, and both  $B_L$ and $B_R$, then the two horizontal translates of $\La_{N/2}$ connect to each-other. 
	{\em Right:} To produce a path separating $\La_{2^k L}$ from $\partial \La_{2^{k+1} L}$, it suffices for two vertical translates of the events in $u_{k-1}$ to occur, 
	and for the blue paths to occur in $\omega^* \cup A^*$. As for $B_L$ and $B_R$, the latter occur with exponentially high probability in $2^k$.}
	\label{fig:L(p)expdecay}	
	\end{center}
	\end{figure}

	Now, let $B_R$ be the event that, in $\omega^* \cap \La_{2N} \cap \calH_{> 0}$,
	there exists a circuit around $\La_{N/4 }(N/2,N)$, contained in $\La_{N/2 }(N/2,N)$, 
	connected to a circuit around $\La_{N/4 }(N/2,-N)$, contained in $\La_{N/2 }(N/2,-N)$.
	Then, by~\eqref{eq:FKG} and~\eqref{eq:no_D},
	\begin{align}
		\phi_{\La_{2N}}^{}\big[ B_R \,\big|\, {\rm Past}(A) \big] 
		\geq \phi_{\La_{2N} \cap \calH_{\geq 0}}^{1}\big[ B_R \big]  \geq 1 - e^{-c_0 2^k}.
	\end{align}
	for some universal positive constant $c_0 >0$. 
		
	Similarly, define $B_L$ be the event that, in $(\omega^* \cup A^*) \cap \La_{2N} \cap  \calH_{< 0}$,
	there exists a circuit surrounding $\La_{N/4 }(-N/2,N)$, contained in $\La_{N/2 }(-N/2,N)$, 
	connected to a circuit surrounding $\La_{N/4 }(-N/2,-N)$, contained in $\La_{N/2 }(-N/2,-N)$.
	Due to \eqref{eq:mon}, when sampled according to $\phi_{\La_{2N}}^1[. \,|\, {\rm Past}(A) ]$, $\omega^* \cup A^*$ in $\calH_{\leq 0}$ dominates the dual configuration sampled according to $\phi_{\La_{2N} \cap \calH_{\leq 0}}^{1}$.
	As such 
	\begin{align}
		\phi_{\La_{2N}}^{1}\big[ B_L \,\big|\, {\rm Past}(A) \big] \geq 	\phi_{\La_{2N} \cap\calH_{\leq 0}}^{1}\big[ B_L \big]  \geq 1 - e^{-c_0 2^k}.
	\end{align}
	
	Notice that, when $ B_L$ and $B_R$ and at least one of $H_+$ or $H_-$ occur, $\La_{N/2}(-N, 0)$ is connected to $\La_{N/2}(N, 0)$ in $\omega^* \cup A^*$. 
	By \eqref{eq:FKG} we conclude that
	\begin{align}
	 \phi_{\La_{2N}}^{1}  \big[ \La_{N/2}(-N, 0) \xlra{\omega^* \cup A^*} \La_{N/2}(N, 0)\,\big|\, {\rm Past}(A)\big] &\\
	\geq\phi_{\La_{2N}}^{1}\big[ (H_+ \cup H_- ) \cap B_L \cap B_R \,\big|\, {\rm Past}(A)\big] 
	&\geq (1 - e^{-c_0 2^k})^2 (1-(1-u_k)^2).
	\label{eq:AABB}
	\end{align}
	This concludes the proof of~\eqref{eq:u_k_rec}.
	Finally, the lower bound on $u_1$ and \eqref{eq:u_k_rec} readily imply \eqref{eq:lower_bound_u_k}.

	We turn to proving \eqref{eq:L(p)expdecay}. Let $k = \lfloor \log_2 (n/L(p))\rfloor -1$. It suffices to consider the cases $k\geq 2$;
	\eqref{eq:RSWnc} allow us to extend \eqref{eq:L(p)expdecay} to smaller values of $n$ by potentially altering the constant.  
	Then
	\begin{align}\label{eq:link_La_u_k}
		\phi_{\La_n }^1\big[\La_L\lra \partial \La_n 	\,\big|\, {\rm Past}(A)\big] 
		 \leq 1 - \phi_{\La_{2^{k+1}L} }^1\big[\text{$\omega^* \cup A^*$ has circuit around } \La_{2^{k}L}\,\big|\, {\rm Past}(A)\big].
	\end{align}
	The event in the right-hand side above may be constructed by two instances of the events in the definition of $u_{k-1}$ 
	(translated by $(0,3 \cdot 2^{k-1}L)$ and $(0,-3 \cdot 2^{k-1}L)$, respectively)
	and a series of connection events relating only to the left and right half-planes. 
	Rather than defining these precisely, we direct the reader to Figure~\ref{fig:L(p)expdecay}, right diagram, 
	and to the explanation behind \eqref{eq:AABB}, which prove that 
	\begin{align}
		\phi_{\La_{2^{k+1}L}}^1\big[\text{$\omega^* \cup A^*$  has circuit around } \La_{2^{k}L}\,\big|\, {\rm Past}(A)\big]
		\geq (1-e^{-c_2 2^{k}})^2 u_{k-1}^2. 
	\end{align}
	Combining the above with \eqref{eq:lower_bound_u_k} and inserting the result in \eqref{eq:link_La_u_k}
	we obtain  \eqref{eq:L(p)expdecay} with a constant $c >0$ that is independent of $p$ and $A$. 
\end{proof}

\section{Exploring the cluster as a Markov renewal process} 
 
The key idea of our approach is to show that the subcritical cluster of the origin admits what we call a ``killed renewal structure", and that the renewal structure persists even when conditioned to hit a far-away half-plane. We start by defining a class of processes which will have a killed random-walk like behaviour. 

\begin{Def}\label{def:KMRP}
	A stochastic process~$(X_t, Y_t )_{t\geq 1} \in ((\R \cup \{\dagger\})\times \{0,1\})^\N$ is called a 
	\emph{killed Markov renewal process} (KMRP in short) with respect to some filtration~$(\calF_t)_{t\in \N}$ if 
\begin{itemize}
\item It is adapted to~$(\calF_t)_{t\in \N}$;
\item If~$X_t = \dagger$, then~$Y_t = 0$ and~$X_{t+1} = \dagger$; 
\item If we set~$T_0 =  0$ and~$T_{k} = \inf\{ t > T_{k-1} : Y_t = 1\}$, then,
for any~$k\geq 1$, the process~$(X_{t+T_k} - X_{T_k}, Y_{t+T_k} )_{t \in \N}$ conditionally on~$\calF_{T_k}$ and~$X_{T_k}\neq \dagger$ has a fixed law~$\mathcal L$.
\end{itemize}
We say that the process has {\em exponential tails} if there exists~$c > 0 $ such that, for all~$k,j,n\geq 1$,
\begin{align}
	\mathbb P\big[X_{n + T_k} \neq \dagger \text{ and }  T_{k+1} - T_k  > n \,\big|\,\calF_{T_k}, \, T_k < \infty\big]& \leq \exp(-c \,n) \quad\qquad \text{ and}  \label{eq:KMRP_exp_tail1}\\	
	\mathbb P\big[X_{T_k+j} \neq \dagger \text{ and }|X_{T_k+j} - X_{T_k}| \geq n \,\big|\,\calF_{T_k}, \, T_k < \infty\big] &\leq \exp(-c \,n/j).  \label{eq:KMRP_exp_tail2}
\end{align}
Moreover, for the initial step, for any $j,n\geq 1$,
\begin{align}
		\mathbb P\big[X_n \neq \dagger \text{ and } T_{1} > n  \,\big|\, X_1 \neq \dagger\big]& \leq \e^{-c n}\quad\qquad \text{ and} \label{eq:KMRP_exp_tail_init1}\\	
	\mathbb P\big[X_{j} \neq \dagger \text{ and }|X_{j}| \geq n \,\big|\, X_1 \neq \dagger\big]& \leq \e^{-c n/j}. \label{eq:KMRP_exp_tail_init2}
\end{align}

We say that the model exhibits a {\em mass-gap} $c > 0$ if the above displays also hold under the additional conditioning $X_{T_k + N} \neq \dagger$ for any $N \geq 1$. 

Finally, we call the process \emph{aperiodic} if the greatest common divisor of the support of the increment $T_2 - T_1$ is one. 
\end{Def}

For such a process and~$k\geq 1$, call the laws of~$X_{T_{k+1}} - X_{T_k}$ and~$T_{k+1} - T_k$ conditionally on~$T_k \neq \infty$
the~$X$-step and the~$T$-step (or vertical and horizontal steps, respectively). 
Notice that these only depend on the law~$\mathcal L$ mentioned above, and therefore do not depend on~$k$.
Define the~$X$-step mean and variance as 
\begin{align*}
\mu_X := \mathbb E[X_{T_{k+1}} - X_{T_k}\,|\, T_{k+1} < \infty]\quad \text{and }\quad
\sigma_X :={\rm Var}[X_{T_{k+1}} - X_{T_k}\,|\, T_{k+1} < \infty]
\end{align*}
and the killing rate as 
\begin{align*}
\kappa := \mathbb P[T_{k+1}=\infty \,|\, T_{k} < \infty].
\end{align*}
We will now describe how such general processes are related to the cluster of~$0$ under different conditionings. \medskip

Fix~$\vec w \in \mathbb S^1$. 
For all practical purposes, think of~$\vec w$ as the unit vector in the horizontal direction, pointing to the right; 
this is only so that the vocabulary and illustrations below make sense, but has no bearing on the arguments. 
Recall the definitions of the half-spaces~$\calH_{\leq t}^{\vec w}$ and~$\calH_{\geq t}^{\vec w}$; the dependence on~$\vec w$ will be omitted whenever possible. 

Write~$\sfC$ for the cluster of~$0$ and~$\sfC_{\leq t}$ for the cluster of~$0$ in~$\omega \cap \calH_{\leq t}$. 
Note that this is contained in, but not always equal to~$\sfC\cap \calH_{\leq t}$. 
For any~$t\in\Z$, set
\begin{align}\label{eq:X_t_def}
	X_t := \max\big\{ h  \in \R: t L(p) \cdot \vec w + h \cdot \vec w^\perp \in \sfC_{\leq t}\big\},
\end{align}
where~$\vec w^\perp\in \bbS^1$ is a unit vector orthogonal to~$\vec w$ (the direction of~$\vec w^\perp$ is irrelevant for now; think of it as pointing upwards). 
If the set above is empty, which is to say that~$\sfC$ does not intersect~$\calH_{\geq t}$, set~$X_t = \dagger$. 
Thus~$X_t$ is the ``highest'' coordinate of the intersection of~$\sfC_{\leq t}$ with~$ \partial \calH_{\leq t}$.

The main objective of this section is the following result. 

\begin{Th}\label{thm:killed_renewal_structure} 
	Fix~$q \in [1,4)$,~$p < p_c$ and~$\vec w \in \mathbb S^1$; sample $\sfC$ according to $\phi_p$ and define $(X_t)_{t\geq 0}$ as above. 
	There exists an enlarged probability space supporting a random process~$Y_t$
	such that~$(X_t/L(p),Y_t)$ is an aperiodic killed renewal Markov process with a mass-gap, killing rate and $X$-step variance, all bounded away from~$0$ uniformly in~$p$ and~$\vec w$.
	Moreover, the killing rate is also bounded uniformly away from~$1$.
	
	Finally, the initial step survival rate satisfies
	\begin{align}\label{equ: first step survival rate}
		\phi_p[X_1 \neq \dagger] \asymp \phi_p[T_1 < \infty] \asymp \pi_1(L(p)),
	\end{align}
	uniformly in~$p$ and~$\vec w$.
\end{Th}

\begin{Rem}\label{rem:cone_contained}
	By the construction below, we will ensure that, for all renewal times $T_k$, the cluster $\sfC_{\geq T_k} := \sfC\setminus \sfC_{\leq T_k}$ after $T_k$ 
	is contained in a cone of fixed aperture and apex $(L(T_k-1), X_{T_k})$. 
	Since the inter-renewal times have exponential tails, this cone-containment property ensures that the full cluster remains close to the linear interpolation of the points $(Lt,X_t)_{t\geq 0}$.
\end{Rem}

Theorem~\ref{thm:killed_renewal_structure} suffices to prove Theorem~\ref{thm:main}, as well as a large number of other properties of subcritical clusters. 
We chose to formulate it using the concept of KMRP so as to separate the model-dependent part of the argument from the generic analysis of a class of processes with random-walk behaviour. The latter is contained in Section~\ref{sec:KMRP_consequences}. 
While the formalism is new in this context, we do not claim the theorem to be entirely original.
Indeed, the only formal novelty compared to~\cite{campaninonioffevelenikozrandomcluster} is the uniformity in~$p < p_c$.

We mentioned uniform lower bounds on~$\sigma_X$, but no upper bounds, as the uniform exponential tails induce uniform upper bounds on~$\sigma_X$,~$\mu_X$. 
Note that we do not claim that~$\mu_X= 0$. This is the case when~$\vec w$ is aligned to the coordinate axis due to symmetry, but is not generally true. As such, under the conditioning on~$\{X_n \neq \dagger\}$, the cluster does not ``aim'' for the point~$n\vec w$, but rather for a point~$n\vec v$ for some~$\vec v$ depending on~$\vec w$ (with~$\langle \vec v, \vec w\rangle = 1$). The relation between~$\vec w$ and~$\vec v$ will yield the strict convexity of the Wulff shape and ultimately will prove how~$\vec w$ needs to be chosen to deduce Theorem~\ref{thm:main} for some direction~$\vec v$ -- see Section~\ref{sec:Wulff} for details.

The proof of Theorem~\ref{thm:killed_renewal_structure} relies on a geometric analysis of~$\sfC$ under the survival event~$\{X_N\neq \dagger\}$. The analysis is performed at a scale~$L(p)$, in particular showing that there exists a density of ``times'' at which the cluster is confined to boxes of size~$L(p)$. Those boxes play the role of ``pre-renewal times''; indeed, it will be shown that at each pre-renewal time, there is a   uniformly positive probability that the ``future'' cluster is sampled independently from its past, thus generating an actual renewal. The pre-renewal times play a similar role to the cone points in~\cite{campaninonioffevelenikozrandomcluster}.

 \begin{Rem}\label{rem:KMRP_pure_exp}
	Any KMRP with exponential tails and a mass-gap has a pure exponential rate of survival
\begin{align*}
	\mathbb P [X_n \neq \dagger] \asymp \exp(- n/\zeta),
\end{align*}
as will be proved in Section~\ref{sec:KMRP_consequences}. 
In our context,~$\zeta = \zeta(p,\vec w)$ depends on~$p$ and~$\vec w$ but is uniformly bounded away from~$0$ and~$\infty$. 
The constants in~$\asymp$ above are not uniform in~$p$; they will be shown to be of order~$\pi_1(L(p))$ and are due to the requirement of survival up to the first renewal time. 
Finally, the mass-gap states that the exponential rate of survival with no renewals is strictly smaller than~$\zeta$. 
\end{Rem}

\begin{Rem}\label{rem:KMRP_infinite_process}
Theorem~\ref{thm:killed_renewal_structure} provides a rigorous construction of ``the infinite cluster conditioned to survive in the direction~$\vec{w}$''. 
Indeed, it is a general property of  KMRPs with a mass-gap that the distribution of $T_{k+1} - T_k$ under the measure 
$\mathbb P[\, \cdot \,|\,\calF_{T_k}, \, T_k < \infty,\, X_n \neq \dagger ]$ converges as $n\to\infty$ to a measure fully supported on $\mathbb N^*$, with exponential tail. 
Remark~\ref{rem: size_bias} allows to identify this distribution as an exponential tilt of the distribution of $T_{k+1}-T_k$ conditioned on $T_1 < \infty$. 

The same type of convergence applies to the law of $(X_{T_k+1} - X_{T_k},\dots, X_{T_{k+1}} - X_{T_k})$, with the limit being supported on finite sequences. 
If we call $\calL_{\rm irred}$ this limiting law, we may construct an infinite sequence $(X_t)$ by concatenating i.i.d. samples from $\calL_{\rm irred}$.
An infinite cluster may also be constructed in a similar way, by sampling i.i.d. pieces of cluster $\sfD_k$ --- see the proof of Theorem~\ref{thm:killed_renewal_structure} for the exact definition of $\sfD_k$ --- according to a similar limiting law. 
Its statistics (for instance the mean value or variance of the $T$-step) are different from the ones of $\calL$ conditioned to survive for only one step, as $\calL_{\rm irred}$ is an exponential tilt of $\calL$. 
\end{Rem}

\begin{Rem}\label{rem:KMRP_condensation}
For KMRP with exponential tails but no mass-gap a \emph{condensation} phenomenon can occur, in which the process survives by making a very large step of linear order rather than many small steps of constant order. 
This was discussed in the context of long range Ising models in~\cite{nonanalyticitycorrelationlength, 2ptfunctionsaturation}.
\end{Rem}

The rest of the section is dedicated to proving Theorem~\ref{thm:killed_renewal_structure}. 
For the rest of the section,~$q \in [ 1,4)$ and~$p< p_c$ are fixed and we omit them from notation. 
A direction~$\vec w \in\bbS^1$ is also fixed and omitted from notation. For simplicity, we will assume $\vec w$ to be the horizontal unit vector~$(1,0)$.
Unless otherwise stated, constants and equivalences below will be uniform in~$p$ and~$\vec w$; they will be referred to as \emph{universal}.

\subsection{Cone-connections}

For~$\alpha > 0$, define  the cone in direction~$\vec w$ with aperture~$2\arctan \alpha$ as 
\begin{align}
	\calY_\alpha = \{ z \in \R^2: |\langle z,\vec w^\perp \rangle| \leq  \alpha \langle z,\vec w \rangle\}. 
\end{align}
Also set~$\partial \calY_\alpha$ to be the set of vertices in~$\calY_\alpha$ that have at least one neighbour outside~$\calY_\alpha$. 

Recall the notion of potential past $A$ and the event ${\rm Past}(A)$ that requires that the edges of $A$ are all open, while those of $\partial_{\leq 0}A$ are closed. 

\begin{Prop}[Connections in cones]\label{prop:cone_contained}
	There exist constants~$\alpha, c >0$ such that for  any~$k,n \geq 1$, and any potential past $A$,
		\begin{align}\label{eq:cone_contained}
		\phi \big[\La_{  L  } \overset{A^c}{\lra} (\calY_\alpha -(k  L  ,0))^c \,\big|\, \La_{  L  } \lra \calH_{\geq n} \text{ and } {\rm Past}(A)\big]
	\leq \e^{-c k}. 
	\end{align}
\end{Prop}

The above states that, when long connections in the direction~$\vec w$ occur, the whole connecting cluster is contained in a cone with high probability. 
Note that the aperture of the cone is fixed, and is not claimed to be arbitrarily small. Indeed, Proposition~\ref{prop:cone_contained} may ultimately be shown to hold for any~$\alpha > 0$, but only after proving the strict convexity of the Wulff shape. For now, it suffices to consider a fixed, large~$\alpha$.

Before proving Proposition~\ref{prop:cone_contained}, we formulate a useful mixing property. 
The proof is surprisingly intricate, but follows the argument of~\cite[Prop. 2.9]{DuminilCopinManolescuScalingRelations}.

\begin{Lemma}\label{lem:cone_to_half-plane_mixing} 
	For any~$\alpha>0$ and events~$A$ and~$B$ depending on the edges of~$\calH_{\leq 0}$ and~$\calY_\alpha \cap \calH_{\geq 1}$, respectively, 
	\begin{equation}\label{eq:cone_to_half-plane_mixing} 
		\phi[A\cap B] \asymp \phi[A]\phi[B],
	\end{equation} 
	where the constants in~$\asymp$ depend on~$\alpha$ but not on~$p$, $A$ or $B$. 
	
	Moreover, there exists $c >0$ such that, 
	for any $A$ as above and $B$ depending only on edges in $\calY_{\alpha} + (kL,0)$ for some $k\geq 1$
	\begin{equation}\label{eq:cone_to_half-plane_mixing2} 
		\Big|\frac{\phi[A\cap B]}{ \phi[A]\phi[B] }- 1\Big| \les \e^{-ck}.
	\end{equation} 
\end{Lemma}

Note that the events $A$ and $B$ are not assumed to be increasing or decreasing. This is why the proof below is more complicated than one might expect. 

\begin{proof}
Fix~$\alpha > 0$.  We will focus on proving~\eqref{eq:cone_to_half-plane_mixing}; the proof of \eqref{eq:cone_to_half-plane_mixing2} is explained at the end. 

Write~$\calD = \calY_{2\alpha}\cap\calH_{\geq 1/2}$ and~$\calE =\calY_\alpha\cap\calH_{\geq 1}$. 
We first prove that there exists a constant~$\eta < 1$  depending on~$\alpha$ but not on~$p$, such that 
\begin{equation}\label{eq:cone_to_cone}
 \phi_{\calD\setminus \calE}^1\big[\partial\calD \leftrightarrow \calE \big] \leq \eta \quad \text{ and }\quad
  \phi_{\calH_{\geq 0}\setminus \calD}^1\big[\partial\calH_{\geq 0} \leftrightarrow \calD \big] \leq \eta.
\end{equation}
We focus on the second inequality and we write~$S$ for the set of vertices in~$\calH_{\geq 0}\setminus \calD$ that are at equal $L^\infty$-distance from~$\partial\calH_{\geq 0}$ and~$ \calD$ (up to an error of~$1$). 
Batch the points in~$S$ according to their distance to~$\calD$:
a batch~$S_k$ for~$k\geq 0$ is the set of points~$x\in S$ such that
\begin{align}
kL \leq {\rm dist}(x, \calD) < (k+1)L.
\end{align}
For the second event in~\eqref{eq:cone_to_cone} to occur,
at least one of the batches~$S_k$ needs to be connected to~$\calD$ --- see Figure~\ref{fig:cone_to_half-plane_mixing}. 
For $k$ large enough, we have 
\begin{align}
\phi_{\calH_{\geq 0}\setminus \calD}^1[S_k \lra \calD] \leq 2\phi_{\La_{c_1 kL}}^1[\La_{L/c_1}\lra \partial \La_{c_1kL}] 
\les \e^{- c_0k},
\end{align}
for some universal constant $c_1 >0$. The final inequality is due to~\eqref{eq:L(p)expdecay}. 

Summing over~$k \geq K_0$ we find
\begin{align}\label{eq:cone_to_cone2}
 \phi_{\calH_{\geq 0}\setminus \calD}^1\big[ \exists k\geq K_0 : \text{~$S_k \lra\calD$}]
 \leq C_0 \sum_{k\geq K_0 }e^{- c_0k}
 \leq \tfrac12,
\end{align}
provided that~$K_0$ is a large constant, which may depend on~$\alpha$, but not on~$L$. 
We may now use~\eqref{eq:RSWnc} to deduce the second inequality in~\eqref{eq:cone_to_cone}. 
Indeed, the separation scale between $\calD$ and $\partial \calH_{\geq 0}$ is by definition of order $L$, which is the proper scale to apply~\eqref{eq:RSWnc}.
The first inequality is proved in the same way. 

\begin{figure}
\begin{center}
\includegraphics[height = 4.4cm]{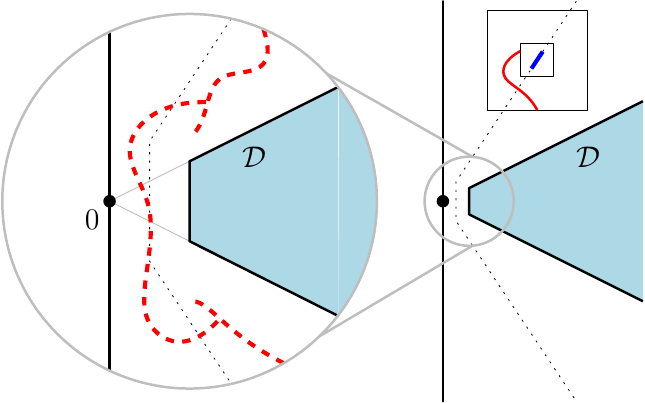}\hspace{2cm}
\includegraphics[height = 4.4cm]{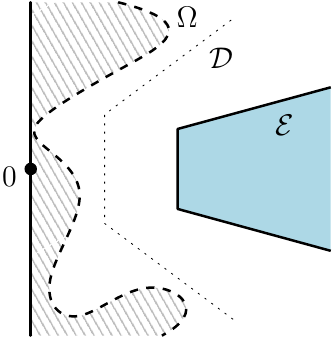}
\caption{{\em Left:} the dotted line represents the points at equal distance from~$\calD$ and~$\partial \calH_\geq 0$;
the solid black lines represent wired boundary conditions. The blue segment is the top half of a batch~$S_k$. 
For it to be connected to~$\calD$, the annulus surrounding it should be crossed, which occurs with an exponentially small probability in~$k$. For small values of~$k$, use~\eqref{eq:RSWnc} to complete the dual path separating~$\partial\calH_{\geq 0}$ form~$\calD$. 
\newline
{\em Right}: For the second part of the argument, consider a domain~$\Omega$ containing~$\calD$ and prove that the probability of~$B$ with arbitrary boundary conditions on~$\Omega$ is not much larger than that with free boundary conditions --- we use here the first estimate of~\eqref{eq:cone_to_cone}. 
One bound of~\eqref{eq:mixing_bc} follows directly. The opposite bound is proved by considering~$\Omega$ to be the complement of the cluster of~$\partial \calH_{\geq 0}$ and using the second bound of~\eqref{eq:cone_to_cone} to show that~$\calD \subset \Omega$ with positive probability. 
}
\label{fig:cone_to_half-plane_mixing}
\end{center}
\end{figure}

With~\eqref{eq:cone_to_cone} in hand, the rest of the proof follows that of~\cite[Prop. 2.9]{DuminilCopinManolescuScalingRelations}. 
Fix an event~$B$ depending on the edges in~$\calE$ and~$\zeta$ a boundary condition on~$\partial \calH_{\geq 0}$. We will show that
\begin{align}\label{eq:mixing_bc}
	 (1-\eta)^2 \phi_{\calH_{\geq 0}}^0[B] \leq \phi_{\calH_{\geq 0}}^\xi[B] \leq (1-\eta)^{-1} \phi_{\calH_{\geq 0}}^0[B].
\end{align}

We start off by proving that, for any set of edges~$\Omega$ with~$\calD \subset\Omega \subset \calH_{\geq 0}$
and any boundary conditions~$\xi$ on~$\partial \Omega$, 
\begin{align}\label{eq:OmegaB}
	\phi^{0}_\Omega[B] \geq (1 - \eta) \, \phi^\xi_\Omega[B].
\end{align}
See Figure~\ref{fig:cone_to_half-plane_mixing}, right-diagram, for an illustration. 
Indeed, consider the increasing coupling $\bfP$ between $\omega\sim\phi^{0}_\Omega$ and $\omega' \sim \phi^\xi_\Omega$
obtained by revealing first the cluster of $\partial \Omega$ in $\omega'$, then the rest of the configurations.
For details about this type of coupling see~\cite{DuminilCopinManolescuScalingRelations}; 
the important feature to keep in mind here is that $\omega  = \omega'$ for all edges not connected to $\partial \Omega$ in $\omega'$.  
Then 
\begin{align}
		\phi^\xi_\Omega[B]- \phi^{0}_\Omega[B]	
		&\leq \bfP\big[\omega \neq \omega' \text{ on $\calE$ and } \omega' \in B\big]\\
		&\leq \bfP\big[\partial \Omega \xlra{\omega'} \calE \text{ and } \omega' \in B\big]\\
		&\leq \phi^1_{\calD \setminus \calE} \big[\partial \calD \xlra{} \calE\big]\cdot\phi^\xi_\Omega[B]\\
		&\leq \eta\cdot\phi^\xi_\Omega[B], \label{eq:154mre}
\end{align}
where the third inequality is obtained by observing that 
\begin{align}
	\phi^\xi_{\Omega} \big[ \Omega \lra  \calE \,|\, B ] \leq \phi^1_{\calD\setminus \calE} \big[\partial \calD \lra \calE], 
\end{align}
and the last is due to the first bound in~\eqref{eq:cone_to_cone}. 
The above immediately implies~\eqref{eq:OmegaB}.
We insist here that $B$ is not assumed increasing and that the argument works because in the second line of~\eqref{eq:154mre}, 
the event depends only on $\omega'$. 
This is why the same reasoning may not be used to prove an opposite bound.

Applying~\eqref{eq:OmegaB} to~$\Omega = \calH_{\geq 0}$ we obtain the upper bound in~\eqref{eq:mixing_bc}. 
Another consequence of the above is that for all sets~$\Omega$ with $\calD \subset\Omega \subset \calH_{\geq 0}$, 
\begin{align}\label{eq:OmeOme}
	\phi^{0}_\Omega[B] \geq (1 - \eta) \, \phi^0_{\calH_\geq 0}[B],
\end{align}
since the probability on the right-hand side above is an average of quantities of the form~$\phi^{\xi}_\Omega[B]$.

We turn to the second bound in~\eqref{eq:mixing_bc}. 
Sample~$\omega$ from~$\phi_{\calH_{\geq 0}}^\zeta$ and let~$\sfC$ be the cluster of~$\partial \calH_{\geq 0}$.
Then
\begin{align}
	\phi^\zeta_{\calH_{\geq 0}}[B] \geq \sum_{\Omega} \phi^0_\Omega[B]\,\phi^\zeta_{\calH_{\geq 0}}[\calH_{\geq 0} \setminus \sfC = \Omega ],
\end{align}
where the sum is over all sets of edges~$\Omega$ containing~$\calD$. 
Using~\eqref{eq:OmeOme}, we conclude that 
\begin{align}
	\phi^\zeta_{\calH_{\geq 0}}[B] 
	\geq (1 - \eta)\, \phi^0_{\calH_\geq 0}[B]\,	\phi^\zeta_{\calH_{\geq 0}}[\calH_{\geq 0} \nxlra{} \calD ]
	\geq (1 - \eta)^2 \phi^0_{\calH_\geq 0}[B],
\end{align}
with the second inequality following from~\eqref{eq:cone_to_cone}.
This proves the second inequality in~\eqref{eq:mixing_bc}.

Finally, note that~\eqref{eq:mixing_bc} implies that, for any event~$A$ depending only on edges in~$\calH_{\leq 0}$, 
\begin{align}
	(1-\eta)^4 \leq \frac{\phi_{p}[B\,|\,A]}{\phi_{p}[B]} \leq (1-\eta)^{-2},
\end{align}
which yields the desired result.

To prove \eqref{eq:cone_to_half-plane_mixing2} follow the same lines, but observe that $\eta \les e^{-ck}$ for some $c > 0$, which implies the result. 
\end{proof}

We are finally ready for the proof of Proposition~\ref{prop:cone_contained}.

\begin{proof}[Proof of Proposition~\ref{prop:cone_contained}] 
	The proof is complicated by the fact that we consider a general potential past $A$ and work in $A^c$. 
	We start with the simpler\footnote{The important feature here is that $\phi_{\calH_{\geq 0}}^1$ is invariant under vertical translations; the same proof also applies to $\phi_{\calH_{\geq 0}}^0$ and the infinite volume measure $\phi$.
	When $\partial \calH_{\geq0}$ is not aligned to the axis of the lattice, the measure is not exactly translationally invariant, but dominates any vertical translate of $\phi_{\calH_{\geq -1}}^1$, which suffices.} case of $A = \calH_{\leq 0}$, which is to say that we work under $\phi_{\calH_{\geq 0}}^1$. We will then explain how to deduce the general case. 

	We start by explaining how to choose $\alpha$. 
	First notice that due to direct \eqref{eq:RSWnc} constructions, we have that for any $\ell \geq 0$
	\begin{align}\label{eq:in_line_conn_LB}
		\phi_{\calH_{\geq 0}}^1 \big[ \La_L \lra \La_L(\ell L ,0) \big]  
		&\ges \exp\big( - C_0 \ell  \big),
	\end{align}
	for some universal constant $C_0 > 0$. 
	Let~$c >0$ be the constant given by~\eqref{eq:L(p)expdecay} and set~$\alpha = 2C_0/c$.

	Fix $k$ and $n$. All constants below are allowed to depend on~$\alpha$, but not on~$p$,~$k$ or~$n$. 

	\medskip 
	
	\noindent {\bf Case 1: measure in the half-space.}
	Consider the cones 
	\begin{align}
	\calY_{\rm out} = \calY_\alpha -(k  L  ,0) \quad\text{ and }\quad	\calY_{\rm in} = \calY_{\alpha/2} -(k  L/2  ,0);
	\end{align}
	see Figure~\ref{fig:cone_contained}, left diagram. 
	To bound the probability in~\eqref{eq:cone_contained}, we will distinguish two scenarios:	
	\begin{itemize}
		\item[(1.a)] either~$\La_{  L  }$ is connected to~$\calH_{\geq n}$ inside~$\calY_{\rm in}$, but is also connected to~$\calY_{\rm out}^c$,
		\item[(1.b)] or~$\La_{  L  }$ is connected to~$\calH_{\geq n}$, but not inside in~$\calY_{\rm in}$.
	\end{itemize}
	Write~$S_a$ and~$S_b$ for the two events above. We will bound their probabilities under~$\phi_{\calH_{\geq 0}}^1[.| \La_{  L  } \lra \calH_{\geq n}]$ separately. 	
	\medskip 
		
	\noindent {\bf(1.a)} To bound the probability of~$S_a$ we condition on the configuration inside of~$\calY_{\rm in}$.  
	If~$S_a$ occurs, then there exists a point~$x \in \partial \calY_{\rm out}$ that is connected to~$\calY_{\rm in}$.
	Applying the same argument of summation by batches as used for~\eqref{eq:cone_to_cone} we conclude that
	\begin{align}\label{eq:S_a_omega_0}
		\phi_{\calH_{\geq 0}}^1\big[S_a \,\big|\, \omega = \omega_0 \text{ on~$\calY_{\rm in}$}\big] 
		\leq C \e^{-c k}
	\end{align}
	for universal constants $c,C>0$ and any $\omega_0$ on $\calY_{\rm in}$ containing a connection between $\La_L$ and~$\calH_{\geq n}$. 
	Now, 
	\begin{align*}
	\phi_{\calH_{\geq 0}}^1[S_a \,|\, \La_{  L  } \lra \calH_{\geq n}]
	= \sum_{\omega_0} \phi_{\calH_{\geq 0}}^1\big[S_a \,\big|\, \omega = \omega_0 \text{ on~$\calY_{\rm in}$}\big]\,  \phi_{\calH_{\geq 0}}^1[\omega =\omega_0  \text{ on~$\calY_{\rm in}$}\,|\, \La_{  L  } \lra \calH_{\geq n}],
	\end{align*}
	where the sum is over all $\omega_0$ containing a connection between $\La_L$ and~$\calH_{\geq n}$. 
	Injecting~\eqref{eq:S_a_omega_0} into the above we conclude that 
		\begin{align}\label{eq:S_a}
		\phi_{\calH_{\geq 0}}^1\big[S_a \,\big|\, \La_{  L  } \lra \calH_{\geq n}\big] 
		\leq C \e^{-c k}.
	\end{align}
		
	\noindent {\bf(1.b)}
	To bound the probability of~$S_b$ under the conditional measure, we will directly compare~$\phi_{\calH_{\geq 0}}^1[S_b]$ to~$\phi_{\calH_{\geq 0}}^1[ \La_{  L  } \lra \calH_{\geq n}]$.
	
	For $x \in L\cdot(\mathbb N \times \mathbb Z)$, write $\calY(x)$ for the horizontal translate of $\calY_{\alpha/2}$ whose boundary passes through $x$. 
	Let $\calB_x$ be the event that $\La_L(x)$ is connected to both~$\La_{  L  }$ and~$\calH_{\geq n}$ inside~$\calY(x)$,
	but that these two connections are produced by disjoint clusters of $\omega \cap \calY(x) \cap \La_{L}(x)^c$. 
	By considering the right-most translate of $\calY_{\alpha/2}$ that contains a connection between $\La_{  L  }$ and~$\calH_{\geq n}$, we conclude that if $S_b$ occurs, then there exists at least one $x \in L \cdot (\mathbb N \times \mathbb Z) \setminus \calY_{\rm in}$ for which $\calB_x$ occurs. 
	See Figure~\ref{fig:cone_contained}, middle diagram.
	
	\begin{figure}
	\begin{center}
	\hspace{-0.1\textwidth}	
	\includegraphics[width = 0.3\textwidth, page= 2]{cone_contained.pdf}
	\includegraphics[width = 0.3\textwidth, page= 1]{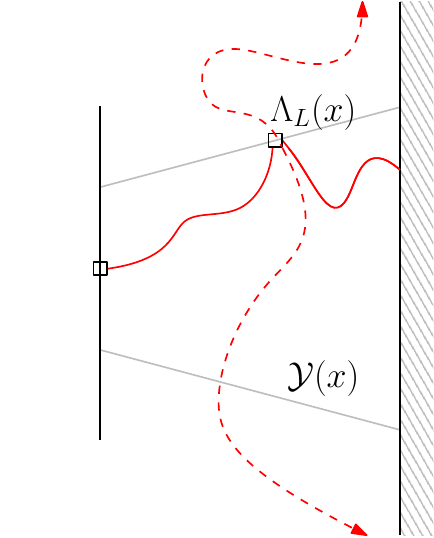}\hspace{0.06\textwidth}	
	\includegraphics[width = 0.3\textwidth, page= 3]{cone_contained.pdf}
	\caption{{\em Left:} The occurence of $S_a$ conditionally on $\La_{  L  } \lra \calH_{\geq n}$ induces an ``unforced'' connection 
	between $\calY_{\rm in}$ and $\calY_{\rm out}^c$, which has an exponential cost in $k$. 
	{\em Middle:} To bound $S_b$, consider the right-most translate of $\calY_{\rm in}$ that contains a connection between $\La_L$ and $\calH_{\geq n}$.
	Its boundary intersects a ``pivotal'' box $\La_L(x)$, that is a box that is connected to $\La_L$ and $\calH_{\geq n}$ inside $\calY(x)$ by distinct clusters.
	{\em Right:} When working with an arbitrary potential past $A$,  
	the connection between $\La_L$ and $\calH_{\geq n}$ may use the left half-plane. 
	Nevertheless, there exists a vertical translate of $\La_L$ that is connected to $\calH_{\geq n}$ inside $\calH_{\geq 0}$.
	We will consider the closest such translate to $0$.}
	\label{fig:cone_contained}
	\end{center}
	\end{figure}
	
	We will bound the probability of each event $\calB_x$ with $x$ as above. 
	Fix one such $x =(x_1,x_2)$ and assume $x_2 \geq 0$; the case $x_2 \leq 0$ is identical. 
	Note that $\calB_x$ is measurable in terms of the configuration inside $\calY(x)$. 
	By the same argument as in the proof of Lemma~\ref{lem:cone_to_half-plane_mixing}, 
	conditionally on $\calB_x$ (or more generally on any event in $\calY(x)$), $\La_{L}(x)$ is connected to infinity by a dual path in $\calY(x)^c$ 
	with positive probability. 
	By performing an additional surgery inside $\La_{L}(x)$ using~\eqref{eq:RSWnc} and doing a similar argument in the lower half-plane (after exploring all the clusters of $\Lambda_L(x)$ in $\calY(x)$), 
	we conclude that
	\begin{align}
		\phi_{\calH_{\geq 0}}^1 \big[\calB_x\big] 
		&\les \phi_{\calH_{\geq 0}}^1 \big[\La_{L}(x) \text{ connected in $\calY(x)$ to $\La_L$ and to $\calH_{\geq n}$ by disjoint clusters}\big]\nonumber\\
		&\les \phi_{\calH_{\geq 0}}^1 \big[\La_{L}(x) \xlra{\calY(x)}\La_L\big] \, \phi_{\calH_{\geq 0}}^1 \big[\La_{L}(x) \xlra{\calY(x)} \calH_{\geq n}\big].
		\label{eq:calB_dec}
	\end{align}
	The second inequality is due to~\eqref{eq:mon} and requires conditioning on the cluster producing one of the connections
	and bounding the probability of the other connection. 
	
	We now bound each term in the right-hand side of~\eqref{eq:calB_dec}.
	First, by vertical translation,
	\begin{align}
		\phi_{\calH_{\geq 0}}^1 \big[\La_{L}(x) \xlra{\calY(x)} \calH_{\geq n}\big] 
		\leq \phi_{\calH_{\geq 0}}^1 \big[\La_{L}(x_1,0) \lra \calH_{\geq n}\big].
	\end{align}
	Moreover,  
	\begin{align}
		\phi_{\calH_{\geq 0}}^1 \big[  \La_{L}\xlra{\calY(x)} \La_{L}(x) \big] 
		&\les \phi_{\calH_{\geq 0}}^0 \big[  \La_{L}\xlra{\calY(x)} \La_{L}(x) \big] \\
		& \les  \, \exp\big( - c(x_1+x_2)/L  \big)  \\
		& \leq \, \exp\big( - C_0 x_1/L -  \tfrac{c}2(x_1+x_2)/L  \big)\\
		& \les \, \phi_{\calH_{\geq 0}}^1\big[ 0 \lra \La_L(x_1,0) \big]  \exp\big( - \tfrac{c}2\|x\|/L  \big),
	\end{align}
	where $c$ is a universal constant and $C_0$ is given by~\eqref{eq:in_line_conn_LB}.
	The first inequality comes from the mixing property~\eqref{eq:cone_to_half-plane_mixing}, 
	the second one from~\eqref{eq:L(p)xi}, and the third one is due to the choice of $\alpha$ and the fact that $x_2 \geq \alpha x_1$.
	
	Inserting the last two displays  into~\eqref{eq:calB_dec} we conclude that 
	\begin{align}
		\phi_{\calH_{\geq 0}}^1 [\calB_x]	
		&\les  \exp\big( -\tfrac{c}2\|x\|/L  \big) 
		 \phi_{\calH_{\geq 0}}^1 [  \La_{L} \lra \La_L(x_1,0)]
		\, \phi_{\calH_{\geq 0}}^1 [\La_{L} (x_1,0)\lra\calH_{\geq n}]
		\\
		&\les  \exp\big( - \tfrac{c}2\|x\|/L  \big) 
		\phi_{\calH_{\geq 0}}^1 [\La_{L} \lra\calH_{\geq n}],
		\label{eq:SlaHuv}
	\end{align}
	The second inequality is due to~\eqref{eq:FKG} and an application of~\eqref{eq:RSWnc} to glue the two connections. 
	Summing the above over all $x \in L \cdot (\mathbb N \times \mathbb Z) \setminus \calY_{\rm in}$, we find that 
	\begin{align}\label{eq:S_b1}
		\phi_{\calH_{\geq 0}}^1  [S_b ] 
		\les  \e^{-c k} \phi_{\calH_{\geq 0}}^1 [\La_{L} \lra\calH_{\geq n}],
	\end{align}
	for a potentially modified value of $c$.
	
	Combining points (a) and (b) we conclude that
	\begin{align}\label{eq:cone_contained1}
		\phi_{\calH_{\geq 0}}^1 \big[\La_{  L  } \lra (\calY_\alpha -(k  L  ,0))^c \,\big|\, \La_{  L  } \lra \calH_{\geq n} \big]
	\leq C \e^{-c k},
	\end{align}
	for universal constants $c,C>0$. 
	The constant $C$ may be removed by modifying $c$. 
	
	A particular consequence of the above and Lemma~\ref{lem:cone_to_half-plane_mixing} is that 
	\begin{align}
		\phi_{\calH_{\geq 0}}^1 [\La_{L} \lra\calH_{\geq n}] 
		&\les 	\phi_{\calH_{\geq 0}}^1 [\La_{L} \xlra{\calY_{\alpha}}\calH_{\geq n}] \\
		&\les		\phi_{\calH_{\geq 0}}^0 [\La_{L} \xlra{\calY_{\alpha}}\calH_{\geq n}] 
		\leq	\phi [\La_{L} \lra\calH_{\geq n} \,|\, {\rm Past}(A)],
		\label{eq:DvH}
	\end{align}
	for any potential past $A$. This will be useful for Case 2, below. 
	\medskip

	\noindent{\bf Case 2: general potential pasts $A$.} 
	We are now ready to  prove the general statement. 
	Fix a potential past $A$. Write $D = (A \cup \partial_{\leq 0}A)^c$ and $\xi$ for the boundary conditions induced on $D$ by the conditioning ${\rm Past}(A)$, so that 
	$\phi [\cdot|_D \,|\, {\rm Past}(A)]= \phi_D^\xi$.
	
	For a point $x \in  \{0\} \times L\mathbb Z$, write $\widetilde \sfC_{x}$ for the cluster of $\omega\cap \calH_{\geq 0}$ that connects $\La_L(x)$ to $\calH_{\geq n}$
	 --- if several such clusters exists, let $\widetilde \sfC_{x}$ denote their union and if no such cluster exists set  $\widetilde \sfC_{x} = \emptyset$.
	 	
	Let $\omega$ be a configuration in which $\La_L$ is connected  to $\calH_{\geq n}$.
	Write $X(\omega)$ be the point $x \in \{0\} \times L\mathbb Z$ of minimal norm for which $\widetilde \sfC_{x} \neq\emptyset$ --- when two such points exist, choose one arbitrarily. 
	
	The fact that $\omega$ contains a connection between $\La_L$ and $\calH_{\geq n}$ guarantees that 
	 $X(\omega)$ is well defined and that 
	 $\La_L$ is connected to $\La_{\|X(\omega)\| - L}^c$ in $\widetilde \sfC_{X(\omega)}^c$; see Figure~\ref{fig:cone_contained}, right diagram.
	 It may be noted that we do not claim $\widetilde \sfC_{X(\omega)}$ to be connected to $\La_L$; it may be that $\La_L$ is connected to $\calH_{\geq n}$ via a different cluster. 
	
	As a consequence, if $x$ is a possible realisation of $X(\omega)$ and $\sfC$ is a possible realisation of $\widetilde \sfC_{X(\omega)}$, we have that 
	\begin{align}
		\phi_{D}^\xi \big[ X(\omega) = x ,\, \widetilde \sfC_{x} = \sfC \text{ and } \La_L \lra \calH_{\geq n}\big]  
		&\leq 
		\phi_{D}^\xi \big[\widetilde \sfC_{x} = \sfC]  \,\phi_{D}^\xi \big[\La_L \xlra{\sfC^c} \partial \La_{\|x\| - L} \,\big|\,  \widetilde \sfC_{x} = \sfC  \big] \\
		&\leq 
		\phi_{D}^\xi \big[\widetilde \sfC_{x} = \sfC\big]  \,\phi_{D \cap \La_{\|x\| - L}}^\zeta \big[\La_L \lra \partial \La_{\|x\| - L}\big]\\
		&\les 
		\phi_{D}^\xi \big[\widetilde \sfC_{x} = \sfC\big]  \,e^{-c\|x\|}. \label{eq:usefulstep}
	\end{align}
	Above, $\zeta$ are the boundary conditions on $D \cap \La_{\|x\| - L}$ that are wired on $\partial \La_{\|x\| - L}$ and identical to $\xi$ elsewhere. 
	Indeed, the conditioning on $ \widetilde \sfC_{x} = \sfC$ produces a mixture of free and wired boundary conditions on $\sfC^c$, 
	together with the additional conditioning that no other connections between boxes $\La_L(x')$ and  $\calH_{\geq n}$ occur in $\calH_{\geq 0}$ 
	for $x' \in  \{0\} \times L\mathbb Z$ with $\|x'\| \leq \|x\|$. 
	The latter is a decreasing conditioning and the boundary conditions induced by the conditioning on $\sfC$ only have wired parts outside of $ \La_{\|x\| - L}$ --- this is due to the choice of $X(\omega)$. 
	The last inequality is due to~\eqref{eq:L(p)expdecay}. 
	The bound~\eqref{eq:usefulstep} will be used below.

	Fix some small constant $c_0 >0$ that depends only on $\alpha$; we will see below how to choose it. 
	If $\omega$ is a configuration containing connections between $\La_L$ and both $\calH_{\geq n}$ and the outside of $\calY_{\rm out} = \calY_{\alpha} -(k  L,0)$, 
	then at least one of the cases below occurs.
	\begin{itemize}
	\item[(2.a)] $\|X(\omega)\| \geq\alpha c_0  k L$;
	\item[(2.b)] $\|X(\omega)\| \leq \alpha c_0  k L$ but $\widetilde \sfC_{X(\omega)}$ is not contained in $\calY_{\alpha} -(2c_0 k  L  ,0)$;
	\item[(2.c)] $\widetilde \sfC_{X(\omega)}$ is contained in $\calY_{\alpha} -(2c_0 k  L  ,0)$, but $\La_L$ is connected to $\calY_{\rm out}^c$.
	\end{itemize} 
	We will bound the probability of each of the events above separately. \medskip 
	
	To bound the probability of (2.a), observe that, if $x$ is a potential realisation of  $X(\omega)$, summing~\eqref{eq:usefulstep} over $\sfC$ implies that
	\begin{align}
		\phi_{D}^\xi [X(\omega) = x \text{ and } \La_L \lra \calH_{\geq n}]
		&\les \phi_{D}^\xi [\La_L(x) \xlra{\calH_{\geq 0}} \calH_{\geq n}]\, \e^{-c\|x\|}\\
		&\leq \phi_{\calH_{\geq 0}}^1 [\La_L(x) \lra \calH_{\geq n}]\, \e^{-c\|x\|}.
	\end{align}
	Summing over $x$ with $\|x\| \geq\alpha c_0  k L$ we conclude that 
	\begin{align}\label{eq:2.a}
		\phi_{D}^\xi \big[\|X(\omega)\| \geq \alpha c_0 k L \text{ and } \La_L \lra \calH_{\geq n}\big]
		&\les \e^{-c \,c_0 k} \, \phi_{\calH_{\geq 0}}^1 [\La_L \lra \calH_{\geq n}] ,
	\end{align}
	for some universal constant $c>0$. 
	 \smallskip 
	
	The probability of (2.b) may be bounded in a similar way. 
	For $x$ a potential realisation of $X(\omega)$, summing~\eqref{eq:usefulstep} over realisations of $\widetilde \sfC_x$ as in (2.b), we find 
	\begin{align}
		\phi_{D}^\xi \big[  X(\omega) = x, \widetilde \sfC_{x}& \text{ intersects $(\calY_{\alpha} -(2 c_0 k  L  ,0))^c$ and }\La_L \lra \calH_{\geq n}\big]\\
		&\les \phi_{D}^\xi \big[\widetilde \sfC_{x} \text{ intersects  $(\calY_{\alpha} -(2c_0 k  L,0))^c$}\big]\, \e^{-c\|x\|}\\
		&\leq \phi_{\calH_{\geq 0}}^1 \big[\La_L \lra \calH_{\geq n} \text{ and } \La_L \lra (\calY_{\alpha} -(c_0 k  L  ,0))^c\big]\,e^{-c\|x\|}\\
		&\les \e^{-c \,c_0k + c\|x\|} \phi_{\calH_{\geq 0}}^1 [\La_L \lra \calH_{\geq n}],
	\end{align}
	with the second inequality due to~\eqref{eq:mon} --- note that the event considered is increasing --- and the last one given by~\eqref{eq:cone_contained} applied to $ \phi_{\calH_{\geq 0}}^1$, which we already proved. 
	Summing over the possible values of $X(\omega)$ we find that 
	\begin{align}
		\phi_{D}^\xi & \big[\|X(\omega)\| \leq \alpha c_0  k L,\,  \widetilde \sfC_{X(\omega)} \not\subset (\calY_{\alpha} -(2c_0kL  ,0)) \text{ and }\La_L \lra \calH_{\geq n} \big]\\
		&\les \e^{-c \,c_0 k}  \phi_{\calH_{\geq 0}}^1 [\La_L \lra \calH_{\geq n}].
		\label{eq:2.b}	\medskip 
	\end{align}
	
	Finally, to bound the probability of (2.c), 
	notice that for this event to occur, there need to exist a connection between $\La_{2\alpha c_0  k L}$ and  $\calY_{\rm out}^c$, 
	outside of $\widetilde \sfC_{X(\omega)}$.
	Using the same argument of conditioning on  $\widetilde \sfC_{X(\omega)}$ as in~\eqref{eq:usefulstep},
	\begin{align}
	\phi_{D}^\xi \big[\La_L \lra \calH_{\geq n},& \, \widetilde \sfC_{X(\omega)}\subset \calY_{\alpha} -(2c_0k  L   ,0) \text{ and } \La_L \lra \calY_{\rm out}^c\big]\\
	&\les \sum_x  \phi_{D}^\xi [\La_L(x) \xlra{\calH_{\geq 0}} \calH_{\geq n}]  
		\,\phi_{D \cap \La_{2\alpha c_0 k L}^c}^{\zeta} \big[\La_{2 \alpha c_0 k L} \lra \calY_{\rm out}^c\big]	\\
	&\les c_0 k \, \e^{Cc_0 k}
	 	\phi_{\calH_{\geq 0}}^1 [\La_L \lra \calH_{\geq n}]
	 	\,\phi_{D}^{\xi} \big[\La_{2 \alpha c_0 k L} \lra \calY_{\rm out}^c\big] \\
		& \les \e^{-c' k}  \phi_{\calH_{\geq 0}}^1 [\La_L \lra \calH_{\geq n}].
	\label{eq:2.c}
	\end{align}
	where the sum is over all $x \in \{0\}\times L \mathbb Z$ with $\|x\| \leq 2 \alpha c_0 k L$ 
	and $\zeta$ are the boundary conditions on $D \cap  \La_{2\alpha c_0 k L}^c$ that are 
	wired on $\partial  \La_{2\alpha c_0 k L}$ and identical to  $\xi$ everywhere else.
	The second inequality  uses~\eqref{eq:mon} and the finite energy property, with $C$ some universal constant. 
	The last inequality is due to~\eqref{eq:L(p)expdecay} and is valid for $c_0$ below some universal constant, and may be ensured by the choice of $c_0$; $c' >0$ is a universal constant.

	Summing~\eqref{eq:2.a},~\eqref{eq:2.b} and~\eqref{eq:2.c}, we obtain
	\begin{align}
	\phi_{D}^\xi \big[\La_{  L  } \lra \calY_{\rm out}^c \text{ and } \La_{  L  } \lra \calH_{\geq n} \big]
	&\les \e^{-c''k} \phi_{\calH_{\geq 0}}^1 [\La_L \lra \calH_{\geq n}] \\
	&\les \e^{-c''k} \phi_{D}^\xi [\La_L \lra \calH_{\geq n}].
	\end{align}
	for some universal constant $c''>0$. 
	 We used~\eqref{eq:DvH} in the last line.
	This proves~\eqref{eq:cone_contained}. 
\end{proof}

\begin{center}
	{\bf Henceforth $\alpha \geq 4$ is fixed to satisfy Proposition~\ref{prop:cone_contained} and we write $\calY := \calY_\alpha$.}
\end{center}

\subsection{The number of active segments is subcritical}\label{subsection: The number of active boxes is subcritical}

For integers~$t \geq 0, k \in \Z$, define~$\mathscr{L}_{t,k}  := \{tL\} \times [ kL, (k+1)L)$, 
so that the hyperplane~$\partial \calH_{\leq t}$ is the disjoint union of the line segments~$\mathscr{L}_{t,k}$ when~$k$ runs over~$\Z$.
Also, write~$x_{t,k} = (tL,(k+1/2)L)$ for the midpoint of~$\mathscr{L}_{t,k}$.

Recall that in Theorem~\ref{thm:killed_renewal_structure} we consider a configuration sampled according to  $\phi$ and $\sfC_{\leq t}$ is the cluster of $0$ in $\calH_{\leq t}$.
A line segment~$\mathscr{L}_{t,k}$ is said to be \emph{active at time}~$t$ if~${\sfC}_{\leq t} \cap \mathscr{L}_{t,k} \neq \emptyset$. 
Write~$N_t$ for the number of active segments at time~$t$. 
A time~$t$ for which~$N_t = 1$ will be called a {\em pre-renewal time}. 
The goal of this section is the following.

\begin{Prop}[Density of pre-renewal times]\label{prop:pre-renewal_density}
	There exists~$c> 0$ such that for any~$n,t,r \geq 1$,
	\begin{align}
	    \phi\big[N_{s} > 1, \,\,\forall s \in \{ t+1, \dots, t+r \} \,\big|\, \sfC_{\leq t},\, N_t = 1, \,X_n \neq \dagger\big] &\leq \exp(-cr)\quad  \text{ and} \label{eq:pre-renewal_density1}\\
	    \phi\big[N_{s} > 1, \,\,\forall s \in \{1, \dots, r \} \,\big|\, \,X_n \neq \dagger \big] &\leq \exp(-cr). \label{eq:pre-renewal_density2}
	\end{align}
\end{Prop}

This tells us that pre-renewal times arrive quickly, even when conditioning on a survival event far in the future. 
Proposition~\ref{prop:pre-renewal_density} will follow from the result below. 

\begin{Prop}\label{prop:equation_martingale}
	There exist constants $C \geq 1$,~$\mu < 1$ and~$K > 0$ such that for any~$t \geq 0$ and~$n\geq t+C$, 
	\begin{equation}\label{eq:equation_sur-martingale}
		\phi[N_{t+C} \,|\, \sfC_{\leq t},\, X_n \neq \dagger  ] \leq \mu N_t + K. 
	\end{equation}
\end{Prop}

The rest of this section is dedicated to the proofs of the two results above. 
In spirit, the results of that section are very close to limit theorems for the distribution of the number of children of a subcritical branching process conditioned to survive for a long time (see~\cite{athreya2004branching}). 
Our proof is a bit more complex due to the lack of control on the offspring distribution and most importantly due to the positive correlations between the different activated segments. 
As mentioned, Proposition~\ref{prop:equation_martingale} is the key step, and we start with it. First we explain how to choose the constant~$C$. 

\begin{Lemma}\label{lem:fixC}
    For any $\mu > 0$, there exists some constant~$C>0$ such that, for any potential past $A$ 
    \begin{align}\label{eq:fixC}
		\phi\big[\# j \text{ such that }\mathscr{L}_{0,0} \leftrightarrow \mathscr{L}_{C,j}\,\big|\, {\rm Past}(A)\big] \leq \mu.
	\end{align}
\end{Lemma}

\begin{proof}[Proof of Lemma~\ref{lem:fixC}]
Fix $\mu>0$,  $A$ a potential past and let $C$ be some positive integer. Then, for $j \in \mathbb Z$, 
for $\mathscr{L}_{0,0}$ to be connected to $\mathscr{L}_{C,j}$, $\La_L$ needs to be connected to $\partial \La_{L \cdot \max\{C,|j|-1\}}$. 
Thus, by~\eqref{eq:L(p)expdecay}, 
\begin{align}
\phi\big[\mathscr{L}_{0,0} \leftrightarrow \mathscr{L}_{C,j}\,\big|\, {\rm Past}(A)\big] \les \exp(-c \max\{C,|j|-1\}). 
\end{align}
Summing the above and taking $C$ large enough implies~\eqref{eq:fixC}.
\end{proof}

\begin{proof}[Proof of Proposition~\ref{prop:equation_martingale}]
Fix the parameters~$n$ and~$t$ as in the statement. 
Also fix some realisation of~$\sfC_{\leq t}$. We will always work conditionally on this realisation of the ``past cluster'', and write~$\phi_{\sfC_{\leq t}}$ for this conditional measure. It is the translate of a measure of the type $\phi[.\,|\, {\rm Past}(A)]$, and Lemma~\ref{lem:fixC} applies to it. 
All notions of connections and clusters below refer to the configuration in~$\sfC_{\leq t}^c$ only. 
All constants and equivalences below are uniform in the choices of~$p$,~$\vec w$,~$n$,~$t$ and~$\sfC_{\leq t}$. 

Fix~$C$ given by Lemma~\ref{lem:fixC} with the choice of an arbitrary $\mu < 1/2$. 
Write~$\bfj$ for the index of the top-most active box~$\mathscr{L}_{t, j}$ such that~$\mathscr{L}_{t, j} \cap \sfC_{\leq t}$ is connected to~$\calH_{\geq n}$. 
If no such connection exists, write~$\bfj = \emptyset$. 
Also write~$\bfK$ for the minimal value~$K \geq 0$ such that the cluster of~$\mathscr{L}_{t, \bfj}$ is contained in~$\calY + x_{t,\bfj} - (K  L  ,0)$ (recall that~$x_{t,j}$ denotes the mid-point of~$\mathscr{L}_{t,j}$).

Our goal is to bound
\begin{align}
	\phi_{\sfC_{\leq t}}[ N_{t+C}  \,|\, \bfj\neq \emptyset  ] 
	= \sum_j \phi_{\sfC_{\leq t}}[\bfj = j  \,|\, \bfj\neq \emptyset] \phi_{\sfC_{\leq t}}[N_{t+C}  \,|\, \bfj=j], 
\end{align}
and we will do so by bounding each of the terms~$\phi_{\sfC_{\leq t}}[N_{t+C}  \,|\, \bfj=j]$ individually. 

We first argue that
\begin{align}\label{eq:exp_dec_K}
	\phi_{\sfC_{\leq t}}[\bfK > k \,|\, \bfj = j] \leq \e^{-c_0 k}
\end{align}
for some constant~$c_0 > 0$ and all~$k\geq 3$. 

We say~$\mathscr{L}_{t, j}$ is a top-most seed if 
$\mathscr{L}_{t, j} \cap \sfC_{\leq t}$ is connected to~$\partial \La_{2  L  }(x_{t, j})$ but not to~$\mathscr{L}_{t,j + 1}$ in~$\sfC_{\leq t}^c \cap \La_{2  L  }(x_{t,j})$. We also write~$s_{t} = \phi_{\calH_{\geq 0}}^1[\mathscr{L}_{0,0} \lra \calH_{\geq t}]$.
We start off by estimating the probability of~$\bfj = j$. 

\begin{Lemma}[Gluing lemma]\label{lem:seed}
	Uniformly in~$n$, $t$ and~$r$ and past clusters $\sfC_{\leq t}$ satisfying $\phi[\sfC_{\leq t} \lra \calH_{n}] < \frac 1 2$,
	\begin{align}\label{eq:seed}
		\phi_{\sfC_{\leq t}}[\bfj = j] &\asymp s_{n-t} \,\phi_{\sfC_{\leq t}}[j \text{ is a top-most seed }].		
	\end{align}
\end{Lemma}

We defer the proof of the lemma to later in the section and finish that of Proposition~\ref{prop:equation_martingale}. 
Due to the condition on $\sfC_{\leq t}$ appearing in Lemma~\ref{lem:seed}, we will distinguish two cases depending on $\sfC_{\leq t}$.
\medskip 

\noindent {\bf Case (A): $\sfC_{\leq t}$ is such that $\phi_{\sfC_{\leq t}}[\sfC_{\leq t} \lra \calH_{n}] < \frac 1 2$.}
Due to~\eqref{eq:seed} we have
\begin{align*}
	\phi_{\sfC_{\leq t}}[&\bfK > k \,|\, \bfj = j] \les \frac{\phi_{\sfC_{\leq t}}[\bfK \geq k \text{ and } \bfj = j]} { s_{n-t} \,\phi_{\sfC_{\leq t}}[j \text{ is top-most seed}]}\\[.3em]
	&\leq \frac{\phi_{\sfC_{\leq t}}\big[j \text{ is top-most seed},\, \La_{2  L  }(x_{t,j}) \xleftrightarrow{} \calH_{\geq n} \text{ and }\La_{2  L  }(x_{t,j}) \leftrightarrow (\calY + x_{t,j} - (k  L  ,0))^c \big]}
	{ s_{n-t} \,\phi_{\sfC_{\leq t}}[j \text{ is top-most seed}]}\\[.3em]
	&\leq \frac{\phi_{\sfC_{\leq t}}\big[\La_{2  L  }(x_{t,j}) \xleftrightarrow{} \calH_{\geq n} \text{ and }\La_{2  L  }(x_{t,j}) \leftrightarrow (\calY + x_{t,j} - (k  L  ,0))^c \,|\, \omega \equiv 1 \text{ on }\La_{2  L  }(x_{t,j}) \big]}{ s_{n-t} }\\
	&\les \e^{-c_0 k},
\end{align*}
The second inequality is simply an inclusion of events: for $\bfj = j$ to occur, $j$ needs to be a top-most seed and $\La_{2  L  }(x_{t,j})$ should be connected to $\calH_{\geq n}$. The third is a consequence of \eqref{eq:FKG} and the last one is due\footnote{The careful reader might note that Proposition~\ref{prop:cone_contained} does not rigorously apply in this setting, 
because of the additional conditioning on the configuration in $\La_{2  L  }(x_{t,j})$ which does not correspond to a potential past. Nevertheless, a direct application of \eqref{eq:RSWnc} proves that the bound \eqref{eq:cone_contained} still applies in this case, with a potential bounded multiplicative factor. 
}
to Proposition~\ref{prop:cone_contained}.
This proves~\eqref{eq:exp_dec_K}.

Now, for~$j$ fixed, by the almost sure finiteness of the cluster of 0, 
\begin{align}\label{eq:skN}
	\phi_{\sfC_{\leq t}}[ N_{t+C}  \,|\, \bfj =j  ] 
	=\sum_{k \geq 1} 	\phi_{\sfC_{\leq t}}[ N_{t+C}  \,|\, \bfK = k,\, \bfj = j  ] \phi_{\sfC_{\leq t}}[\bfK = k \,|\,\bfj=j]. 
\end{align}
To bound~$\phi_{\sfC_{\leq t}}[ N_{t+C}  \,|\, \bfK = k,\, \bfj = j ]$ explore first the connected component~$\tilde \sfC$ of~$\mathscr{L}_{t,j}$ 
and observe that it is connected to at most~$2\alpha (k + C)$ intervals~$\mathscr{L}_{t + C,\ell}$, since it is contained in~$\calY + x_{t,j} - (k  L  ,0)$. 

Conditionally on~$\tilde \sfC$, all other active intervals~$\mathscr{L}_{t,\ell}$ are connected in~$\tilde \sfC^c$ to a random number of intervals~$\mathscr{L}_{t + C,\ell'}$ with an average at most~$\mu$. Indeed, the conditioning on~$\tilde \sfC$ induces free boundary conditions on~$\tilde \sfC^c$, and by~\eqref{eq:mon} the estimate~\eqref{eq:fixC} applies.

We conclude that~$\phi_{\sfC_{\leq t}}[ N_{t+C}  \,|\, \bfK = k,\, \bfj = j  ]\leq \mu N_t + 2\alpha( k + C)$. Inserting this into~\eqref{eq:skN} and using~\eqref{eq:exp_dec_K} we conclude that 
\begin{align*}
	\phi_{\sfC_{\leq t}}[ N_{t+C}  \,|\, \bfj =j  ] 
	&\leq \mu N_t + 2\alpha   \sum_{k\geq0}   (k+C)   \phi_{\sfC_{\leq t}}[\bfK = k \,|\,\bfj=j]\\
	&\leq \mu N_t + 2\alpha C_0   \sum_{k\geq0} (k+C)  \e^{-c_0 k}
\end{align*}
which produces the desired result with~$K=2\alpha C_0   \sum_k (k+C)  \e^{-c_0 k} <\infty$. 
\medskip 

\noindent {\bf Case (B): $\sfC_{\leq t}$ is such that $\phi[\sfC_{\leq t} \lra \calH_{n}] \geq \frac 1 2$.}
Here we may simply write 
\begin{align*}
\phi_{\sfC_{\leq t}}[N_{t+C} \,|\,X_n \neq \dagger] \leq 2\phi_{\sfC_{\leq t}}[N_{t+C}] \leq 2\mu N_t, 
\end{align*}
with the last bound obtained from summing \eqref{eq:fixC} over all active segments. As $\mu$ was chosen strictly below $1/2$, Proposition~\ref{prop:equation_martingale} follows with $\mu' = 2\mu < 1$. 
\end{proof}

\begin{Rem}
	A minor modification of the proof above also implies the existence of a universal constant $c$ such that, for all $t$ and $n \geq t+C$ 
	\begin{align}\label{rem:going_down_to_1}
		\phi[N_{t+C} = 1 \,|\, \sfC_{\leq t},\, X_n \neq \dagger  ] \geq \e^{-c N_t}.
	\end{align}
	This observation shall be useful later on. 
\end{Rem}

Recall that we still need to prove the ``gluing estimate''~\eqref{eq:seed}. 
This is a relatively standard, but tedious use of the RSW property~\eqref{eq:RSWnc}.

\begin{proof}[Proof of Lemma~\ref{lem:seed}]	
	Fix all quantities as in the statement of the lemma and Proposition~\ref{prop:equation_martingale}.
	All constants appearing below will be independent of these quantities.
	
	The upper bound on~$\phi_{\sfC_{\leq t}}[\bfj = j]$ is immediate, since~$\{\bfj = j\}$ requires~$j$ to be a top-most seed and for~$\La_{4L}(x_{t,j})$ to be connected to~$\calH_{\geq n}$. 
	These two events depend on the inside of~$\La_{2L}(x_{t,j})$ and the outside~$\La_{4L}(x_{t,j})$, respectively. 
	The mixing property proved in~\eqref{eq:DvH} allows one to factorise the probabilities of these two events, up to a universally bounded multiplicative constant. 
	
	We now focus on the lower bound. 
	Fix~$R =   L   / 10$ and write~$x = x_{t,j}$ for the center of the interval~$\mathscr{L}_{t,j}$.
	Write~$\calA$ for the points of the cluster~$\sfC_{\leq t}$ on~$\mathscr{L}_{t,j}$. 
	Also write~$\calB$ for of the points of~$\sfC_{\leq t}$ on~$\mathscr{L}_{t,\ell}$ with~$\ell > j$; these all lie on~$\partial \calH_{\geq t}$, above~$\mathscr{L}_{t,\ell}$. 
	Let~$x_+$ and~$x_-$ be the top-most and bottom-most points of~$\calA$ respectively. 
	Write~$y_-$ for the bottom-most point of~$\calB$. 
	Note that~$x_+$ and~$y_-$ are linked by a dual path of the boundary of~$\sfC_{\leq t}$. 
	
	We start off by exploring certain interfaces starting on the boundary of $\sfC_{\leq t}$ up to certain stopping times; write ${\rm Exp}$ for the set of edges thus explored. 
	If {\em successful} (see below for a definition), the exploration will produce two {\em exposed} arcs, which is to say a primal arc $\chi_1$ connected to $\calA$ and a dual arc $\chi_0$ connected in the dual configuration the free path of $\partial \sfC_{\leq t}$ between $x_+$ and $y_-$, 
	such that both $\chi_1$ and $\chi_0$ may be linked  to the right side of  $\partial \La_{2L}(x)$  by disjoint tubes contained in  $ \La_{2L}(x) \setminus  (\sfC_{\leq t} \cup {\rm Exp})$,  of width~$R/4$ and length at most~$100 R$.
	See the left diagram of Figure~\ref{fig:gluing1} for an example. 
	
	Our procedure will be such that ${\rm Exp}$ is contained within distance $5R$ of $\mathscr{L}_{t,j}$ and 
	\begin{align}\label{eq:exploration_success}
		\phi_{\sfC_{\leq t}} [\text{exploration is successful}] \ges \phi_{\sfC_{\leq t}}[ j \text{ top-most seed}].
	\end{align}
	
	Once this is done, a simple argument based on~\eqref{eq:cone_contained},~\eqref{eq:DvH},~\eqref{eq:RSWnc} and the condition on $\sfC_{\leq t}$
	will show that 
	\begin{align}\label{eq:j=j_if_success}
		\phi_{\sfC_{\leq t}} [\bfj = j\,|\, \text{exploration is successful}] \ges s_{n-t}. 
	\end{align}
	
	The combination of \eqref{eq:exploration_success} and \eqref{eq:j=j_if_success} yields the desired lower bound. 
	Below we focus on proving these two equations. The former is the major difficulty of the proof, and is obtained via a tedious, but almost deterministic construction. 
	For illustration, note that, when $\calA$ contains a segment of length $R$ and  $x_+$ and $y_-$ are at a distance at least $R$ of each other, 
	no exploration is necessary, and \eqref{eq:exploration_success} is trivial. 
	\medskip 
	
	\noindent{\bf Explorations and proof of \eqref{eq:exploration_success}.}	
	The exploration is done in two steps, each defining one of $\chi_0$ and $\chi_1$. 
	\smallskip 
	
	\noindent {\em Step 1:} Write~$\partial_\infty \La_{2R}(x_+)$ for the arc of~$\partial \La_{2R}(x_+)$ separating~$x_+$ from~$\infty$ in~$\sfC_{\leq t}^c$. 
	This contains at least~$\partial \La_{2R}(x_+) \cap \calH_{\geq t}$. Split it into the top and bottom part, which lie above and below the point~$x_+ + (2R,0)$, respectively.

	\begin{figure}
	\begin{center}
	\includegraphics[height = 5cm]{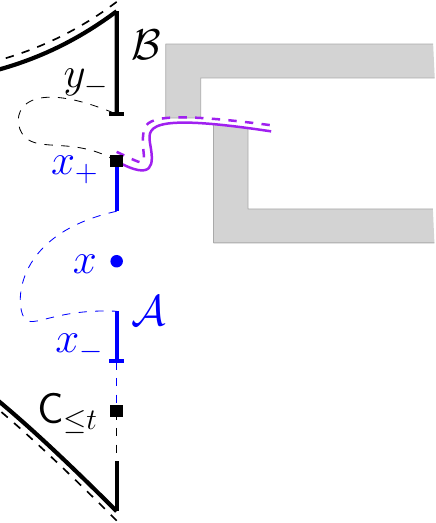}\qquad\qquad\qquad
		\includegraphics[height = 5cm]{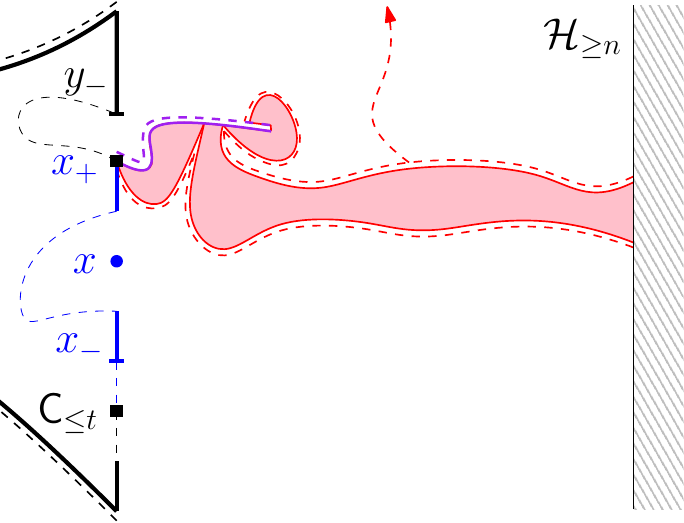}
	\caption{{\em Left:} An example of a successful exploration (purple), with the arcs $\chi_0$ and $\chi_1$ being the top and bottom sides of the interface and the tubes depicted in grey. The solid blue lines form $\calA$, while $\calB$ is the solid black segment above $y_-$. 
	An exploration is needed when either $\calA$ or the dual arc between $x_+$ and $y_-$ are inaccessible. 
	{\em Right:} For a successful exploration, $\chi_1$ may be connected to $\calH_{\geq n}$ by a primal cluster that does not intersect $\calB$ with probability at least $s_{n-t}$. After exploring the cluster $\sfC_{\chi}$ of $\chi_1$ (red) $\calB$ may be separated from $\calH_{\geq n}$ with uniformly positive probability due to our assumption on $\sfC_{\leq t}$. }
	\label{fig:gluing1}
	\end{center}
	\end{figure}

	Write~$\Gamma$ for the top boundary of the cluster of~$\calA$ in~$\sfC_{\leq t}^c$. 
	This is a path that will be indexed by~$[0,1]$ starting from the point~$x_+$. Since the cluster of~$\calA$ is a.s. finite,~$\Gamma$ eventually ends at~$x_-$.
	Below we will explore $\Gamma$ up to a convenient stopping time $\tau$. 
	For now set $\tau$ to be the first time $\Gamma$ hits~$\partial_\infty \La_{2R}(x_+)$; 
	the actual stopping time will be different, but we avoid introducing new notation.

	We distinguish several cases:
	\begin{itemize}
	\item[(1)]~$x_+$ and~$y_-$ are at a distance larger than~$2R$ of each-other. Then update $\tau$ to be equal to $0$; 
	\item[(2)]~$x_+$ and~$y_-$ are at a distance smaller than~$2R$ of each-other, $\Gamma_{[0,\tau]}$ does not connect $\calA$ to $\calB$ and
	$\Gamma_{\tau}$ is in the bottom part of $\partial_\infty \La_{2R}(x_+)$;
	\item[(3)]~$x_+$ and~$y_-$ are at a distance smaller than~$2R$ of each-other, $\Gamma_{[0,\tau]}$ does not connect $\calA$ to $\calB$ and
	$\Gamma_{\tau}$ is in the top part of $\partial_\infty \La_{2R}(x_+)$;
	\item[(4)] $\Gamma_{[0,\tau]}$ connects $\calA$ to $\calB$. 
	\end{itemize}
	In case (4), the exploration is considered unsuccessful. 
	Note however that, when~$j$ is a top-most seed,~$\Gamma$ exits~$\La_{  2L  }(x)$ without connecting~$\calA$ to~$\calB$, 
	and therefore excludes case~$(4)$. 
	Thus,  
	\begin{align}
	\phi_{\sfC_{\leq t}} [\text{cases (1), (2) or (3)}] \geq \phi_{\sfC_{\leq t}}[ j \text{ top-most seed}].
	\end{align}
	
	Next we define $\chi_0$ or $\chi_1$ for each of the cases (1)-(3) differently. 
	In case (1), set~$\chi_0$ to be the part of the boundary of~$\sfC_{\leq t}$ between~$x_+$ and~$y_-$. Also write~$z = x_+$ and~$\tilde \sfC = \sfC_{\leq t}$. 
	
	In the second and third case set~$z = \Gamma_{\tau}$ and $\tilde \sfC = \sfC_{\leq t} \cup \Gamma([0,\tau])$. 
	In the second case, write~$\chi_0$ for the top part of~$\Gamma([0,\tau])$; 
	while in the third case we denote by~$\chi_1$ the bottom part of~$\Gamma([0,\tau])$.

	Notice that in the first two cases,~$\chi_0$ is a dual arc accessible in~$\tilde \sfC^c$,  in the sense defined above.
	In the third case $\chi_1$ is a primal arc accessible in~$\tilde \sfC^c$.
	Regardless, our present exploration up to this point only produces one of the two accessible arcs required. 
	In the second step, we will perform further explorations which will produce the second arc.
	This step depends on which case occurred in Step 1. \smallskip 

	\noindent{\em Step 2 in cases (1) and (2):}
	Define the region~$\La'$ as~$\La' = \La_R(x_+)$ in case (1) and, in case (2), 
	as the union of~$\La_{R}(u)$ for~$u$ in the bottom part of~$\partial_\infty \La_{2R}(x_+)$. 
	Write~$\partial_\infty \La'$ for the arc of~$\partial \La'$ separating~$z$ from~$\infty$ in~$\tilde\sfC^c$. 
	Split~$\partial_\infty \La'$ at~$x_+ + (R,0)$ in case (1) and at $x_+ + (3R,0)$ in case (2) into its top and bottom sections -- see Figure~\ref{fig:gluing2}.

	\begin{figure}
	\begin{center}
	\includegraphics[height = 5.5cm, page=1]{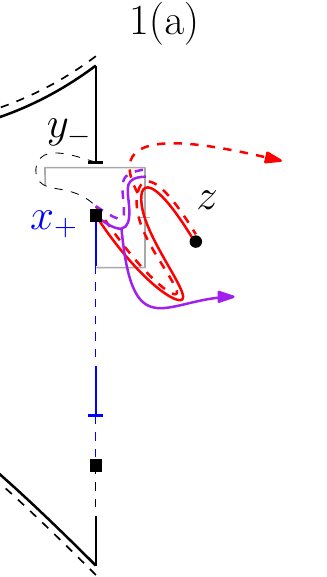}\hspace{.2cm}
	\includegraphics[height = 5.5cm, page=2]{gluing2.pdf}\hspace{.4cm}
	\includegraphics[height = 5.5cm, page=3]{gluing2.pdf}\hspace{.4cm}
	\includegraphics[height = 5.5cm, page=4]{gluing2.pdf}
	\caption{The different scenarios that may occur when exploring $\Gamma$, and potentially $\tilde\Gamma$.
	The exploration up to $\tau$ and the resulting connection is drawn in red, the subsequent construction is in purple.
	The red and purple arrow refer connections to the paths from  $\Lambda_{R}(x + (5R,\pm R))$ to $\calH_{\geq n}$ and $\infty$, which occur with probability comparable to $s_{n-t}$.  
	From left to right: scenarios 1(a), 2(a), 2(b) and 3.
	Note that in scenario 2(b), the exploration path $\Gamma$ never exposes an accessible primal arc. That is why we do not explore $\Gamma$ between $\tau$ and $\tau'$ (grey) but rather explore $\tilde \Gamma$ up to $\tau''$ (purple). }
	\label{fig:gluing2}
	\end{center}
	\end{figure}

	Consider now~$\tau'$ as the first time after~$\tau$ when~$\Gamma$ touches~$\partial_\infty \La'$.
	We distinguish two sub-cases
	\begin{itemize}
	\item[(a)] 	$\Gamma_{\tau'}$ is in the top part of~$\partial_\infty \La'$ or
	\item[(b)] 	$\Gamma_{\tau'}$ is in the bottom part of~$\partial_\infty \La'$.
	\end{itemize}
	In case~$(a)$, write~$\chi_1$ for the wired arc of~$\Gamma_{[0,\tau']}$. 
	Then~$\chi_1$ is a primal arc and~$\chi_0$ is a dual arc, both accessible in~$(\tilde \sfC \cup \Gamma_{[0,\tau']})^c$.

	In case~$(b)$ a more complicated construction is needed. 
	Write~$\tilde \Gamma$ for the exploration path starting at~$x_+$, 
	leaving vertices connected to~$\calA$ in~$\sfC_{\leq t}^c$ on its left 
	(including those of~$\calA$, but not other vertices of~$\sfC_{\leq t}$), 
	and all other vertices on the right. See Figure~\ref{fig:gluing2} for an illustration. 
	
	Write~$\tau''$ for the first time this path touches~$\partial_\infty \La'$; 
	Figure~\ref{fig:gluing2} shows why such a time exists when~$j$ is a top-most seed, 
	and why~$\tilde \Gamma_{\tau''}$ is ``below''~$\Gamma_{\tau'}$ on~$\partial_\infty \La'$. 
	It follows that the wired arc~$\chi_1$ of~$\tilde \Gamma_{[0,\tau'']}$ is accessible in~$(\tilde \sfC \cup \tilde \Gamma_{[0,\tau'']})^c$.
	
	Notice that we do not claim that~$\chi_1$ is accessible when~$\Gamma_{[\tau,\tau']}$ has been explored.
	Thus, one should make a choice at the stopping time~$\tau$ whether to continue exploring~$\Gamma$ up to time~$\tau'$ 
	or whether to explore~$\tilde \Gamma$ up to time~$\tau''$. 
	
	Notice however that either 
	\begin{align}
	&\phi\big[\text{case (a) }\big| \text{ case (1) or (2),~$\Gamma_{[0,\tau]}$ and~$j$ top-most seed}\big]  \geq 1/2 \text{ or }\\
	&\phi\big[\text{case (b) }\big| \text{ case (1) or (2),~$\Gamma_{[0,\tau]}$ and~$j$ top-most seed}\big] \geq 1/2.
	\end{align}

	If (a) is more probable, then explore~$\Gamma$ up to time~$\tau'$. 
	If the resulting exploration corresponds to case (a), define~$\chi_1$ as above and set ${\rm Exp} = \Gamma_{[0,\tau']}$.
	Otherwise say the exploration is unsuccessful.  

	If (b) is more probable, explore~$\tilde\Gamma$ up to time~$\tau''$. 
	If the wired part of~$\tilde \Gamma_{[0,\tau'']}$ is indeed accessible, denote it by~$\chi_1$ and set 
	${\rm Exp} = \Gamma_{[0,\tau]}\cup \tilde \Gamma_{[0,\tau'']}$.
	Otherwise say the exploration is unsuccessful.

	With this construction, if the exploration is successful, the arcs $\chi_0$ and $\chi_1$ are indeed both accessible in $(\sfC_{\leq t} \cup {\rm Exp})^c$. 
	Furthermore, our analysis proves that 
	\begin{align}
		\phi\big[ \text{exploration is successful } \big| \text{ case (1) or (2) and~$j$ top-most seed}\big] \geq 1/2,
	\end{align}
	which implies \eqref{eq:exploration_success} in cases (1) and (2). 
	\medskip 

	\noindent{\em Step 2 in case (3):} This case is treated similarly to case (2), except that $\chi_1$ is now defined and we aim to define~$\chi_0$. 
	Set~$\La'$ to be the union of~$\La_{R}(u)$ for~$u$ in the top part of~$\partial_\infty \La_{2R}(x_+)$. 
	Define~$\partial_\infty \La'$ and its top and bottom sections similarly to how this was done in case (2). 
	
	Define sub-cases (a) and (b) depending on where~$\Gamma$ exits~$\La'$; set~$\tau'$ to be the exit time. 
	If it is more likely to exit on the bottom part, define~$\chi_0$ as the dual (top) side of~$\Gamma_{[\tau,\tau']}$.
	If it is more likely to exit on the top part, start an exploration~$\tilde \Gamma$ from~$x_+$ leaving primal open edges on the left and dual open ones on the right 
	(this exploration runs along~$\partial \calC_{\leq t}$ up to~$y_-$, and continues at least up to~$\partial \La_{2  L  }(x)$ when~$j$ is a top-most seed). 
	Define then~$\chi_0$ as the dual side of~$\tilde \Gamma$ up to the first exit time~$\partial_\infty \La'$.
	
	The same type of analysis as above shows that 
	\begin{align*}
		\phi_{\sfC_{\leq t}}\big[\text{exploration is successful}  \,\big| \text{ case (3) and~$j$ top-most seed}\big] \geq 1/2,
	\end{align*}
	which implies \eqref{eq:exploration_success} in case (3). 
	\medskip
	
	\noindent{\bf Conclusion: proof of~\eqref{eq:j=j_if_success}.}
	Fix now a successful realisation of ${\rm Exp}$, with the primal and dual arcs denoted $\chi_1$ and $\chi_0$, respectively. 
	To ensure $\bfj = j$, it suffices to connect $\chi_1$ to $\calH_{\geq n}$	
	and produce a dual path starting from $\chi_0$ that prevents any connection between $\calB$ and $\calH_{\geq n}$.
	See Figure~\ref{fig:gluing1}, right diagram. 
	
	By~\eqref{eq:cone_contained},~\eqref{eq:DvH} and~\eqref{eq:RSWnc}, 
	with probability comparable to $s_{n-t}$, 
	the cluster $\sfC_{\chi}$ of $\chi_1$ in $(\sfC_{\leq t} \cup {\rm Exp})^c$ intersects $\calH_{\geq n}$, but does not connect to any point in $\calB$. 
	
	Furthermore, conditionally on $\sfC_{\chi}$, the measure in $(\sfC_{\leq t} \cup {\rm Exp} \cup \sfC_{\chi})^c$ is dominated by the measure $\phi_{\sfC_{\leq t}}$. 
	Indeed, the exploration of $\sfC_{\chi}$ produces free boundary conditions in the remaining part of the space. 
	Thus, 
	\begin{align}
	\phi_{\sfC_{\leq t}}\big[ \calB \nxlra{} \calH_{\geq n}  \,\big|\, {\rm Exp},\, \sfC_{\chi}\big] 
	\geq 
	\phi_{\sfC_{\leq t}}\big[ \sfC_{\leq t} \nxlra{} \calH_{\geq n}  \big] \geq 1/2,
	\end{align}
	Averaging over $\sfC_{\chi}$ and using the previous observation, we conclude \eqref{eq:j=j_if_success}.
\end{proof}

We finally turn to the proof of Proposition~\ref{prop:pre-renewal_density}. 

\begin{proof}[Proof of Proposition~\ref{prop:pre-renewal_density}]
Let $C$, $K$ and $\mu$ be given by Proposition~\ref{prop:equation_martingale}.
Set 
\begin{align}
\ell = \frac{2K}{1-\mu}	\quad \text{ so that }\quad  \mu\ell + K = \frac{1 + \mu}{2} \ell.
\end{align}

The proof is in two steps. 
First we prove that, for any $n,t \geq 1$, 
    \begin{align}
	    \phi\big[N_{s} > \ell , \,\,\forall s \in \{ t+1, \dots, t+r \} \,\big|\, \sfC_{\leq t},\, N_t \leq \ell, \,X_n \neq \dagger\big] &\les \exp(-cr)\quad  \text{ and}\label{eq:pre-renewal_density11}\\
	    \phi\big[N_{s} > \ell , \,\,\forall s \in \{1, \dots, r \} \,\big|\, \,X_n \neq \dagger \big] &\les \exp(-cr), \label{eq:pre-renewal_density12}
    \end{align}
    for some universal constant $c > 0$. 
Then will we argue that
    \begin{align}
	    \phi\big[N_{t+C} =1 \,\big|\, \sfC_{\leq t},\, N_t \leq \ell, \,X_n \neq \dagger\big] \ges 1. \label{eq:pre-renewal_density13}
    \end{align}
    It is clear that~\eqref{eq:pre-renewal_density11},~\eqref{eq:pre-renewal_density12} and~\eqref{eq:pre-renewal_density13} imply Proposition~\ref{prop:pre-renewal_density}: indeed, at each time that $N_t \leq \ell$,~\eqref{eq:pre-renewal_density13} ensures that there is a uniform chance that the process comes back to 1. Moreover,~\eqref{eq:pre-renewal_density11} and\eqref{eq:pre-renewal_density12} ensure that the process comes back exponentially often to the state $N_t \leq \ell$.
Thus, we focus on the proof of these three bounds below. \medskip 

We start off with the proof of \eqref{eq:pre-renewal_density11}. Let $\alpha = \frac2{1 + \mu} > 1$.
Fix $t \geq 1$ and work under the conditional measure $P:=   \phi[\, \cdot \,\big|\, \sfC_{\leq t},\,  \,X_n \neq \dagger]$.
For $s \geq 0$, set $M_s =\alpha^{s} N_{t + Cs}$ and $\tau = \inf\{s \geq 1: N_{t + Cs}  \leq \ell\}$. 
Then, \eqref{eq:equation_sur-martingale} and the choice of $\ell$ imply that 
$M_{s \wedge \tau}$ is a positive super-martingale. 

It follows that 
\begin{align}
P[\tau > r ] \leq P[M_{r \wedge \tau} \geq \alpha^{r}  \ell ] \leq \frac1{\alpha^r \ell} E[M_{r \wedge \tau} ]
\leq \frac1{\alpha^r \ell} E[M_{0}].
\end{align}
This directly proves \eqref{eq:pre-renewal_density11}, since  $E[M_{0}] \leq \ell$.

The same reasoning applies to \eqref{eq:pre-renewal_density12}, provided that we have a uniform bound on the expectation
$\phi\big[N_{1} \,|\, \,X_n \neq \dagger]$. Such a bound is easily obtained from \eqref{eq:cone_contained} since
\begin{align}
	\phi\big[N_{1}  \geq 2 \alpha k \,|\, \,X_n \neq \dagger] 
	\les
	\phi \big[\La_{  L  } \lra (\calY_\alpha -(k  L  ,0))^c \,\big|\, \La_{  L  } \lra \calH_{\geq n} \big]
	\les \e^{-c k}.
\end{align}

Finally, \eqref{eq:pre-renewal_density13} is a direct consequence of \eqref{rem:going_down_to_1}.
\end{proof}

\subsection{Mixing when $N_k=1$}\label{subsection: uniform mixing}

We now prove that the process~$(X_k)_k$ satisfies a mixing property at every time $t$ for which~$N_t=1$. 

Write~$\sfC_{\geq t}=\sfC\setminus \sfC_{\leq t}$ and~${\sfC}_{ s\leq \cdot \leq t}= \sfC_{\leq t} \setminus \sfC_{\leq s}$ for~$0 \leq s \leq t$.
We will argue that, conditionally on the past $\sfC_{\leq t}$, the future $\sfC_{\geq t + 1/2}$ may be sampled independently of the past with positive probability. 
To state our result well, some vocabulary needs to be introduced. 
First of all, for any $n\geq 0$, we call $\mathcal{L}_n := \phi^1_{\calH_{\geq 0}}[\, \cdot\,|\, \mathscr{L}_{0,0} \lra \calH_{\geq n} ]$.

A realisation $\chi$ of $\sfC_{\geq t+1/2} - (Lt,X_t)$ is called a {\em cluster-future}.  
It may be formed of several connected components, each containing at least some edge in $\calH_{\geq 1/2}$ and at least one point on $\partial \calH_{\geq 1/2}$; see Figures~\ref{fig:well_behaved}  and~\ref{fig:link} for illustrations. 
The section $\sfC_{t\leq . \leq t+1/2} = \sfC \setminus (\sfC_{\leq t} \cup \sfC_{\geq t+1/2})$ connecting the past cluster to the (translate of the) cluster-future is called the {\em link}.
We now define a particular class of {\em well-behaved} cluster-futures. 

Consider the measure $\phi_{\calH_{\geq 0}}^1$ and write $\sfC$ for the cluster of $\mathscr{L}_{0,0}$. 
Fix a cluster-future $\chi$ and define $\calE(\chi)$ as the event that the rectangle $[\frac38 L , \frac12 L] \times [-L,L]$ contains horizontal open crossings, 
that the lowest one is contained in $[\frac38 L , \frac12 L] \times [-L/2,-L/3]$ and the highest one is contained in $[\frac38 L , \frac12 L] \times [L/3,L/2]$, 
and that they are connected via an open path in $[\frac38 L , \frac12 L] \times [-L,L]$. 
Moreover, $\calE(\chi)$ requires that, if we consider the cluster $\Xi$ of $\chi$ in $\calH_{\geq 3/8}$, 
we have $\Xi\setminus\chi \subset [\frac38 L , \frac12 L] \times [-L,L]$ and it crosses horizontally $[\frac38 L , \frac12 L] \times [-L,L]$.
See Figure~\ref{fig:well_behaved} for an illustration.

\begin{figure}
\begin{center}
\includegraphics[width = 0.4\textwidth]{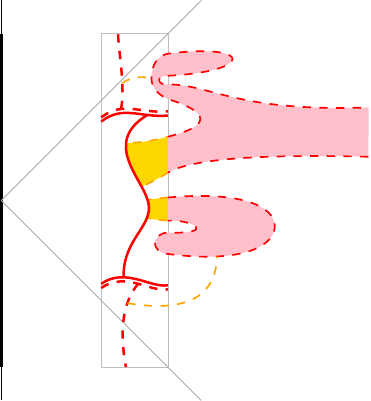}
\caption{The cluster future $\chi$ is formed of the pink connected components, surrounded by dual paths. 
The thick red crossings in the rectangle $[\frac38 L , \frac12 L] \times [-L,L]$ are required for $\calE(\chi)$. 
Moreover, the orange part --- including the dual orange connections --- ensures that $\chi$ is linked to these crossings in such a way that $\calE(\chi)$ occurs. } 
\label{fig:well_behaved}
\end{center}
\end{figure}

For a constant $c_{\rm WB} >0 $, we call $\chi$ $c_{\rm WB}$-well-behaved if 
$\chi$ is contained in $\calY$ and 
\begin{align}\label{eq:c_WB}
	\phi_{\calH_{\geq 0}}^1[\calE(\chi) \,| \,\sfC_{\geq 1/2} = \chi] \geq c_{\rm WB}.
\end{align}
The following is a direct consequence of Proposition~\ref{prop:cone_contained} and \eqref{eq:RSWnc}. 

\begin{Cor}\label{cor:cone_contained}
	There exists a universal constant $c_{\rm WB} >0$ such that, for all $n,s \geq 1$,
	\begin{align}\label{eq:cone_contained_calL}
	  \calL_{n}[\sfC_{\geq 1/2}  \text{ $c_{\rm WB}$-well-behaved}]  \geq c_{\rm WB} \quad \text{ and } \quad 
	  \calL_{n}[\sfC \subset \calY - (s L,0)]  \geq 1- \e^{-c_{\rm WB} s}.
	\end{align}
	Moreover, the same holds under 
	$\phi [\, \cdot \,|\, {\rm Past}(A) \text{ and } \sfC \cap \calH_{\geq n} \neq \emptyset ]$ for any potential past $A$.
\end{Cor}

\begin{proof}
	The second inequality is indeed a direct consequence of Proposition~\ref{prop:cone_contained}. 
	We focus on the first and treat the case of $\calL_{n}$ for simplicity. 
	
	By Proposition~\ref{prop:cone_contained} and \eqref{eq:RSWnc} we find 
	\begin{align}\label{eq:well_behaved_0}
		\phi_{\calH_{\geq 0}}^1\big[\sfC_{\geq 1/2} \subset \calY, \, \calE(\sfC_{\geq 1/2}) \text{ and } \mathscr{L}_{0,0} \lra \calH_{\geq n}\big] \ges
		\phi_{\calH_{\geq 0}}^1\big[\mathscr{L}_{0,0} \lra \calH_{\geq n}\big] 
	\end{align}
	Indeed, we may first explore the top and bottom boundaries of $\sfC$, starting on $\partial \calH_{\geq 0}$, up to the first time they reach $\calH_{\geq 1}$.
	By \eqref{eq:RSWnc}, we may direct these to satisfy the requirements of the event $\calE(\cdot)$ and to arrive at a macroscopic distance from each other. 
	Conditionally on their realisation, using \eqref{eq:RSWnc} again, we may connect them by a primal path in $[\frac38 L , \frac12 L] \times [-L,L]$, 
	and connect them by dual paths to the top and bottom of $[\frac38 L , \frac12 L] \times [-L,L]$. 
	Finally, we connect them to $\calH_{\geq 1}$ while ensuring that their cluster remains in $\calY$. 
	Conditionally on the previous events, the latter has probability of the order of $\phi_{\calH_{\geq 0}}^1[\mathscr{L}_{0,0} \lra \calH_{\geq n}]$ due to Proposition~\ref{prop:cone_contained} and Lemma~\ref{lem:cone_to_half-plane_mixing}.
	
	From \eqref{eq:well_behaved_0}, we conclude that 
	\begin{align}
		\sum_{\chi} \phi_{\calH_{\geq 0}}^1[\calE(\chi) \,| \,\sfC_{\geq 1/2} = \chi] \calL_{n}[\sfC_{\geq 1/2} = \chi] =	\calL_n \big[\sfC_{\geq 1/2} \subset \calY, \, \calE(\sfC_{\geq 1/2})\big] \ges 1.
	\end{align}
	where the sum is over all possible cluster-futures $\chi \subset \calY$. 
	The result follows readily. 
\end{proof}

Henceforth we fix $c_{\rm WB}$ given by Corollary~\ref{cor:cone_contained} and omit it from notation. 
We may now state the crucial mixing estimate for the process $(X_t)_{t\geq 0}$, and more generally for the exploration $(\sfC_{\leq t})_{t\geq 0}$ of the cluster of $0$.

\begin{Prop}\label{prop:pre-renewal_is_renewal}
	There exists~$\eta > 0$ such that for any~$1 \leq t\leq n~$ and 
	any well-behaved cluster-future~$\chi$ 
	\begin{align}\label{eq:pre-renewal_is_renewal}
	        \phi\big[& \sfC_{\geq t+1/2}= \chi +(Lt,X_t) \text{ and } \sfC_{\geq t}\subset \calY + (L(t-1),X_t)  \,\big|\, \sfC_{\leq t}, N_t = 1, X_{n} \neq \dagger\big] \nonumber\\
	        &\geq \eta\,  \calL_{n-t}[\sfC_{\geq 1/2} = \chi],
	 \end{align}
	 where~$\chi +(Lt,X_{t})$ is the translation of~$\chi$ by~$(Lt,X_{t})$. 
	 
	Moreover, for any~$1 \leq s < t \leq n$, any possible realisation~$\zeta$ of~$\sfC_{s + \frac12 \leq \cdot  \leq t} -  (Ls,X_s)$ such that~$N_t=1$, 
	and any realisation $\chi$ of $\sfC_{\geq t} - (Lt,X_t)$ contained in $\calY - (\frac{t-s}2 L,0)$, 
	\begin{align}\label{eq:exp_mem_tail}
	        \sum_\chi \Big|\phi\big[\sfC_{\geq t}= \chi +(Lt,X_t) \,\big \vert\,  \sfC_{\leq s}, \sfC_{s + \frac12 \leq \cdot  \leq t} = \zeta +(Ls,X_s) , X_{n} \neq \dagger\big] & \nonumber \\
				- \calL_{n-s}\big[\sfC_{\geq t-s} = \chi +(L(t-s),X_{t-s}) \,\big|\, \sfC_{\frac12 \leq .\leq t-s} = \zeta\big]& \Big| 
		\leq \e^{-\eta (t-s)}. 
	 \end{align}
\end{Prop}

The first inequality together with Corollary~\ref{cor:cone_contained} state that, at every pre-renewal time, 
there is a positive probability for the cluster-future to be sampled independently of the past, thus creating an actual point of renewal. 
It is tempting to think that this implies exponential tails for the spacing between renewal times, thus proving Theorem~\ref{thm:killed_renewal_structure}. 

Unfortunately this is not the case. Indeed, \eqref{eq:pre-renewal_is_renewal} implies a first-moment bound on the number of renewals, but may not be used to control correlations between renewal times. 
Hence the need for~\eqref{eq:exp_mem_tail}, which  states that, if the clusters sampled according to 
$\phi[. \, \vert\,  \sfC_{\leq s} , X_{n} \neq \dagger]$ and~$\calL_{n-s}$ are coupled for~$t-s$ steps, then they remain coupled for the rest of the process  with probability exponentially close to~$1$. 
Crucially, the value of~$\eta>0$ may be chosen uniformly in~$p$, the direction~$\vec w$ as well as in~$t,s$ and~$n$ and the realisation of the cluster up to~$t$.

Finally, note that in \eqref{eq:pre-renewal_is_renewal} we require that both the translated cluster-future and the link be contained in $\calY + (L(t-1), X_t)$. 
This will eventually ensure the cone-containment property of the cluster; it is a technical detail which may be ignored in a first instance. 

\begin{Rem}\label{rem:link}
Eventually~\eqref{eq:pre-renewal_is_renewal} will be used to state that at any pre-renewal time,  $\sfC_{\geq t+1/2}$ is sampled with positive probability independently from the past. 
We do not claim this about the entirety of $\sfC_{\geq t}$; indeed,~\eqref{eq:pre-renewal_is_renewal} is not expected to apply to the link.
See also Figure~\ref{fig:link}. 

This is a crucial difference with~\cite{campaninonioffevelenikozrandomcluster}, where the link is trivial due to the use of cone points, 
and independent sampling estimates apply to the whole of $\sfC_{\geq t}$.
\end{Rem}

The rest of the section is dedicated to proving Proposition~\ref{prop:pre-renewal_is_renewal}.

\begin{figure}
\begin{center}
\includegraphics[width =.6\textwidth, page = 1]{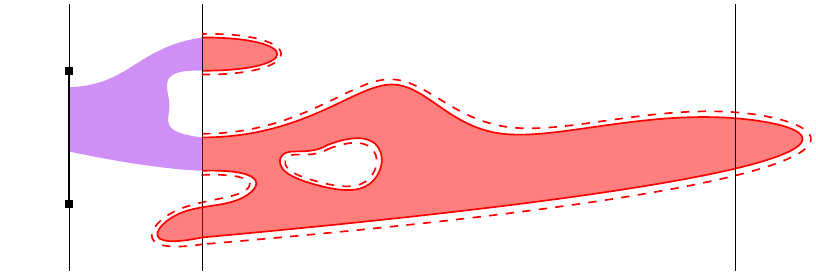}\vspace{.3cm}
\\
\includegraphics[width =.6\textwidth, page = 2]{link.pdf}
\caption{\emph{Top:} A realisation $\chi$ of $\sfC_{\geq 1/2}$ under $\calL_{n-t}$ in red, and the link in purple. 
\emph{Bottom:} The same realisation, shifted by $(Lt,X_t)$ as $\sfC_{\geq t+1/2}$. The measure here is $\phi[ \cdot \,|\, \sfC_{\leq t}, X_{n} \neq \dagger\big]$. Note that the link is different in the two diagrams.}
\label{fig:link}
\end{center}
\end{figure}

\begin{proof}[Proof of Proposition~\ref{prop:pre-renewal_is_renewal}]
	We start with the proof of~\eqref{eq:pre-renewal_is_renewal}, which is the most complicated property.

	Working under $\phi_{\calH_{\geq 0}}^1$,
	write $\Gamma_{N}$ and $\Gamma_{S}$ for the highest and lowest, respectively, open horizontal crossings of $[\frac38 L , \frac12 L] \times [-L,L]$
	and $\Gamma_{E}$ for the right-most open path connecting these. 
	When all three paths exist, write ${\rm UnExp}$ for the region of $[\frac38 L , \frac12 L] \times [-L,L]$ between $\Gamma_N$ and $\Gamma_S$ and left of $\Gamma_E$. 
	Furthermore, write $\sfC^{\Gamma}$ the cluster of $\Gamma_N$ in the region $\calH_{\geq \frac38 L} \setminus {\rm UnExp}$.
	When the paths $\Gamma_N$, $\Gamma_S$ or $\Gamma_E$ do not exist, set $\sfC^{\Gamma} = \emptyset$. 
	Finally, define $\sfC^{\Gamma}_{\geq 1/2}$ in a same way as $\sfC_{\geq 1/2}$. 
	
	Fix a well-behaved cluster-future $\chi$ that reaches $\calH_{\geq n-t}$.  The event $\calE(\chi)$ may be determined by $\sfC^{\Gamma}$. Indeed, it requires that
	all three paths $\Gamma_{N},\,\Gamma_S$ and $\Gamma_E$ exists, that the first two be contained in specific regions and that $\chi$ is connected to those paths in a specific way inside $\calH_{\geq \frac38 L} \setminus {\rm UnExp}$ .
	Recall the notation $s_{n} = \phi_{\calH_{\leq 0}}^1[\mathscr{L}_{0,0} \lra \calH_{\geq n}]$.
	We have
	\begin{align}
		\calL_{n-t}[\sfC_{\geq 1/2} = \chi] 
		&\les \tfrac{1}{s_{n-t}} 		
		\phi_{\calH_{\geq 0}}^1[\calE(\chi) \text{ and }\sfC_{\geq 1/2} = \chi] \\
		& \leq \tfrac{1}{s_{n-t}}  \phi_{\calH_{\geq 0}}^1[\calE(\chi) \text{ and } \sfC^{\Gamma}_{\geq 1/2} = \chi],
			\label{eq:pre-renewal_is_renewal1}
	\end{align}
	where the first line is a consequence of the definition of well-behaved cluster-futures and the second is due to the inclusion of events. 
	
	Similarly, when working under $\phi[\,\cdot \,|\, \sfC_{\leq t}]$, with a past $\sfC_{\leq t}$ such that $N_t = 1$, 
	define $\tilde \sfC^{\Gamma}$ in the same way as $\sfC^{\Gamma}$, but shifted by $(Lt,X_t)$.
	Then 
	\begin{align}
  		&\phi\big[ \sfC_{\geq t+1/2}= \chi +(Lt,X_t)  \text{ and } \sfC_{\geq t} \subset \calY + (L(t-1),X_t)\,\big|\, \sfC_{\leq t}, X_{n} \neq \dagger\big] \\
		& \geq  \frac{1}{\phi\big[ X_n \neq \dagger \,\big|\, \sfC_{\leq t} \big]}\sum_{\chi^\Gamma} \phi\big[\tilde\sfC^{\Gamma} = \chi^\Gamma +(Lt,X_t) \,\big| \, \sfC_{\leq t} \big] \cdot  \phi_{D}^\tau \big[\sfC_{\leq t} \Leftrightarrow \Gamma_E + (Lt,X_t) \big],
	\end{align}
	where the sum is over all possible realisations $\chi^\Gamma$ of $\sfC^{\Gamma}$ such that 
	$\sfC^{\Gamma}_{\geq 1/2} = \chi$ and such that $\calE(\chi)$ occurs.
	Above $D$ is the complement of $\sfC_{\leq t} \cup (\chi^\Gamma + (Lt,X_t))$ and $\tau$ are the boundary conditions induced by conditioning on $\tilde\sfC^{\Gamma} = \chi^\Gamma +(Lt,X_t)$, that is, those induced by $\sfC_{\leq t}$ on its boundary 
	and which are wired on the portion on $\partial (\chi^\Gamma  + (Lt,X_t))$ formed of $\Gamma_N +  (Lt,X_t)$, $\Gamma_S + (Lt,X_t)$ and $\Gamma_{E} + (Lt,X_t)$, and free elsewhere. 
	Finally, we write $\sfC_{\leq t} \Leftrightarrow \Gamma_E + (Lt,X_t)$ for the event that $\sfC_{\leq t}$ and  $\Gamma_E + (Lt,X_t)$  are connected, but that they are not connected to any other point of $\calH_{\geq t+1/2}$ or $(\calY + (L(t-1),X_t))^c$. 
	
	We now argue that 
	\begin{align}
		\phi_{D}^\tau \big[\sfC_{\leq t} \Leftrightarrow \Gamma_E + (Lt,X_t)\big] 
		\ges \phi_{D}^\tau \big[\sfC_{\leq t} \lra  \partial \calH_{\geq t + \frac14} \big] 
		&\ges \phi \big[\sfC_{\leq t} \lra  \partial \calH_{\geq t + \frac14} \,\big|\, \sfC_{\leq t} \big] 
		\\\text{ and }\quad 
		\phi\big[ X_n \neq \dagger \,\big|\, \sfC_{\leq t} \big]& \les \phi \big[\sfC_{\leq t} \lra  \partial \calH_{\geq t + \frac14} \,\big|\, \sfC_{\leq t} \big]  s_{n-t}.
	\end{align}
	The first inequality is proved as Lemma~\ref{lem:seed} and uses \eqref{eq:RSWnc}. Indeed, under the event $\sfC_{\leq t} \lra  \partial \calH_{\geq t + \frac14}$, 
	we may explore either the top-most or bottom-most interface stemming from $\sfC_{\leq t}$ so that, with positive probability, it exposes a wired arc. 
	Then, it suffices to connect the wired arc to $\Gamma_E$ via a primal path and 
	to connect the top and bottom of $\sfC_{\leq t}$ to the top of $\Gamma_N$ and the bottom of $\Gamma_S$ by dual paths, all contained in $\calH_{\leq t+1/2} \cap \big(\calY + (L(t-1),X_t)\big)$ --- \eqref{eq:RSWnc}  and the wide opening of $\calY$ allow one to construct these connections with positive probability. 
	The second and third inequalities follow from Lemma~\ref{lem:cone_to_half-plane_mixing}. 

	We conclude that 
	\begin{align}
  		\phi\big[ &\sfC_{\geq t+1/2}= \chi +(Lt,X_t)  \text{ and } \sfC_{\geq t} \subset \calY + (L(t-1),X_t)\,\big|\, \sfC_{\leq t}, X_{n} \neq \dagger\big] \\
		&\ges  \tfrac{1}{s_{n-t}}\sum_{\chi^\Gamma} \phi\big[\tilde\sfC^{\Gamma} = \chi^\Gamma +(Lt,X_t) \,\big| \, \sfC_{\leq t} \big]\\
		&\ges  \tfrac{1}{s_{n-t}} \sum_{\chi^\Gamma}  \phi_{\calH_{\geq 0}}^1\big[\sfC^{\Gamma} = \chi^\Gamma \big]\\
		&=  \tfrac{1}{s_{n-t}} \phi_{\calH_{\geq 0}}^1\big[\calE(\chi) \text{ and } \sfC^{\Gamma}_{\geq 1/2} = \chi\big],
		\label{eq:pre-renewal_is_renewal2}
	\end{align}
	where the sum is over all $\chi^\Gamma$ as above. The second inequality is due to the mixing property of Lemma~\ref{lem:cone_to_half-plane_mixing} and the fact that $\chi^\Gamma$ is contained in $\calY \cap \calH_{\geq 3/8}$. 
	Finally, \eqref{eq:pre-renewal_is_renewal1} and \eqref{eq:pre-renewal_is_renewal2} together imply \eqref{eq:pre-renewal_is_renewal}.
	\medskip 	
	
	The second property~\eqref{eq:exp_mem_tail} follows by the reasoning leading to the mixing property \eqref{eq:cone_to_half-plane_mixing2}.
	Indeed, even though we are conditioning on the edges of the ``middle part'' given by $\zeta$, observe that the conditioning actually helps the existence of dual paths separating $\calH_{\geq 0}$ and $\calY + (\frac{t-s}{2}L, 0)$ in $\zeta^c$.
	Thus, the proof of the mixing property can be repeated \emph{mutatis mutandis} and yields~\eqref{eq:exp_mem_tail}, and we do not give further details.   
\end{proof}

\subsection{Proof of Theorem~\ref{thm:killed_renewal_structure}}

	Fix a direction~$\vec w$. 
	Recall the notations~$\sfC$ for the cluster of $0$ and $N_t$ for the number of active segments on $\partial \calH_{\leq t}$. 
	Write~$S_0 = 0$ and for~$k\geq 0$ set 
	\begin{align}
		S_{k+1}  = \inf\{ t \geq S_k+2: \, N_t=1\}
	\end{align}
	Note that we impose that~$S_{k+1} - S_k \geq 2$; this is purely for technical reasons. 
	Recall that the times $(S_k)_{k\geq 1}$ are called the pre-renewal times of $(X_t)_t$. 
	Let~$K$ be the first index for which~$S_K =\infty$. 
			
	Let~$\sfD_k = \sfC_{S_{k-1}\leq .\leq S_k}$ for~$k = 1,\dots, K-1$. 
	We also set~$\sfD_{K} = \sfC_{\geq S_{K-1}}$; 
	this is the only piece~$\sfD_k$ which does not end with a pre-renewal time.
	First we will describe how to sample the pieces~$(\sfD_k)_{k\geq0}$ sequentially, which in turn constructs the sequence~$(S_k)_{k\geq0}$. 
	By constructing the pieces $(\sfD_k)_{k\geq0}$, we construct the cluster $\sfC$ and therefore the sequence $(X_t)_{t\geq0}$.
	We will later construct the variables $(Y_t)_{t\geq 0}$ and prove the various properties of Definition~\ref{def:KMRP}.
	
	For this proof, write~$\bbP$ for the probability measure used to sample~$\sfC$ according to the procedure described below. 
	To directly prove that the process has a mass-gap, we will sample $\sfC$ under the conditional measure $\phi[\cdot\,|\,0 \lra \calH_{\geq n}]$, or equivalently $\phi[\cdot\,|\, X_n \neq \dagger]$, for some arbitrary $n\geq 0$.
	To sample under the unconditional measure, it suffices to set $n= 0$. 
	\medskip 
	
	\noindent {\bf Sequential sampling of~$\sfC$.} 
	Sample the initial step $\sfD_1$ as $\tilde\sfC_{\leq S_1}$ where $\tilde \sfC$ is a sample of the cluster of $0$ under $\phi [  .\, | \, X_n \neq \dagger ]$.
	The terminology $\tilde \sfC$ will be used to denote a random sample of the cluster of 0, defined with a sequential procedure starting from equation~\eqref{eq:tilde_C_geqs0}, 	and the law of which will be identified with the ``true'' distribution of $\sfC$. 
	
	Fix $k\geq 1$ and assume the pieces $\sfD_1,\dots, \sfD_k$ already defined.
 	Write $\zeta$ for the realisation of~$\sfC_{\leq S_k}$ and assume that $S_k < \infty$ (otherwise the sampling procedure is finished). 
	We now describe how to sample~$\sfD_{k+1}$ conditionally on $\zeta = \sfC_{\leq S_k}$.	
	For simplicity write~$s := S_k$; the value of~$S_k$ is determined by the conditioning, so may be treated as a constant. 
	
	Proposition~\ref{prop:pre-renewal_is_renewal} states that the information on $\sfC_{\leq s}$ needed to sample the future
	is limited to a random number $j$ of past steps, which is to say $\sfC_{s-j \leq \cdot\leq s}$, with $j$ having exponential tails. 
	Our first task is to formalise this. 
	
	For $j > 0$ and~$\chi$ a potential realisation of~$\sfC_{\geq s} - (Ls,X_{s})$ contained in~$\calY - (\frac{jL}2,0)$ set
	\begin{align*}
		q_j (\zeta, \chi)
		= \min_\xi \phi \big[\sfC_{\geq s} = \chi + (Ls, X_s)  \, \big| \, \sfC_{\leq s} = \xi, \, X_n \neq \dagger \big],
	\end{align*}
	where the minimum is taken over all possible realisations $\xi$ of $\sfC_{\leq s}$ with $\xi_{s-j \leq \cdot\leq s} = \zeta_{s-j \leq \cdot\leq s}$. 
	For $j>s$ we simply take $\xi = \zeta$; but the cone condition on $\chi$ still depends on $j$. 
	For all futures $\chi$ not contained in the appropriate cones, we set $q_j (\zeta, \chi) = 0$. 
	
	When $j = 0$, $q_0(\zeta, \chi) = q_0(\chi)$ requires a special definition. 
	For any well-behaved cluster-future~$\chi$ set
	\begin{align}
		q_0 (\chi)
		= \min_\xi \phi \big[\sfC_{\geq s+1/2} = \chi + (Ls, X_s),  \, \sfC_{\geq s} \subset \calY + (L(s-1),X_s)  \, \big| \, \sfC_{\leq s} = \xi, \, X_n \neq \dagger \big],
	\end{align}
	where the minimum is over all potential realisations $\xi$ of $\sfC_{\leq s}$ producing $N_{s} = 1$. For all other cluster futures $\chi$ set $q_0(\zeta,\chi) = 0$.

	Also, define  
	\begin{align}
		Z_j(\zeta) := \sum_{\chi} q_j (\zeta, \chi),
	\end{align}
	where the sum is over all potential realisation $\chi$ of~$\sfC_{\geq s} - (Ls,X_s)$ (or all cluster-futures when $j = 0$). 
	By definition, the quantities~$q_j (\zeta, \chi)$ and~$Z_j(\zeta)$ are increasing in~$j \geq1$.

	The combination of Propositions~\ref{prop:pre-renewal_is_renewal} and~\ref{prop:cone_contained} states that, uniformly in $\zeta$ as above 
	\begin{align}\label{eq:exp_tails23}
		Z_0(\zeta) \ges 1 \quad \text{ and }\quad  1 - Z_j(\zeta) \les \e^{-\eta j}.
	\end{align}
	Define a random variable~$\tilde M_k$ taking values in~$\mathbb N$ with 
	\begin{align}
		\bbP[\tilde M_k \leq j\,|\, \sfC_{\leq s} = \zeta] = Z_j(\zeta). 
	\end{align}
	Sample~$\tilde M_k$, then, conditionally on~$\tilde M_k$ we will sample a potential future for $\sfC$ as follows. 
	
	If $\tilde M_k= 0$ set 
	\begin{align}\label{eq:tilde_C_geqs0}
		\bbP\big[\tilde \sfC_{\geq s + 1/2} = \chi  + (Ls, X_s) \,\big|\, \sfC_{\leq s}= \zeta,\, \tilde M_k= 0\big] 
		= \tfrac{1}{Z_0(\zeta)}  q_0 (\zeta, \chi),
	\end{align}
	for any possible cluster-future $\chi$.
	Then sample $\tilde \sfC_{s \leq \cdot \leq s+1/2}$ according to the conditional measure 
	\begin{align}
	\phi[\, \cdot \, |\,  \sfC_{\leq s} = \zeta, \,\sfC_{\geq s+1/2} = \tilde \sfC_{\geq s + 1/2}  \text{ and }  \sfC_{\geq s} \subset \calY -  (L(s-1), X_s)] .
	\end{align}
	Note here that $\sfC_{\geq s+1/2}$ is sampled independently of $\sfC_{\leq s}$, but that the link is allowed to depend on it.
	
	For $j\geq 1$ and $\chi$ a potential realisation of $\sfC_{\geq s} - (Ls, X_s)$, define $\bbP[\tilde \sfC_{\geq s} = \chi + (Ls, X_s)  \,|\, \sfC_{\leq s},\, \tilde M_k= j]$ inductively by
	\begin{align}
		&\bbP\big[\tilde \sfC_{\geq s} = \chi + (Ls, X_s)  \,\big|\, \sfC_{\leq s}=\zeta,\, \tilde M_k= j\big] \\
		&\qquad= \tfrac{1}{\bbP[\tilde M_k= j  \,|\, \sfC_{\leq s} = \zeta] } \Big( q_j (\zeta, \chi) - \bbP\big[\tilde \sfC_{\geq s} = \chi + (Ls, X_s)   \text{ and } \tilde M_k  < j \,\big|\, \sfC_{\leq s} \big] \Big).
		\label{eq:tilde_C_geqs3}
	\end{align}
	
	It is immediate from the construction above that the sampling of $\tilde \sfC_{\geq s}$ follows the law of $\sfC_{\geq s}$ under $\phi[\, \cdot \, |\,  \sfC_{\leq s} = \zeta]$.	

	If $\tilde \sfC_{\geq t}$ contains a pre-renewal, set~$\sfD_{k+1} = \sfD_1(\tilde \sfC_{\geq t})- (Ls, X_s) $, which is to say the piece of~$\tilde \sfC_{\geq t}$ up to its first pre-renewal. 	
	Otherwise, set $\sfD_{k+1} = \tilde \sfC_{\geq s} - (Ls, X_s)$ and set $S_{k+1} = \infty$; the sampling of the sequence $(\sfD_j)_{j\geq 1}$ is  finished. 
	
	At this stage, we have constructed $\sfC$ as a concatenation of pieces $(\sfD_k)_{k = 0,\dots, K}$.
	The sequence $(X_t)_{t\geq 0}$ is implicitly defined, with $X_1,\dots, X_{S_k}$ depending on $\sfD_1,\dots, \sfD_k$ only. 
	Note that each piece $\sfD_{k+1}$ needs to be translated by $(L S_k,X_{S_k})$ when attached to $\sfC_{\leq S_k}$. 

	Write $ {\rm len}(\sfD_{k})$ is the horizontal ``length'' of $\sfD_{k}$, that is the maximal $t$ for which $\sfD_{k}$ intersects $\calH_{\geq S_{k} + t}$
	--- unless $\sfD_{k}$ is the last step of the process, this is equal to $S_{k + 1} - S_{k}$. 
	For $k >K$ we formally set $\sfD_k = \emptyset$ and  $ {\rm len}(\sfD_{k}) = 0$.
	
	Then, Proposition~\ref{prop:pre-renewal_density} states the existence of a universal constant $c  >0$ such that 
	\begin{align}
	    \bbP \big[{\rm len}(\sfD_{k})  \geq r \,\big|\, \sfC_{\leq S_k} \big] &\leq \exp(-cr) \label{eq:pre-renewal_density6}
	\end{align}
	for all $k,r, n  \geq 0$.
	\medskip 
	
	\noindent {\bf Memory variables and filtration.}
	Heuristically, the random variables $\tilde M_k$ represent ``memory variables''. 
	They will correspond to how far it is needed to look back in the past to sample the piece-to-be-added $D_k$.
	In particular, according to this heuristic, the condition for $k$ to be a renewal time is that $\tilde M_{k+j} \leq jj$ for all $j\geq 0$. 
	We make this precise in what follows.

	By~\eqref{eq:exp_tails23}, 
	\begin{align}\label{eq:exp_tails24}
		\bbP[\tilde M_k> j\,|\,\sfC_{\leq t}] \leq \exp(-\eta j), 
	\end{align}
	for all~$j\geq 0$ and some universal constant $\eta > 0$.
	We may always decrease $\eta$, so we will henceforth assume that $0<\eta \leq c$, where $c>0$ is the constant appearing in \eqref{eq:pre-renewal_density6}.
	
	By \eqref{eq:exp_tails24}, we may bound the variables~$\tilde M_k$ from above by i.i.d. modified geometric variables\footnote{That is, variables $M_k$ with $\bbP[M_k \geq j] = \e^{-\eta j}$ for $j\geq 0$.}~$M_k$ with parameter~$1- \e^{-\eta} >0$. To be precise, at each sampling step $S_k$, sample a geometric variable $M_{k+1}$ starting at $0$, independent of $\sfC_{\leq t}$, then sample $\tilde M_{k+1} \leq M_{k+1}$
	and $\sfD_{k+1}$ according to the procedure described above. 
	
	Set  $Y_t = 1$ if 
	$t = S_k$ for some $k$ and $M_{k+j} \leq j$ for all $j\geq0$.
	We have now defined the sequences $(X_t)_{t\geq 0}$ and $(Y_t)_{t\geq 0}$.
	Observe that $Y_t$ depends on the ``future'' variables $M_{k+ j}$ (if $t$ is such that $S_k = t$). 

	To render dependencies clear, let us properly define the filtration $(\calF_t)_{t\geq 0}$ associated to the KMRP process. 
	For $t\geq 0$, let $\calF_t$ be the $\sigma$-algebra generated by $\sfC_{\leq t}$, 
	all variables $M_k$ and $\tilde M_k$ with $S_k < t$ 
	and all variables $Y_s$ with $s \leq t$. 
	
	Finally, the dependence between the memory variables $M_k$ and the piece-lengths ${\rm len}(\sfD_{k+1})$ will be of interest. 
	Notice that, due to the relationship between $\eta$ in \eqref{eq:exp_tails24} and $c$ in  \eqref{eq:pre-renewal_density6},
	for any $k,r,j \geq 0$,
\begin{align}
	   \bbP \big[{\rm len}(\sfD_{k+1})  \geq r +j  \,\big|\, \sfC_{\leq S_k}, \, M_{k} =  j \big]
		&\leq   \frac{\bbP [{\rm len}(\sfD_{k+1})  \geq r +j  \,\big|\, \sfC_{\leq S_k}]}{\bbP [ M_{k} = j] }\\
			&= 	    \frac{\bbP [{\rm len}(\sfD_{k+1})  \geq r +j  \,\big|\, \calF_{S_k}]}{\bbP [ M_{k} = j] }
	    \les  \e^{-cr}.
	     \label{eq:pre-renewal_density7}
	\end{align}
	The equality above comes from the fact that the additional conditioning on $\calF_k$ has no bearing on $\sfD_{k+1}$. 
	It follows that the variables $({\rm len}(\sfD_{k+1})  - M_{k})_{k\geq 0}$ may be bounded by i.i.d. geometric variable $({\rm ExtraLen}_k)_{k\geq0}$ of some universal parameter, which are also independent of the variables $(M_k)_{k\geq 0}$ 
	\medskip 
	
	\noindent{\bf Renewal structure.}
	We now argue that a time $t$ such that $Y_t = 1$ is indeed a renewal time. 
	Fix $t$ such that $t=  S_k$ and $Y_t=  1$. Then, under $\calF_t$, the variables $(M_{k+j})_{j \geq 0}$ are independent geometric variables conditioned on $M_{k+j} \leq j$ for each $j$. 
	In particular, they are independent of $\calF_t$, and therefore so is the sequence $(Y_s)_{s \geq t}$. 
	
	Additionally, $\sfD_{k+1}$ only depends on $\calF_t$ via its link 
	and $\sfD_{k+j}$ is independent of~$\sfC_{\leq S_{k+j} - j-1}$ and in particular of~$\sfC_{\leq S_{k}+1}$ (here we use that $S_{j+1} - S_j \geq 2$ for all $j$). 
	It follows that $\sfC_{\geq S_k+1/2}$ is independent of $\calF_t$. 
	As a consequence $(X_{k+j}- X_k)_{j\geq 1}$ is independent of $\sfC_{\leq k}$. 
	
	The law $\mathcal L$ appearing in the definition of a KMRP 
	is explicitly constructed when sampling $\tilde \sfC_{\geq s}$ in~\eqref{eq:tilde_C_geqs0} and~\eqref{eq:tilde_C_geqs3} 
	under the unconditioned measure $\phi$.
	\medskip

\noindent {\bf Exponential tails and mass-gap.}
	We now discuss the exponential tails of \eqref{eq:KMRP_exp_tail1} and  \eqref{eq:KMRP_exp_tail2}; those of 
	\eqref{eq:KMRP_exp_tail_init1} and \eqref{eq:KMRP_exp_tail_init2} will be discussed below. 
	Recall the definition of the times $T_k$ defined from the $Y_k$'s in Definition~\ref{def:KMRP}. 
	
	Fix $\ell \geq 1$ , condition on $\calF_{T_\ell}$ so that $T_\ell < \infty$ and let $k \geq \ell$ be the index such that $S_k = T_\ell$. 
	Due to our construction, the whole of $\sfC_{\geq T_\ell + \frac12}$ is contained in the cone $\calY + (LS_k ,X_{S_k})$, which implies \eqref{eq:KMRP_exp_tail2}.  	
	We turn to~\eqref{eq:KMRP_exp_tail1}.

	Define $\kappa > k$ to be the first index for which $M_{\kappa+j} \leq j$ for all $j\geq0$. 
	In other words, $\kappa$ is such that $T_{\ell + 1} = S_\kappa$. 
	Then, for $m\geq 1$,
	\begin{align}\label{eq:exp_tails_mem0}
		\PP\big[ X_{m + T_{\ell}} \neq \dagger \text{ but }  T_{\ell+1} - T_{\ell}  > m \,\big|\, \calF_{T_{\ell}} \big] 
		= \PP\big[ \sum_{j= k+1}^{\kappa} {\rm len}(\sfD_{j}) > m  \,\big|\, \calF_{T_{\ell}} \big].
	\end{align}
	We will now argue that $\kappa-k$ has exponential tails under $\PP[ . \,|\, \calF_{T_{\ell}} ]$
	and that the quantity above is exponentially decreasing in $m$. 
	
	The tails for $\kappa$ are relatively standard and appear in \cite{OttVelenikPottsDefect}. 
	We give a full proof for completeness. 
		
	We say a memory variable $M_i$ reaches an index $j < i$ if $i -M_i \leq j$. 
	With this formulation, $\kappa > k$ is the first index such that none of the $M_i$ with $i\geq \kappa$ reaches indices strictly smaller than $\kappa$. 
	
 	Set $J_0 = k$ and, for $i \geq 0$, let 
	\begin{align}
	J_{i+1} = \max\{ j > J_i:  j - M_j \leq  J_i \}
	\end{align}
	to be the index of the last memory variable reaching $J_i$.
	The above is well defined up to the first $i$ for which the set is empty. Then $\kappa= J_i+1$ and we formally define $J_{i+1} = \emptyset$.
	
	Notice that, for all $i\geq 0$
	\begin{align}
	\PP\big[ J_{i+1} = \emptyset \,\big|\, \calF_{T_{\ell}}, J_1,\dots, J_i \big] 	
	= \prod_{j \geq 1} \PP\big[ M_{j+J_i} < j \,\big|\, \calF_{T_{\ell}}, J_1,\dots, J_i \big]  
	> 0. 
	\end{align}
	Indeed, the variables $(M_{j+J_i})_{j\geq1 }$ under the measure above are independent geometrics,	each conditioned not to reach $J_{i-1}$ (nor $J_0-1$).
	As such the probability above may be computed explicitly and be shown to be uniformly positive. 
	
	The same computation shows that 
	\begin{align}
	\PP\big[ J_{i+1} -J_i \geq m \,\big|\, \calF_{T_{\ell}}, J_1,\dots, J_i \big] 	
	\leq \sum_{j \geq m} \PP\big[ M_{j+J_i} \geq j \,\big|\, \calF_{T_{\ell}}, J_1,\dots, J_i \big]  
	\les \e^{- \eta m}.
	\end{align}
	
	Thus, $\kappa  - k - 1$ is the sum of variables $(J_{i+1} - J_i)_{i\geq 0}$ which have exponential tails, up to the first time $J_{i+1} = \emptyset$. 
	It follows that $\kappa  - k$ also has exponential tails. 
	
	Conditionally on $J_1,J_2,\dots$, the variables $(M_j)_{j > k}$ are independent, but have certain conditionings. 
	Indeed, each variable $M_{J_i}$ is conditioned to be larger than $J_{i} - J_{i-1}$, but strictly smaller than $J_{i} - J_{i-2}$, 
	while all other memory variables are bounded above. 
	We conclude that, under $\PP[ . | \calF_{T_{\ell}}]$, 
	 \begin{align}\label{eq:exp_tails_mem1}
		 \sum_{j=k+1}^\kappa M_{j-1} - (\kappa - k) = \sum_{j=k+1}^\kappa M_{j-1} - \sum_{i\geq1}  (J_{i} - J_{i-1})
	 \end{align}
	 has exponential tails. 
	Indeed, the above is the sum of $(\kappa - k)$ random variables which may be bounded by i.i.d. geometrics independent of $(\kappa - k)$. 
	
	Finally, by \eqref{eq:pre-renewal_density7}, we conclude that, under  $\PP[ . | \calF_{T_{\ell}}]$, 
	\begin{align}\label{eq:exp_tails_mem2}
		\sum_{j= k+1}^{\kappa} ({\rm len}(\sfD_{j})  - M_{j-1})
	\end{align}
	also has exponential tails. 
	
	Inserting the exponential tails for $\kappa - k$, \eqref{eq:exp_tails_mem1} and \eqref{eq:exp_tails_mem2} into \eqref{eq:exp_tails_mem0}
	we conclude the existence of a universal constant $c > 0$ such that 
	\begin{align}
		\PP\big[ \sum_{j= k+1}^{\kappa} {\rm len}(\sfD_{j}) > m  \,\big|\, \calF_{T_{\ell}} \big] \leq \e^{-cm} \quad \text{ for all $m\geq 1$}. 
	\end{align}
	\medskip
	
	\noindent{\bf Initial step.}
	It remains to prove~\eqref{equ: first step survival rate} as well as the exponential bounds \eqref{eq:KMRP_exp_tail_init1} and \eqref{eq:KMRP_exp_tail_init2}  for the initial step. 
	It is a direct consequence of~\eqref{eq:one_arm}	
	\begin{align}
		\phi_p[X_1 \neq \dagger] =	\phi_p[ 0 \lra \partial \calH_{\geq 1}]\les \pi_1(L(p)).
	\end{align}
	Conversely, \eqref{eq:RSWnc} and~\eqref{eq:one_arm} allow us to prove that 
	\begin{align}
		\phi_p[X_1 \neq \dagger \text{ and } N_1 =1 ] \ges  \pi_1(L(p)).
	\end{align}
	Then, by the properties of the memory variables, 
	\begin{align}
		\phi_p[T_1 < \infty] \geq \phi_p[T_1 = 2] \ges \phi_p[X_1 \neq \dagger \text{ and } N_1 =2] \ges  \pi_1(L(p)).
	\end{align}
	We recall that we enforced that $S_{k+1} - S_k \geq 2$, which is the reason for which we consider the event $T_1=2$ in the previous display. 
	This proves \eqref{equ: first step survival rate}.
	
	Proposition~\ref{prop:cone_contained} implies an exponential tail for $\min \{ {j\geq0} : \sfC \subset Y - (Lj,0)\}$ when $\sfC$ is sampled under $\phi_p[\,\cdot\,|\, 0 \lra \partial \calH_{\geq n}]$,
	whence \eqref{eq:KMRP_exp_tail_init2} follows readily. 
	
	Finally,~\eqref{eq:pre-renewal_density2} states that, under $\phi_p[\,\cdot\,|\, 0 \lra \partial \calH_{\geq n}]$, $S_1$ has a uniform exponential tail. 
	The rest of the proof of~\eqref{eq:KMRP_exp_tail_init1} and is identical to that of \eqref{eq:KMRP_exp_tail1} above.
	\medskip
	
	\noindent{\bf Killing rate and step variance.}
	From~\eqref{eq:RSWnc} we conclude directly that, for any time $T_k$ with $k\geq 1$, 
	\begin{align}
	\phi[0 \lra \calH_{\geq T_k+1}\,| \, \calF_{T_k}] \leq 1 - c < 1,
	\end{align}
	for some universal constant $c > 0$. This implies that the killing rate is bounded away from $0$ uniformly in all parameters. 
	Furthermore, a quick analysis shows that
	\begin{align}
	\phi[T_{k+1} = T_{k}+2     \text{ and }  X_{T_{k+1}} \geq X_{T_{k}}+1 \,| \, \calF_{T_k}  ] \ges 1, 
	\end{align}
	which implies a uniform lower bound on the $X$-step variance and a uniform upper bound on the killing rate. 
	\medskip
	
	\noindent{\bf Aperiodicity.} Recall that our construction formally prevents us to have $T_{k+1} = T_{k}+1$. However, direct \eqref{eq:RSWnc} constructions prove that 
	\begin{align}
	\phi[T_{k+1} = T_{k}+2  \,| \, \calF_{T_k}] \ges 1 \quad  \text{ and } \quad
		\phi[T_{k+1} = T_{k}+3  \,| \, \calF_{T_k} ] \ges 1,
	\end{align}
	which implies the aperiodicity of the process. 
	 \hfill $\square$

\section{Local limit theorem for the Markov renewal process}\label{sec:KMRP_consequences}

This section contains general results about killed Markov renewal processes formulated for the process $(X_t,Y_t)_{t\geq 1}$ of Theorem~\ref{thm:killed_renewal_structure}. 
While the details of this specific process are not important, we find it easier to formulate the results in this context. 
Even though the notion of Markov renewal process (or compound Markov processes) is classical and has been thoroughly studied, see for instance the monograph~\cite{Borovkov_2022}, the addition of the killing part changes the phenomenology, and we need to reprove a certain number of results.
In particular the hypothesis of the mass gap is fundamental and cannot be bypassed. 

\subsection{Probability of hitting a half-space}

First, we show that the probability that a KMRP with a mass-gap survives for $n$ steps behaves like a pure exponential. We highlight that this result is purely stated for KMRPs. 
In particular, its proof is not based on percolation arguments but rather on analytic ones, and is somewhat orthogonal to the rest of the paper. 

\begin{Th}\label{th:pure_expKMRP}
	Let $(X_t,Y_t)_{t\geq 1}$ be an aperiodic KMRP with a mass-gap and killing rate $\kappa > 0$. Then there exists $\zeta \leq \kappa^{-1}$ such that 
	 \begin{equation}\label{eq:pure_expKMRP}
		\mathbb P[X_n \neq \dagger] \asymp \mathbb P[X_1 \neq \dagger] \, \e^{- n/\zeta} \qquad \text{ for all $n\geq 1$},
	\end{equation}
	where the constants in the equivalence and $\zeta$ are bounded away from $0$ and $\infty$ depending only on the killing rate and the mass-gap. 
\end{Th}

Expressed in terms of the random-cluster model, Theorem~\ref{th:pure_expKMRP} states the following.

\begin{Cor}\label{cor:pure_exp}
	For any~$\vec w \in \bbS^1$ and~$p<p_c$, there exists $\zeta(p,\vec w) > 0$ such that
	\begin{equation}\label{eq:pi1pure_exp}
    	\phi_p[0\lra \calH_{\geq n}^{\vec w}] \asymp \pi_1(L(p))\, \e^{-\frac{n}{\zeta(p,\vec w)L(p)}},
	\end{equation}
	uniformly in~$n \geq L(p)$,~$p$ and~$\vec w$. Moreover, $\zeta$ is bounded uniformly away from $0$ and $\infty$. 
\end{Cor}

The above does not imply that~$\xi_p(\vec w) =  \zeta(p,\vec w)L(p)$, since~$\xi_p(\vec w)$ is not defined in terms of hitting a half-space, but rather a specific point. 
We will see how to deduce~$\xi_p(\vec v)$ in the next section. 

\begin{proof}[Proof of Theorem~\ref{th:pure_expKMRP}]
When studying $\mathbb P[X_n \neq \dagger]$, the first step needs to be treated with particular care.
Indeed, 
\begin{align}\label{eq:decomp1}
	\mathbb P[X_n \neq \dagger] 
	&= \sum_{1 \leq k\leq n} \mathbb P[T_1  = k ] \mathbb P[X_{n } \neq \dagger \,|\, \calF_{k}\, , T_1  = k ] +\mathbb P[T_1 > n \text{ but } X_n \neq \dagger] \\
	&= \sum_{1 \leq k\leq n} \mathbb P[T_1  = k ] \mathbb P[X_{n} \neq \dagger \,|\, \calF_{k}\, , T_1  = k ] + O(e^{-c  n} 	\mathbb P[X_n \neq \dagger]),
\end{align}
where the second line is due to the mass-gap $c  > 0$. 
Furthermore, also due to the mass-gap, $\mathbb P[T_1  = k ] = O(e^{-c  k} 	\mathbb P[X_k \neq \dagger])$.

The main analysis is dedicated to 
\begin{align}
	s_{n-k} := \mathbb P[X_{n} \neq \dagger \,|\, \calF_{k},\, T_1  = k ],
\end{align}
which is a quantity that depends only on $n-k$ and the law $\mathcal L$ of Definition~\ref{def:KMRP}.
In the previous definition, it is a minor inconvenience that $T_1$ is not allowed to be 0 or 1; we arbitrarily set $s_n = s_{n-1} = s_{n-2}$.

For $n\geq 1$, define the auxiliary quantities 
\begin{align}
	a_n := &\PP\big[T_{2} - T_1 = n \,\big|\, \calF_{T_1},\, T_1 < \infty\big] \\
	c_n := &\PP\big[X_{n+ T_1} \neq \dagger \text{ but }T_2 > n+ T_1 \,\big|\, \calF_{T_1},\, T_1 < \infty\big]  \quad \text{ and }\qquad\\
	p_n := &\tfrac{1}{\kappa}\PP\big[\exists \ell \geq 0, T_{\ell} - T_1 = n \text{ and } T_{\ell+1} = \infty \,\big|\, \calF_{T_1},\, T_1 < \infty\big] \\
	=& \PP\big[\exists \ell \geq 0, T_{\ell} - T_1 = n \,\big|\, \calF_{T_1},\, T_1 < \infty\big], 
\end{align}
where we recall that $\kappa$ denotes the killing rate. 
The equality is due to the renewal property. 

By analysing the value of the last renewal time, we find 
\begin{align}\label{eq:decomp2}
	s_n  = \sum_{1 \leq k \leq n}  p_{k} c_{n-k} + c_n =  \sum_{1 \leq k \leq n}  p_{k} c_{n-k} + O(e^{-c  n}  s_n). 
\end{align}
due to the mass-gap.

We will prove that there exists $\zeta > 0$ such that
\begin{align}\label{eq:p_n}
	p_n \asymp \e^{-n/\zeta}
\end{align}
for $n$ large enough. Before doing so, observe that \eqref{eq:p_n} combined with~\eqref{eq:decomp1},~\eqref{eq:decomp2} and the consequences of the mass-gap, yields
\begin{equation}\label{eq:pn_sufficient_to_conclude}
s_n \asymp \e^{-n/\zeta} \qquad \text{ and }\qquad 	\mathbb P[X_n \neq \dagger]  \asymp \mathbb P[X_1  \neq \dagger ] \e^{-n/\zeta}.
\end{equation}

The rest of the proof is dedicated to~\eqref{eq:p_n}. 
By analysing the first renewal time, we find 
\begin{align}\label{eq:p_a_relation}
	p_n = \sum_{k=1}^n a_k p_{n-k} + \1_{n=0}.
\end{align}
Denote by~$P$ and~$A$ the generating series of the sequences~$(p_n)$ and~$(a_n)$, respectively.
Note that these are both power series with positive coefficients.
Writing $R_p$ and $R_a$ for their respective radii of convergence, the mass-gap states that $R_p <  R_a$. 
From \eqref{eq:p_a_relation} we deduce the ``killed renewal equation'' 
\begin{equation}\label{eq: renewal equation}
	P(z) = \frac{1}{1-A(z)}.
\end{equation}
The equality holds in the whole disk of convergence of~$P$.
     
Observe that~$A$ is a power series with positive coefficients that satisfies~$A(1) = 1 - \kappa < 1$. 
Also observe that, due to~\eqref{eq: renewal equation}, it is the case that $A(R_p)= 1$. 
    
Write $\mathbb{D}_{R_p} \subset \mathbb{C}$ for the open disc centered at 0 of radius $R_p$ . We now are going to argue that $R_p$ is the only 0 in~$\Bar{\mathbb D}_{R_p}$ of the series~$A - 1$ and that it is simple. 
Indeed, it easy to see that if~$|z|<1$, then 
\begin{equation*}
        \big\vert \sum_{n \geq 1} a_n (R_p z)^n \big\vert\leq  \sum_{n \geq 1} a_n (R_p |z|)^n < 1. 
\end{equation*}
Furthermore, we argue that the aperiodicity of the process implies that~$A - 1$ does not have an additional 0 on~$\partial \Bar{\mathbb D}_{R_p} \setminus \{R_p\}$. 
Indeed, if it were the case, then there would exist some~$\theta \in (0, 2\pi)$ such that 
~$\sum_{k=0}^{+\infty} a_k R_p^k \e^{\mathrm ik\theta} = 1$. 
By the equality case in the triangular inequality, one would then have that for all the~$k \in \mathrm{Supp}(T_2-T_1)$, the~$\e^{\mathrm i k\theta}$ are aligned; this contradicts the aperiodicity of~$T$.
Finally, since $R_p < R_a$,~$A$ has a positive derivative at $R_p$, which implies that $R_p$ is a simple zero of $A-1$. 

Summarising the above, we proved that the function~$$g(z) := \frac{1-A(z)}{R_p-z}$$ does not vanish on~$\Bar{\mathbb D}_{R_p}$. 
We wish to apply Wiener's~$1/f$ theorem (see~\cite[Theorem 5.2]{TirgonometricseriesZygmund}) to the function~$g$. 
To that end, we need to check that the corresponding series is summable at $R_p$. It is a simple observation that, for all $|z| < R_p$, 
\begin{equation*}
        g(z) = \sum_{n= 0}^{+\infty} \Big(\frac{z}{R_p}\Big)^n\Big(1 - \sum_{k=1}^{n}R_p^k\, a_k\Big) = \sum_{n= 0}^{+\infty} \Big(\frac{z}{R_p}\Big)^n\Big(\sum_{k=n+1}^{+\infty}R_p^k\, a_k\Big).
\end{equation*}
Due to the positivity of the terms $a_k$, when inserting $z = R_p$ in the above, we find
 \begin{equation*}
 \sum_{n=0}^{+\infty} \sum_{k=n+1}^{+\infty}R_p^k\, a_k = \sum_{k=1}^{+\infty} R_p^k\, ka_k = R_p\,  A'(R_p) <\infty.
 \end{equation*}
 
By Wiener's~$1/f$ theorem~\cite[Theorem 5.2]{TirgonometricseriesZygmund},~$1/g$ may be expanded as a power series~$\sum_n b_n z^n$ in~$\Bar{\mathbb D}_{R_p}$ 
and this series is absolutely summable at the point $z=R_p$. 
By~\eqref{eq: renewal equation}, for any~$z \in {\mathbb D}_{R_p}$,
 \begin{equation*}
     \frac{1}{g(z)} =  (R_a-z) P(z) = \sum_{n\geq 0} z^n (R_p p_{n} - p_{n-1}).
 \end{equation*}
As the above expression can be evaluated at $z=R_p$, by Abel's radial theorem
 \begin{equation}\label{eq: limite Abel's theorem}
 	\lim_{z \rightarrow R_p^-} \frac{1}{g(z)}   
	=\sum_{n\geq 0} (R_p^{n+1} p_{n} - R_p^{n}p_{n-1}) = \lim_{n\to\infty} R_p^{n+1}p_n.
 \end{equation}
 Notice however that 
 \begin{align}
 	\lim_{z \rightarrow R_p^-} \frac{1}{g(z)}  =    \lim_{z \rightarrow R_p^-} \frac{R_p-z}{1-A(z)} =   \frac{1}{A'(R_p)} > 0.
 \end{align}
The last two diplays allow us to conclude that 
$p_n = \tfrac{1}{R_p A'(R_p)}R_p^{-n}(1+o(1))$ when $n \rightarrow \infty$.
The explicit value of $\lim_n p_n R_p^{n}$ is irrelevant for us; one simply needs to observe that it is strictly positive and finite. 
Setting $\zeta = (\log R_p)^{-1}$, we obtain that $p_n \asymp \e^{-n/\zeta}$ as required. 
\end{proof}

\begin{proof}[Proof of Corollary~\ref{cor:pure_exp}]
	Fix~$\vec w$ and~$p<p_c$. All constants below will be uniform in~$\vec w$,~$p$ and~$N \geq L(p)$. We have 
	\begin{align}
      		\phi_p\big[X_{\lfloor N/L(p)+1\rfloor} \neq \dagger\big] 
		\leq   	\phi_p\big[0\lra \calH_{\geq N}^{\vec w}\big] 
		\leq     	\phi_p\big[X_{\lfloor N/L(p)\rfloor} \neq \dagger\big] .
	\end{align}
	Apply Theorem~\ref{th:pure_expKMRP} to conclude that 
	\begin{equation*}
	\phi_p\big[0\lra \calH_{\geq N}^{\vec w}\big] \asymp \phi_p[X_1 \neq \dagger]\exp(-\tfrac{n}{L(p)\zeta(p, \vec w)})
	\end{equation*}
	for some $\zeta(p, \vec w)$. Moreover, by~\eqref{equ: first step survival rate}, $\phi_p[X_1 \neq \dagger] \asymp \pi_1(L(p))$, which concludes the proof of the corollary.
\end{proof}

 \begin{Rem}\label{rem: size_bias}
This computation allows us to identify the law of the inter-renewal times when the process is conditioned on surviving at least $n$ steps. 
Indeed, observe that for any $\ell \in \{1,\dots, n-1\}$, 
\begin{align*}
\PP[T_2-T_1 = \ell \,|\, \exists k \geq 0, T_k - T_1 = n, \calF_{T_1}, T_1 < \infty ] 
&= \tfrac{s_{n-\ell}}{s_n}\PP[T_2-T_1 = \ell \,|\, T_1 < \infty] \\
&\asymp \e^{\ell/ \zeta} \PP[T_2-T_1 = \ell \,|\, T_1 < \infty].
\end{align*}

We conclude that when $n$ tends to infinity, the law of the inter-renewal times converges towards an exponentially tilted sample of $\mathcal L( T_2 - T_1 = \cdot \,|\, T_1 < \infty)$. 
The first and last steps of the process, however,  have a different distribution. 
The same computation applied to finite-dimensional marginals yields that the joint distribution of the different inter-renewal times (except the first and last ones) converges towards i.i.d. samples of the exponential tilt of $\mathcal L( T_2 - T_1 = \cdot\,|\, T_1 < \infty)$.

Finally, the distribution of the $X$-steps also converges towards i.i.d. samples of some probability distribution. 
We call $\calL_{\rm irred}$ the limit of the law of $(T_2-T_1, X_2-X_1)$ under the conditionings $\{X_n \neq \dagger\}$ with $n\to\infty$. 
The mass gap implies that this distribution has exponential tails both for the $T$ and the $X$-steps. 
\end{Rem}

\subsection{Endpoint concentration when conditioned on survival}

We now prove a local limit theorem for~$X_n$ under the conditioning~$X_n \neq \dagger$.
Let~$g_\sigma(x) = \frac{1}{\sqrt{2\pi} \sigma}e^{-x^2/2\sigma^2}$ be the Gaussian density with variance~$\sigma^2$. 

\begin{Prop}\label{prop:LCLT}
	Fix~$\vec w \in \bbS^1$ and~$p < p_c$. There exists~$\mu = \mu(p,\vec w)$ and~$\sigma = \sigma(p,\vec w)$ such that, for any~$k \in \Z$, when $n \rightarrow \infty$, 
    \begin{equation}\label{eq:LCLT for the MRP}
       \left| \sqrt n \phi_p\Big[ \lfloor X_n/L(p)\rfloor = \lfloor n \cdot \mu\rfloor  + k \,\Big\vert\, X_n \neq \dagger \Big] - g_\sigma \left(\tfrac{k}{\sqrt{ n}}\right)\right| \to 0.
    \end{equation}
    with the asymptotics being uniform in~$p$ and~$\vec w$. 
    Furthermore~$|\mu(p,\vec w)|$ is uniformly bounded away from~$\infty$,
   and~$\sigma(p,\vec w)$ is uniformly bounded away from~$0$ and~$\infty$.
\end{Prop}

\begin{proof}
We argued in Remark~\ref{rem: size_bias} that when conditioned on $X_n \neq \dagger$, the process has steps sampled from some i.i.d. distribution.
It thus converges towards a so-called \emph{compound Markov process}, using the terminology of~\cite{Borovkov_2022}.
Local limit-type theorems for compound Markov processes are classical; for instance, Proposition~\ref{prop:LCLT} directly follows from~\cite[Theorem 2.1.2]{Borovkov_2022}.

The uniform bounds on the  $\mu(p,\vec w)$ and~$\sigma(p,\vec w)$ follow from those on the mass-gap, mean, variance and killing rate of the KMRP in Theorem~\ref{thm:killed_renewal_structure}. 
\end{proof}

\begin{Rem}\label{rem:additional_RW}
	Other properties typical of random walks may be extended to the process~$(X_n)_n$ such as the existence of a uniform constant~$C > 0$ such that for any $k \in \mathbb{Z}$, 
	\begin{align}\label{eq:k0_is_max}
		\phi_p\Big[ \lfloor X_n/L(p)\rfloor = \lfloor n \cdot \mu \rfloor  + k \,\Big\vert\, X_n \neq \dagger \Big] 
		\leq C \phi_p\Big[ \lfloor X_n/L(p)\rfloor = \lfloor n \cdot \mu \rfloor \,\Big\vert\, X_n \neq \dagger \Big]
	\end{align}
	and a large deviation estimate
	\begin{align}\label{eq:large_deviation}
		\phi_p\Big[ |X_n/L(p) - \lfloor n \cdot \mu|\geq \alpha n\,\Big\vert\, X_n \neq \dagger \Big] 
		= \e^{-I(\alpha) n + o(n)}
	\end{align}
	where~$I$ is a differentiable function on an interval~$(-\eps,\eps)$ for some~$\eps>0$ with~$I(0) = 0$ and~$I(\alpha) > 0$ for~$\alpha \neq 0$.  
	All constants appearing in the above are uniform in~$n$,~$\vec w$ and~$p<p_c$. The rate function~$I$ does depend on the direction~$\vec w$. 
	For proofs of such statements, see~\cite{Borovkov_2022}.
\end{Rem}

As a consequence, we deduce a preliminary form of the OZ-formula. Recall the quantity~$\zeta(p, \vec w)$ defined in Corollary~\ref{cor:pure_exp}.

\begin{Cor}\label{cor:OZw}
	Fix~$\vec w \in \bbS^1$ and~$p < p_c$. There exists~$\mu = \mu(p,\vec w)$  such that
    \begin{align}\label{eq:OZw}
       	\phi_p\big[ 0\lra  x \big]
		\asymp \frac{\pi_1(L(p))^2}{\sqrt{n}}\e^{-\frac {n}{\zeta(p, \vec w)}},
	\end{align}
	uniformly in~$p$,~$\vec w$,~$n \geq 1$ and any~$x \in \Z^2 \setminus \La_{L(p)}$ with~$\|x -nL(p)(\vec w + \mu \cdot \vec w^\perp)\| \leq L(p)$.
\end{Cor}

\begin{Rem}\label{rem:large_deviation2}
	The proof below also allows one to deduce from~\eqref{eq:large_deviation} a large deviation estimate for the hitting position 
	\begin{align}\label{eq:large_deviation2}
		\phi_p\Big[ 0 \lra \lfloor nL(p)(\vec w + (\mu(p,\vec w) + \alpha) \cdot \vec w^\perp)\rfloor  \,\Big\vert\, 0 \lra  \calH^{\vec w}_{\geq n} \Big] 
		= \pi_1(L(p))^2\cdot  \e^{-I(\alpha) n + o(n)}
	\end{align}
	for any~$\alpha$ close enough to~$0$ and some differentiable  rate function~$I$ with~$I(0) = 0$ and~$I(\alpha) > 0$ for~$\alpha \neq 0$. 
\end{Rem}

\begin{proof}[Proof of Corollary \ref{cor:OZw}]
	Fix~$p$,~$\vec w \in \bbS^1$
	and~$x \in \La_{L(p)/2}(nL(p)(\vec w + \mu(p,\vec w)\cdot \vec w^\perp))$ for some~$n\geq 1$. 
	For~$n<3$ it was proved in~\cite{DuminilCopinManolescuScalingRelations} that~$\phi_p [ 0\lra  x ] \asymp \pi_1(L(p))^2$. We will henceforth assume that~$n\geq 3$. 
	The upper and lower bound on~$\phi_p[ 0\lra  x]$ will be treated differently. All constants below are uniform in $p$ and $\vec w$. 
	
	\begin{figure}
	\begin{center}
	\includegraphics[width = 0.6\textwidth]{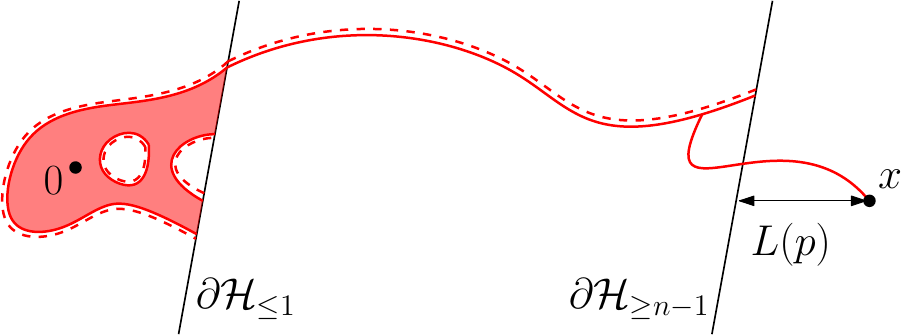}
	\caption{If we assume that $X_{n-1} \neq \dagger$ and that it is at a distance of order $L(p)$ from $x$,
	then connecting $x$ to $0$ requires an additional cost of $\pi_1(L(p))$.
	The picture appears skewed since the horizontal direction is $\vec w + \mu(p,\vec w)\cdot \vec w^\perp$, while the hyperplanes $\partial\calH_{\leq t}$ are orthogonal to $\vec w$. }
	\label{fig:connecting}
	\end{center}
	\end{figure}
	
	We start with the lower bound. 
	Write~$\Gamma$ for the top-most path connecting $\sfC_{\leq 1}$ to~$\calH_{\geq n-1}$. 
	This path is explorable starting from the top of $\sfC_{\leq 1} \cap \partial\calH_{\leq 1}$ and~$X_{n-1}$ is its endpoint on~$\calH_{\geq n-1}$. 
	Conditioning on a realisation~$\Gamma = \gamma$ with~$\lfloor X_{n-1}/L(p)\rfloor = \lfloor (n-1) \cdot \mu (p, \vec w)\rfloor$,
	by applying some RSW-type construction in~$\La_{2L(p)}(x)$ (see Figure~\ref{fig:connecting}) we find that 
	\begin{align}
		\phi_p[ 0\lra  x \,|\, \Gamma = \gamma ] \geq 	c\, \phi_{\La_{L(p)},p}^0[x \lra \partial \La_{L(p)/2}] \asymp \pi_1(L(p)),
	\end{align}
	where~\eqref{eq:RSWnc} and~\eqref{eq:one_arm} state that all constants may be chosen uniform in~$\vec w$ and~$p$.
	Summing over~$\gamma$ as above 
	we conclude that 
	\begin{align}\label{eq:0toxlb}
		\phi_p[ 0\lra  x] 
		\ges \pi_1(L(p)) \,\phi_p\big[ \lfloor X_{n-1}/L(p)\rfloor = \lfloor (n-1) \cdot \mu _{\vec w}\rfloor\big]  
		\asymp \frac{\pi_1(L(p))^2}{\sqrt{n}}\e^{-\frac {n}{\zeta(p, \vec w)}},
	\end{align}	
	where we used~\eqref{eq:LCLT for the MRP} and~\eqref{eq:pi1pure_exp} in the last equivalence. 
	\smallskip 
	
	We turn to the upper bound, for which we decompose
	\begin{align}\label{eq:LaLaub}
		\phi_p[ 0\lra  x] 
		&\leq \phi_p[0 \lra \La_{L(p)} (x)\text{ and } x \lra \partial \La_{L(p)} (x)] \nonumber\\
		&\leq  \phi_p[0 \lra \La_{L(p)} (x)]\,\phi_{\La_{L(p)},p}^1[ 0 \lra \partial \La_{L(p)}],
	\end{align}
	where the second inequality is due to the fact that the two events are measurable in terms of what happens outside and inside of~$\La_{L(p)} (x)$, respectively. 
	We used also used a translation by $x$ for the second factor. 
	The second probability may be bounded above by a universal multiple $\pi_1(L(p))$ due to~\eqref{eq:one_arm}.

	To bound the first quantity, decompose depending on the value of the last renewal time $T_\ell$ before $n$, and keep in mind that cluster after this renewal time is contained in $\calY + (L(p) T_k, X_k)$. 
	Thus, applying \eqref{eq:KMRP_exp_tail1}, we find
	\begin{align}
		 \phi_p[&0 \lra \La_{L(p)} (x)] \\
		& \leq 
		 \sum_{k\geq 0}  \phi_p\big[|X_{n-k} - nL(p)\mu \vec w^\perp | \leq \alpha (k + 1) ,\, Y_{n-k} = 1 \text{ and } Y_{n-k+1} = \dots = Y_n = 0\big] \\
		 &\les
		 \sum_{k\geq 0}  \phi_p\big[|X_{n-k} - nL(p)\mu \vec w^\perp | \leq \alpha (k + 1)\big]  \e^{-n (c + 1/\zeta(p, \vec w))}\\
			 &\les \alpha \pi_1(L(p)) \e^{-\frac {n}{\zeta(p, \vec w)}} \cdot 		 \sum_{k\geq 0}  \tfrac{k+1}{\sqrt{n-k}}\e^{-c n}\\
		&\les
		 \tfrac{\pi_1(L(p))}{\sqrt{n}}\e^{-\frac {n}{\zeta(p, \vec w)}},
	\end{align}
	with the before-last inequality due to \eqref{eq:k0_is_max}.
	Inserting these bounds in \eqref{eq:LaLaub} yields the converse of~\eqref{eq:0toxlb}.
\end{proof}

\subsection{Invertibility of the drift and proof of Theorem~\ref{thm:main}}

In light of the previous section, write 
\begin{align}\label{eq:vecv(vecw)}
	\vec v(\vec w) :=\frac{\vec w + \mu(p,\vec w)\cdot \vec w^\perp}{\|\vec w + \mu(p,\vec w)\cdot \vec w^\perp\|} \in \bbS^1.
\end{align}
Taking~$n\to\infty$ in~\eqref{eq:OZw}, we conclude that 
\begin{align}\label{eq:zeta_xi}
	\xi_p(\vec v(\vec w)) = L(p)\zeta(p,\vec w) \|\vec w + \mu(p,\vec w)\cdot \vec w^\perp\| 
	= \frac{ L(p)\zeta(p,\vec w)}{ \langle\vec v(\vec w), \vec w\rangle}.
\end{align}

The following proposition shows that $\vec w \mapsto \vec v(\vec w)$ is a one-to-one function from $\bbS^1$ to itself.
Theorem~\ref{thm:main} will be a consequence of the surjectivity of this map. Moreover, we shall see in the next section that its injectivity has interesting consequences regarding strict convexity properties of the Wulff shape.
Henceforth the choice of $\vec w^\perp$ is important, and we will assume that it is the rotation by $\pi/2$ of $\vec w$ in the positive direction.

\begin{Prop}\label{prop:v_one_to_one}
	For any~$\vec v \in \bbS^1$, there exists exactly one~$\vec w \in \bbS^1$ so that 
	\begin{align}
 		\vec v(\vec w) = \vec v. 
	\end{align}
\end{Prop}

\begin{proof}	
	The lemma decomposes in two disjoint statements, the surjectivity and injectivity of the function $\vec v$. We start with the former.
	
	We will prove that the function~$\vec w \mapsto \vec v(\vec w)$ is continuous. We start off by proving that $\vec w \mapsto \zeta(p,\vec w)$ is continuous. 
	This is a simple geometric construction; see Figure~\ref{fig:continuity_zeta} for an illustration. 
	
	\begin{figure}
	\begin{center}
	\includegraphics[width = 0.45\textwidth, page = 1]{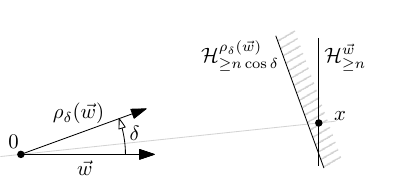}\hspace{0.05\textwidth}
	\includegraphics[width = 0.45\textwidth, page = 2]{continuity_mu.pdf}	
	\caption{{\em Left:} when $\mu(\vec w)$ and $\delta$ have same sign, 
	the point $x = n(\vec w + \mu(\vec w)\vec w^\perp)$ is contained in $\calH^{\rho_\delta(\vec w)}_{\geq n \cdot \cos \delta}$.
	{\em Right:} when $\mu(\vec w) > 0$ and $\delta < 0$, the same point $x$ is contained in $\calH^{\rho_\delta(\vec w)}_{\geq n \cdot\frac{ \cos (\theta - \delta)}{\cos \theta}}$.
	In both images, the grey direction is given by $\vec v(\vec w)$. }
	\label{fig:continuity_zeta}
	\end{center}
	\end{figure}
	
	Fix some~$\vec w \in \bbS_1$ and assume for simplicity that~$\mu(p, \vec w) \geq 0$. 
	Write~$\rho_\delta (.)$ be the rotation by an angle~$\delta$. 
		For~$\delta > 0$, a connection between~$0$ and~$x = nL(p)(\vec w + \mu \cdot \vec w^\perp)$
	intersects~$\calH^{\rho_\delta(\vec w)}_{\geq n \cdot \cos \delta}$
	Together with~\eqref{eq:OZw} and~\eqref{eq:pi1pure_exp} this implies that 
	\begin{align}\label{equ:upper_bound_cos}
		\zeta(p,\rho_\delta(\vec w)) \geq	\cos \delta \cdot \zeta(p,\vec w).
	\end{align}
	
	For~$\delta < 0$, write~$\theta = \arctan (\mu(p,\vec w))$ -- note that~$\theta$ is uniformly bounded away from~$\infty$ by Proposition~\ref{prop:LCLT}.
	Then, a connection between~$0$ and~$x = nL(p)(\vec w + \mu \cdot \vec w^\perp)$ ensures that 
	the cluster of~$0$ intersects~$\calH^{\rho_\delta(\vec w)}_{\geq n \frac{\cos(\theta+ \delta)}{\cos\theta}}$. 
	Thus
	\begin{align}\label{equ:upper_bound_cos2}
		\zeta(p,\rho_\delta(\vec w)) \geq	 \tfrac{\cos(\theta+ \delta)}{\cos(\theta)} \cdot  \zeta(p,\vec w). 
	\end{align}	
	Using the two displays above, and their versions with the roles of~$\vec w$ and~$\rho_\delta(\vec w)$ inverted, we conclude that~$\vec w \mapsto \zeta(p,\vec w)$ is continuous.

	The reasoning above shows more than just the continuity of $\zeta$. Indeed, due to the specific definition of $x$ we conclude that, for any~$\eps > 0$ and~$\delta$ small enough
	\begin{align*}
		\phi_p\Big[ 0\lra  \tfrac{n}{\langle\vec v(\vec w), \rho_\delta(\vec w) \rangle} \vec v(\vec w)  \,\Big\vert\, 0 \lra \calH_{\geq n}^{\rho_\delta( \vec w) }\Big]	
		\ges \exp( - \eps \, n).
	\end{align*}
	Together with the large deviation estimate~\eqref{eq:large_deviation2} and the continuity of the rate function, 
	this implies that~$\vec v(\vec w)$ may be rendered arbitrarily close to~$\vec v(\rho_\delta(\vec w))$, provided $\delta$ is small enough. 
	Thus we conclude the proof of the continuity of~$\vec w \mapsto \vec v(\vec w)$.
	
	To conclude the surjectivity of $\vec w\mapsto \vec v(\vec w)$ observe that, for~$\vec w$ a coordinate vector,~$\mu(p,\vec w) = 0$ due to the symmetry of the lattice.
	It follows that~$\vec v(\vec w) = \vec w$ in this case. 
	Furthermore, symmetry also implies that  $\vec v$ commutes with any reflexion around a lattice axis.
	By the intermediate value theorem and the precedent observation, 
	we conclude that the image of the function~$\vec w \mapsto \vec v(\vec w)$ covers the whole unit circle, as claimed. \smallskip

	We turn to the proof of the injectivity of $\vec w \mapsto \vec v(\vec w)$. 
	We will prove that for any $\vec w$ and $\delta >0 $ small enough, $\vec v(\rho_\delta (\vec w))$ is a small, strictly positive rotation of $\vec v(\vec w)$.
	
	Fix $\delta \in [-\pi/4,\pi/4]$. 
	Applying the large deviation estimate \eqref{eq:large_deviation2} we conclude that, for all $\alpha$ small enough (independent of $\delta$), 
	\begin{align}
\exp\big(-(I(\alpha)  + \zeta(\vec w)^{-1})(n+ o(n))\big)
&\asymp  \phi_p\Big[ 0 \lra \lfloor nL(p)(\vec w + (\mu(p,\vec w) + \alpha) \cdot \vec w^\perp)\rfloor \Big] \\
&\ges		\phi_p\Big[ 0 \lra    \calH^{\vec \rho_\delta(w)}_{\geq n \langle  \vec w + (\mu(p,\vec w) + \alpha) \cdot \vec w^\perp, \vec \rho_\delta(w)\rangle } \Big]  \\
&		\asymp
		\exp\Big(- \tfrac{\cos \delta +  (\mu(p,\vec w) + \alpha)  \sin \delta }{  \zeta(\rho_\delta(\vec w))} (n+ o(n))\Big)
		\label{eq:large_deviation3}
	\end{align}
	where $I$ is the large deviation rate function associated to the walk in direction $\vec w$.
	The constants in the asymptotics above depend on $p$, but not on $n$. 
	Furthermore, we have asymptotic equality in \eqref{eq:large_deviation3} if and only if 
	\begin{align}
		\vec w + (\mu(p,\vec w) + \alpha) \cdot \vec w^\perp &= \rho_\delta(\vec w) + \mu(p,\rho_\delta(\vec w))  \cdot \delta(\vec w)^\perp ,\\ 
		\text{ which translates to }\qquad 
		\vec v(\rho_\delta(\vec v))& =  \frac{\vec w + (\mu(p,\vec w) + \alpha) \cdot \vec w^\perp}{\|\vec w + (\mu(p,\vec w) + \alpha) \cdot \vec w^\perp\|}. \label{eq:v_rho_delta(w)}
	\end{align}
	
	By the continuity of  $\vec w \mapsto \vec v(\vec w)$, if $\delta$ is small enough, 
	the equality above holds for $\alpha$ in the range of validity of \eqref{eq:large_deviation3}.
	We conclude that \eqref{eq:v_rho_delta(w)}
	holds for the unique $\alpha$ minimising
	\begin{align}
		\frac{I(\alpha)  + \zeta(\vec w)^{-1}}{\cos \delta +  (\mu(p,\vec w) + \alpha)  \sin \delta}.
	\end{align}
	Using the fact that $I$ is an even, convex function that is differentiable at $0$ with $I(0) = I'(0) = 0$, we conclude that the minimiser occurs 
	for $\alpha_{\min} = \alpha_{\min}(\delta) > 0$ within the range of \eqref{eq:large_deviation3}. 
	
	Recalling that $\vec v(\vec w)$ is the unit vector in the direction $\vec w + \mu(p,\vec w) \vec w^\perp$ and that 
	$\vec v(\rho_\delta(\vec w))$ is the one in the direction $\vec w + \mu(p,\vec w)\cdot \vec w^\perp  + \alpha_{\min}(\delta) \cdot \vec w^\perp$, 
	we conclude that the latter is indeed a small positive rotation of the former.

	This readily implies the injectivity of the function $\vec w \mapsto \vec v(\vec w)$. 
	Indeed, consider its expression in terms of angle arguments 
	\begin{align}
	\varphi : & [0,\pi/2] \longrightarrow [-\pi,\pi)\\
	&\theta \mapsto \arg ( \vec v( \cos\theta,\sin \theta));
	\end{align}
	recall that the symmetries of the lattice ensure that $\varphi$ maps $0$ and $\pi/2$ to themselves. 
	Then our analysis says that, for any $\theta \in [0,\pi/2)$, there exists $\delta > 0$ such that $\varphi(\theta') > \varphi(\theta)$ for all $\theta < \theta' \leq \theta + \delta$. 
	Together with the  previously proved continuity of $\varphi$, this implies that $\varphi$ is strictly increasing, hence injective. 
	The injectivity extends to the whole range of angles $[0,2\pi)$ by symmetries. 
\end{proof}

We can finally prove Theorem~\ref{thm:main}.

\begin{proof}[Proof of Theorem~\ref{thm:main}]
Fix~$q \geq 1$,~$\vec v \in \mathbb S^1$,~$p< p_c$ and~$n \geq \xi_p(\vec v)$. Define~$\vec w$ as a vector such that~$\vec v(\vec w) = \vec v$. 
Then~\eqref{eq:OZw} implies that
	 \begin{align}§\label{eq:OZ33}
 		\phi_p\big[0 \lra \lfloor n  \vec v \rfloor\big]
		\asymp \frac{\pi_1(L(p))^2}{\sqrt{m}}\e^{-\frac {m}{\zeta(p, \vec w)}},
	\end{align}
	where~$m = \frac{n}{\langle\vec v,\vec w \rangle L(p)}$. 
	Recall that the drift~$\mu(p,\vec w)$ is bounded uniformly, and therefore~$\langle\vec v,\vec w \rangle$ is uniformly positive. 
	Thus~$m\asymp n/L(p)$ and therefore $ \xi_p(\vec v) = \langle \vec v,\vec w \rangle L(p) \zeta(p, \vec w)$.
	Inserting these observations in~\eqref{eq:OZ33} we find
	 \begin{align*}
 		\phi_p\big[0 \lra \lfloor n  \vec v \rfloor\big]
		\asymp 
		\pi_1( \xi_p(\vec v)) ^2 \sqrt{\tfrac{\xi_p(\vec v)}{n}}\, \e^{-\frac {n}{\xi_p(\vec v)}},
	\end{align*}
	as claimed.	
\end{proof}

\subsection{Strict convexity of Wulff shape}\label{sec:Wulff}

We first introduce some notation and discuss classical properties.
Fix~$p< p_c$. For $\vec w \in \mathbb R^2 \setminus \{0\}$ define
\begin{align*}
	(\xi^*_p(\vec w))^{-1} = - \lim_{n\to\infty} \tfrac{1}{n} \log \phi_{p} [0 \lra \calH^{\vec w}_{\geq n}],
\end{align*}
with the existence of the limit being guaranteed by sub-additivity arguments. 
This definition is reminiscent of the one of the inverse correlation length, with the point-to-point connection being replaced by a point-to-hyperplane connection.
For coherence, we set~$\xi_p(0) = \xi^*_p(0) = \infty$. 
Observe that~\eqref{eq:pi1pure_exp} states that, for $\vec w \in \bbS^1$, 
\begin{align}\label{eq:xi*zeta}
	\xi^*(\vec w) = \zeta(p,\vec w)L(p).
\end{align}

It is classical (see~\cite{miraclesolesurfacetension} or~\cite{campaninonioffevelenikozrandomcluster} for a more modern exposition) that both~$\xi_p(\cdot)^{-1}$ and~$\xi^*_p(\cdot)^{-1}$ define norms on~$\R^2$. 
Indeed they are both positive homogeneous, and FKG inequalities imply that they are convex. 

Define their unit balls
\begin{align*}
	\mathcal U_p = \{\vec v \in \R^2: \xi_p(\vec v) \geq 1\} \quad \text{ and } 	\quad \mathcal W_p = \{\vec w \in \R^2: \xi^*_p(\vec w) \geq 1\}.
\end{align*}
The latter is called the {\em Wulff shape} (see~\cite{bodineau2000rigorous} for an extensive review on the subject).

The goal of this section is to explain how our approach allows to re-prove some known results about~$\mathcal U_p$ and~$\mathcal W_p$.
The method is different of~\cite{campaninonioffevelenikozrandomcluster}, as it does not rely on the analysis of the regularity properties of the associated Ruelle--Perron--Frobenius operator.
Indeed we crucially use the invertibility of the function $\vec v$ defined in the last section.

\begin{Th}\label{thm: strict convexity}
	For any~$p<p_c$,~$\mathcal U_p$ and~$\mathcal W_p$ are strictly convex bounded sets of~$\R^2$, symmetric with respect to the coordinate axis with differentiable boundaries. 
	Moreover, they are convex dual to each other, \emph{i.e.} they are linked by the following relation:
	\begin{align}\label{eq:UV}
	\mathcal U_p = \{\vec v \in \R^2: \langle \vec v, \vec w \rangle \leq \xi^*_p(\vec w),  \forall \vec w \in \bbS^1\}
	~
	\text{ and }
	~
	\mathcal W_p = \{\vec w \in \R^2: \langle \vec w, \vec v \rangle \leq \xi_p(\vec v),  \forall \vec v \in \bbS^1\}.
	\end{align}
\end{Th}

In the language of convex bodies,~$\mathcal U_p$ is the convex body with \emph{gauge function} given by~$\xi_p$ 
and the first equality in~\eqref{eq:UV} states that its  \emph{support function} is~$\xi^*_p$.
The same holds for its dual $\mathcal W_p$, with the roles of $\xi_p$ and $\xi^*_p$ reversed; see \cite{webster1994convexity} for details. 
The strict convexity of one shape is equivalent to the fact that the boundary of the other is differentiable~\cite{webster1994convexity}.
The strict convexity of $\mathcal U_p$ and $\mathcal W_p$ is a consequence of our construction;
their duality may be proved directly.

\begin{proof}

Classical results on convex duality~\cite[polar duality theorem, p. 238]{webster1994convexity} imply that the two identities of~\eqref{eq:UV} are equivalent. 
We prove the first one.

It is immediate that $0$ satisfies the conditions of the right-hand side terms of~\eqref{eq:UV}. 

Let~$\vec v \in \mathbb R^2\setminus \{0\}$ and $\vec w \in \bbS^1$ be a unit vector with $\langle \vec v, \vec w \rangle>0$. 
The hyperplane with normal vector~$\vec w$ containing the vertex~$n\vec v$ is given by~$\partial\calH^{\vec w}_{\geq  n\langle \vec v, \vec w\rangle}$.
Thus,
\begin{equation*}
\phi_p[0 \lra n\vec v] \leq \phi_p[ 0 \lra \calH^{ \vec w}_{\geq n \langle \vec v, \vec w\rangle}].
\end{equation*}
Considering the exponential rate of decay of these two quantities as $n\to\infty$, we find
\begin{equation}\label{equ: duality wulff}
\xi_p(\vec v) \langle \vec v, \vec w\rangle \leq \xi_p^*(\vec w).
\end{equation}
The inequality holds trivially when $\langle \vec v, \vec w \rangle\leq 0$. Thus
\begin{equation*}
\mathcal U_p \subset \{\vec v \in \R^2: \langle \vec v, \vec w \rangle \leq \xi^*_p(\vec w),  \forall \vec w \in \R^2\}.
\end{equation*}

For the converse inclusion, let~$\vec v$ be a vector such that for any~$\vec w \in \bbS^1$, $\langle \vec v, \vec w \rangle \leq \xi^*_p(\vec w)$.
Lemma~\ref{prop:v_one_to_one} yields the existence of~$\vec w \in \bbS^1$ such that~$v(\vec w) = \vec v / \|\vec v\|$. 
Then~\eqref{eq:zeta_xi} and~\eqref{eq:xi*zeta} imply that
\begin{equation*}
\xi_p(\vec v ) 
= \frac{1}{\|\vec v\|}\xi_p(v(\vec w)) 
= \frac{\xi^*_p(\vec w)}{\|\vec v\| \langle \vec v(\vec w), \vec w \rangle}
= \frac{\xi^*_p(\vec w)}{\langle \vec v, \vec w \rangle}
\geq 1.
\end{equation*}

We turn to the strict convexity and differentiability properties.
As a standard consequence of the polar duality~\eqref{eq:UV}, for any pair $\vec v, \vec w$ with $\vec v = \vec v(\vec w)$,
$\vec w$ is the normal vector (pointing outwards) of a tangent to $\mathcal U_p$  at the point $\xi_p(\vec v) \cdot \vec v$
and $\vec v$ is the normal vector (pointing outwards) of a tangent to $\mathcal W_p$ at the point $\xi_p^*(\vec w) \cdot \vec w$.

By the bijectivity of $\vec{w} \mapsto \vec v(\vec w)$ proved in Proposition~\ref{prop:v_one_to_one}, we conclude that neither $\mathcal U_p$ or $\mathcal V_p$ have facets; that is, they are strictly convex. 
As strict convexity of a convex body implies differentiability of the boundary of its convex dual, this concludes the proof of the theorem.
\end{proof}

\subsection{Invariance principle}

Fix~$p < p_c$ and~$\vec w \in \bbS^1$. 	
Define the linear interpolation~$X(t)_{t \geq 0}$ of the discrete process~$(X_n)_{n \geq 0}$.
The representation of the cluster in terms of a random-walk like object allows to prove several invariance principles, similar to Donsker's Theorem for random walks. 
For instance, non-uniform invariance principles similar to Theorem~\ref{thm: invariance principle} were derived in~\cite{campaninonioffevelenikozrandomcluster} and~\cite{DAl2024}.
Our representation allows to prove the two following invariance principles. 

\begin{Th}\label{thm: invariance principle}~
\begin{itemize}
\item Under the family of measures~$\phi_p[~\cdot~\vert~0 \lra \calH_{\geq n}^{\vec w}]$, uniformly in~$p$ and~$\vec w$, when~$n \rightarrow \infty$,
\begin{align*}
\tfrac{1}{\sqrt n L(p)} \big( X_{nt} - t  n \mu(p, \vec w) \big)_{t \in (0,1)} \Rightarrow \big( \mathsf B_{t}^{\sigma(p, \vec w)} \big)_{t \in (0,1)},
\end{align*}
where~$B_{t}^{ \sigma(p)}$ denotes the one-dimensional Brownian motion started at 0 with variance~$\sigma(p, \vec w)$ given by Proposition~\ref{prop:LCLT}.
\item Under the family of measures~$\phi_p[~\cdot~\vert~0 \lra n\vec v]$, uniformly in~$p$ and~$\vec v$, when~$n \rightarrow \infty$,
\begin{align*}
\tfrac{1}{\sqrt n L(p)} \big( X_{nt} - t  n \mu(p, \vec w) \big)_{t \in (0,1)} \Rightarrow \big( \mathsf{BB}_{t}^{\sigma(p, \vec w)} \big)_{t \in (0,1)},
\end{align*}
where~$\vec w$ is the unique unit vector such that~$\vec v(\vec w)= \vec v~$ and~$\mathsf{BB}_{t}^{\sigma(p, \vec w)}$ is the one-dimensional Brownian bridge started at $0$ and ending at $0$ with variance~$\sigma(p, \vec w)$.
\end{itemize}
In both statements, the convergence occurs in the space of continuous functions from~$[0,1]$ to~$\mathbb R$, endowed with the topology of uniform convergence. 
\end{Th}

\begin{proof}
This follows directly from Theorem~\ref{thm:killed_renewal_structure} together with classical considerations on Markov renewal processes. 
The first invariance principle directly follows from~\cite[Theorem 1.5.3]{Borovkov_2022}.
The second is a consequence of the first together with the local limit estimate of Proposition~\ref{prop:LCLT}. 
As this is classical, we do not give further details, and refer to~\cite{DAl2024} for the formal reasoning leading to the invariance principle. 
\end{proof}

By a standard union bound, it may be proved that, with probability tending to $1$ 
under both  $\phi_p[~\cdot~\vert~0 \lra \calH_{\geq n}^{\vec w}]$ and  $\phi_p[~\cdot~\vert~0 \lra n\vec v]$, 
\begin{align}
	\max_k (T_{k+1} - T_k) \leq C\log(n/L(p)),
\end{align} 
for $C>0$ a universal constant.

Due to the cone-containment property (see Remark~\ref{rem:cone_contained}), the full cluster is at a distance $C' L(p) \max_k(T_{k+1} - T_k)$ from the points of renewal $(L(p) T_k,X_{T_k})_{k}$. Thus, with probability tending to $1$ under both of the measures above, 
\begin{align}
d_{\calH} \big(\sfC, (L(p)t,X_t)_{0 \leq t \leq n/L(p)}\big) \leq C L(p) \log(n/L(p)),
\end{align}
for some universal constant $C$, where $d_{\calH}$ is the Hausdorff distance and $(L(p)t,X_t)_{0 \leq t \leq n/L(p)}$ is the family of points defined as in \eqref{eq:X_t_def}, expressed in the basis $(\vec w, \vec w^\perp)$. 
 
Thus, one may easily deduce from Theorem~\ref{thm: invariance principle} that the full cluster $\sfC$ sampled according to $\phi_p[~\cdot~\vert~0 \lra \calH_{\geq n}^{\vec w}]$ or  $\phi_p[~\cdot~\vert~0 \lra n\vec v]$, when properly rescaled, converges in law\footnote{We consider the convergence in law for sets of $\mathbb R^2$ under the Hausdorff distance.}  to the graph of a Brownian motion and bridge, respectively.

\section*{Acknowledgements}
The authors are grateful to S\'ebastien Ott for helpful discussions and insights into the Ornstein--Zernike theory. 
We also thank Gordon Slade for useful comments related to this work.
This work was partially supported by the Swiss National Science Foundation.

\bibliographystyle{amsalpha}
\bibliography{biblioncoz.bib}

\end{document}